\documentclass[10pt]{amsart}
\usepackage{amsmath,amsthm,graphicx,amscd,multirow}
\usepackage{ amssymb} 
\usepackage{ mathrsfs} 
\allowdisplaybreaks[3]

\theoremstyle{definition}
\numberwithin{equation}{section}
\newtheorem{thm}{Theorem}[section]

\newtheorem{prop}[thm]{Proposition}
\newtheorem{cor}[thm]{Corollary}
\newtheorem{lem}[thm]{Lemma}
\newtheorem{conj}[thm]{Conjecture}
\newtheorem{rem}[thm]{Remark}

\title{Explicit inner product formulas and Bessel period formulas  for HST lifts}

\author[K. Namikawa]{Kenichi Namikawa}
\date{\today}
\address{ Faculty of Mathematics, Kyushu University, 744 Motooka, Nishi-Ku, Fukuoka, 819-0395, Japan
}
\email{namikawa@math.kyushu-u.ac.jp}
\date{}
\begin{document}
\begin{abstract}
We explicitly give an inner product formula and a Bessel period formula for theta series on ${\rm GSp}_4$, which was studied by Harris, Soudry and Taylor.  
As a consequence, we prove a non-vanishing of the theta series of small weights and we give a criterion for the non-vanishing of the theta series modulo a prime.
\end{abstract}

\maketitle
\tableofcontents

\section{Introduction}

In \cite{hst93}, Harris-Soudry-Taylor studied a non-vanishing of certain theta series on ${\rm GSp}_4({\mathbf A})$ 
  which coming from cohomological cuspidal automorphic representations $\pi$ on ${\rm GL}_2(F_{\mathbf A})$, where $F$ is an imaginary quadratic field, with some additional datum.
By using these theta series which we call HST lifts, Taylor constructed the Galois representation associated with $\pi$ in the subsequent paper \cite{tay94}. 
The aim of this paper is to give further arithmetic properties of HST lifts. 

As we can find in \cite{ri76}, \cite{mw84} and \cite{su14},  
constructions of congruences between automorphic forms are one of important tools in Iwasawa theory.  
In \cite{ht94}, Hida-Tilouine proved the anti-cyclotomic main conjecture for CM fields $F$ by studying 
congruences between cusp forms and theta series on ${\rm GL}_2$ over the maximal totally real subfield of $F$. 
One of methods for the authors's proof  
    is to combine congruences with certain $L$-values by studying Petersson inner products of algebraically normalized theta series. 
In the present paper, we consider HST lifts on ${\rm GSp}_4({\mathbf A})$ to construct a higher rank analogue of \cite{ht94}.  
Namely, we give  a Petersson inner product formula and a Bessel periods formula for HST lifts which are given in an explicit manner. 

The idea to construct congruences between cusp forms and theta series on ${\rm GSp}_4({\mathbf A})$ in Iwasawa theory  
can be found in some literature.  
For instance \cite{ak13}, \cite{bdsp12}, \cite{jia}, \cite{hn17} and \cite{hn}, 
the authors consider congruences between Siegel cusp forms and Yoshida lifts (\cite{yo80}, \cite{yo84})  
to study the Bloch-Kato conjecture for Asai $L$-functions for ${\rm GL}_2$ over ${\mathbf Q}^{\oplus2}$ or real quadratic fields
under certain strong conditions.     
As an analogue of these studies, \cite{ber} considers HST lifts case assuming an existence of congruences between 
Galois representations attached to Siegel cusp forms and HST lifts.
However, so far in the literature, no explicit construction of  congruences for HST lifts is given and explicit non-vanishing HST lifts are not yet discussed. 
In this paper, we will give an explicit construction of HST lifts and we show some basic properties.

To state the precise statements of the main theorem of this paper, 
we introduce some notation. 
Let ${\mathfrak N}$ be an ideal of the ring of integers of an imaginary quadratic field $F$ with the absolute discriminant $\Delta_F$.
Denote by $U_0({\mathfrak N})$ a congruence subgroup of ${\rm GL}_2(F_{\mathbf A})$, which is introduced in  Section \ref{sec:autoGL2}.
Consider an irreducible cohomological cuspidal automorphic representation $\pi$ of ${\rm GL}_2(F_{\mathbf A})$ of level $U_0(\mathfrak N)$ 
with the trivial central character and the Langlands parameter
\begin{align*}
  W_{\mathbf C}={\mathbf C}^\times \to {\rm GL}_2({\mathbf C}) :  z \mapsto  \begin{pmatrix}  (z/\overline{z})^{\frac{n+1}{2}}  &  0 \\ 0 & (\overline{z}/z)^{\frac{n+1}{2}} \end{pmatrix},   
 \quad (n\geq 0, n:\text{even}), 
 \end{align*}
 where $W_{\mathbf C}$ is the Weil group of ${\mathbf C}$. 
 Assume that $\pi$ is not conjugate self-dual. 
Let  $f: {\rm GL}_2(F_{\mathbf A}) \to {\mathcal W}_{2n+2}({\mathbf C}) :={\rm Sym}^{2n+2}({\mathbf C}^{\oplus 2}) $ 
be a normalized cusp form  in $\pi$.
By using an accidental isomorphism between ${\rm GL}_2(F_{\mathbf A})\times_{F^\times_{\mathbf A}}{\mathbf A}^\times$ and ${\rm GSO}_{3,1}({\mathbf A})$,  
     $f$ is extended  to an automorphic form ${\mathbf f}$ on ${\rm GSO}_{3,1}({\mathbf A})$.          
Consider a map $\delta: \Sigma_{\mathbf Q} \to \{\pm 1\}$ satisfying  
\begin{align*}
  \delta(\infty)=-1; \quad \delta(v)=1 \ (v < \infty). 
\end{align*} 
Then the datum $({\mathbf f}, \delta)$ gives rise to an automorphic form $\widetilde{\mathbf f}$ on ${\rm GO}_{3,1}({\mathbf A})$.   
Denote by $\widetilde{\mathbf f}^\dag$ the $p$-stabilization of $\widetilde{\mathbf f}$, which is introduced in Section \ref{secautoGO}.
Let 
\begin{align*}
{\mathcal L}_\lambda({\mathbf C}) = {\rm Sym}^{n}({\mathbf C}^{\oplus 2}) \otimes {\rm det}^{\otimes 2}, \quad (\lambda=(n+2,2)), 
\end{align*}
and we consider ${\mathcal L}_\lambda({\mathbf C})$ as one of realizations of the representation of ${\rm U}_2({\mathbf R})$ of a highest weight $\lambda$.   
By the theta correspondence between ${\rm GO}_{3,1}$ and ${\rm GSp}_4$, 
        each choice of a Bruhat-Schwartz function $\widetilde{\varphi}$  
        provides a holomorphic Siegel cusp form $\theta^\ast_{\widetilde{\mathbf f}^\dag}$  taking values in ${\mathcal L}_\lambda({\mathbf C})$: 
\begin{align*}  
  \theta^\ast_{\widetilde{\mathbf f}^\dag} 
        : {\mathfrak H}_2 \to {\mathcal L}_\lambda({\mathbf C}), 
\end{align*}
        where ${\mathfrak H}_2$ is the Siegel upper half space of genus 2. 
We call $\theta^\ast_{\widetilde{\mathbf f}^\dag}$  HST lift, since basic arithmetic properties of $\theta^\ast_{\widetilde{\mathbf f}^\dag}$ are discussed in \cite{hst93}. 
Let $N_F={\rm l.c.m.}(N, \Delta_F)$
 for $N{\mathbf Z}=\mathfrak N\cap {\mathbf Z}$
 and $\Gamma^{(2)}_0(N_F)$  a congruence subgroup of ${\rm Sp}_4({\mathbf Z})$ which is introduced in Section \ref{sec:autoSp4}.
Then this HST lift $\theta^\ast_{\widetilde{\mathbf f}^\dag}$ has weight $\lambda$ 
and level  $\Gamma^{(2)}_0(N_F)$.    

In this paper, we will fix a distinguished Bruhat-Schwartz function $\widetilde{\varphi}$ in Section \ref{sectheta}, 
and we prove a Petersson inner product formula and a Bessel period formula  for $\theta^\ast_{\widetilde{\mathbf f}^\dag}$. 
In this section, we introduce the main results of this paper assuming  
\begin{align*}
  n=0 
\end{align*}
for the sake of the simplicity. Note that this assumption makes HST lifts $\theta^\ast_{\widetilde{\mathbf f}^\dag}$ scalar-valued functions.

We firstly introduce a result on Petersson inner product formulas. 
Define the Petersson inner product of $\theta^\ast_{\widetilde{\mathbf f}^\dag}$
to be
\begin{align*}
       \langle \theta^\ast_{ \widetilde{\mathbf f}^\dag}, \theta^\ast_{ \widetilde{\mathbf f}^\dag} \rangle_{\mathfrak H_2}  
  =  \int_{\Gamma^{(2)}_0(N_F) \backslash {\mathfrak H}_2}   
          \theta^\ast_{ \widetilde{ \mathbf f}^\dag}(Z)  \overline{\theta^\ast_{ \widetilde{ \mathbf f}^\dag}(Z)} (\det Y)^2 
           \cdot  \frac{ {\rm d}X {\rm d} Y  }{   ({\rm det} Y)^3},  
\end{align*}
where $Z=X+\sqrt{-1}Y\in {\mathfrak H}_2$ 
and ${\rm d}X = \prod_{j\leq l} {\rm d}x_{jl}, {\rm d}Y=\prod_{j\leq l} {\rm d}y_{jl}$ for
$X=(x_{jl})$ and $Y=(y_{jl})$.  
Also denote by $\langle {\mathbf f}, {\mathbf f} \rangle_{\mathscr H}$ a Petersson inner product of ${\mathbf f}$, 
   which is defined in Section \ref{sec:clainnprd}.   
Let $\varepsilon_v$ be the local root number of $\pi_v$ for a place $v$ of $F$. 
For each rational prime $p$, define
\begin{align*}
    \varepsilon_p  = \begin{cases} \varepsilon_v,   &  (p=v:\text{non-split in } F), \\ 
                                                \varepsilon_{v_1}\varepsilon_{v_2},    &   (p=v_1v_2:\text{split in } F).   \end{cases}      
\end{align*}
Let $L(s, {\rm As}^+(\pi))$ be the Asai $L$-function attached to $\pi$.   
Then the explicit inner product formula for $\theta^\ast_{\widetilde{\mathbf f}^\dag}$ in this paper is given as follows:
\ \\

\noindent
{\bf Theorem A}
(Theorem \ref{thm:alginnprd}, Corollary \ref{cor:nonvan}, $n=0$ case)
{\itshape Assume that
\begin{itemize}
  \item $\pi$ is not conjugate self-dual: $\pi^\vee\not\cong \pi^c$; 
  \item  $\Delta_F$  and ${\rm Nr}_{F/{\mathbf Q}}({\mathfrak N})$ are coprime;
  \item ${\mathfrak N}$ is square-free; 
\end{itemize}
Then we have 
\begin{align*}
    \frac{  \langle  \theta^\ast_{\widetilde{\mathbf f}^\dag}, \theta^\ast_{\widetilde{\mathbf f}^\dag}   \rangle_{\mathfrak H_2}  }{\langle {\mathbf f}, {\mathbf f} \rangle_{\mathscr H} }
 =  2^\beta N_F \Delta^{-3}_F  
      \times   L(1, {\rm As}^+(\pi)) 
                  \prod_{p\mid N } (1+\varepsilon_p) 
                  \cdot \prod_{p\mid \Delta_F}  (1+p^{-1})
\end{align*}
where $\beta\in {\mathbf Z}$ is an explicit integer, which is given in Theorem \ref{thm:alginnprd}. 
Furthermore, if we assume that
\begin{itemize}
\item[(LR)]  for each rational prime number $p\mid N$, $\varepsilon_p=1$,    
\end{itemize}
then, $\theta^\ast_{\widetilde{\mathbf f}^\dag}$ is non-zero. 
}
\ \\

\begin{rem}
\begin{enumerate}
\item 
Due to the works of Hida (\cite{hi81a}, \cite{hi81b}) in ${\rm GL}_2$ case,    
it is well-known that
an explicit inner product formula is 
one of important tools for constructions  of congruences between cuspidal automorphic forms.   
In ${\rm GSp}_4$ case, 
Agarwal and Klosin (\cite{ak13}) construct  
a congruence between Siegel cusp forms 
    and Yoshida lifts under their conjectural explicit inner product formula for Yoshida lifts \cite[Conjecture 5.19]{ak13}.    
In \cite[Theorem 5.7]{hn},  
a generalization of \cite[Conjecture 5.19]{ak13}  is proved. 
Theorem A is an analogue of \cite[Theorem 5.7]{hn} in the HST lifts case.
\item 
To deduce Theorem A, we follow the strategy for the proof of explicit inner product formulas for Yoshida lifts in \cite{hn}, 
that is, 
we compute explicitly certain local integrals, which come from the Rallis inner product formula in \cite{gqt}.    
In our proof, 
we need a conjectural formula  (Conjecture \ref{conj:arint}) on a certain archimedean local integral for general $n$.    
We can verify this conjectural formula for $n=0,2,4,6,8$ with some helps of Mathematica,     
however we need some special formulas on Bessel functions of the second kind and the hypergeometric functions in our way to check the conjecture
    even for these small weight cases.     
This is the reason why we include Appendix in this paper, where we explain how to check Conjecture \ref{conj:arint} for $n=0$.  
\end{enumerate}
\end{rem}

The explicit inner product formula in Theorem A does not depend on a normalization of an initial automorphic form $f$. 
However it depends on a normalization of the distinguished Bruhat-Schwartz function $\widetilde{\varphi}$.   
Hence it is necessary to discuss that $\widetilde{\varphi}$ is suitably normalized. 
In this paper, we show that  
   certain types of Fourier coefficients of $\theta^\ast_{\widetilde{\mathbf f}^\dag}$ are $p$-adically integral after dividing a period of $\pi$  
    and we study a relation between the non-vanishing modulo a prime of Fourier coefficients 
        and the central value $L(\frac{1}{2}, \pi\otimes\phi)$  of the standard $L$-function of $\pi$ 
        with twists of finite-order anti-cyclotomic characters $\phi: F^\times{\mathbf A}^\times \backslash F^\times_{\mathbf A} \to {\mathbf C}^\times$. 
This provides a reason of our normalization of $\widetilde{\varphi}$.  

We introduce our second result on Fourier coefficients of HST lifts $\theta^\ast_{\widetilde{\mathbf f}^\dag}$, 
which is a corollary of an explicit Bessel periods formula.   
Let 
\begin{align*}
 \Lambda_2 = \left\{ S= \begin{pmatrix} a & \frac{b}{2}  \\ \frac{b}{2} & c \end{pmatrix} 
                                 | \ a, b, c \in {\mathbf Z}, S \text{ is semi-positive definite}  \right\}.  
\end{align*}
Write the Fourier expansion of $\theta^\ast_{\widetilde{\mathbf f}^\dag}$ as 
\begin{align*}
   \theta^\ast_{ \widetilde{\mathbf f}^\dag }(Z)
    = \sum_{S\in \Lambda_2} {\mathbf a}(S) q^S, 
       \quad \left(q^S= \exp(2\pi\sqrt{-1} {\rm Tr}(SZ)), {\mathbf a}(S) \in {\mathbf C}\right).   
\end{align*}

Fix an odd prime number $p$, an embedding $\overline{\mathbf Q} \to {\mathbf C}$ and an isomorphism ${\mathbf C} \cong {\mathbf C}_p$ as fields. 
Denote by ${\mathcal O}_{{\mathbf C}_p}$ the ring of integers  of ${\mathbf C}_p$. 
Then we have a period $\Omega_{\pi, p} \in {\mathbf C}^\times_p$ 
    which is determined by $\pi$ and $p$ up to a multiplication of elements in  ${\mathcal O}^\times_{{\mathbf C}_p}$.
We briefly recall the definition of $\Omega_{\pi, p}$ in Section \ref{sec:claBess}. 
Since the idea of the definition of periods $\Omega_{\pi, p}$ is introduced in \cite[Section 8]{hi94cr},    
we call $\Omega_{\pi, p}$ Hida's canonical period. 
The definition of  $\Omega_{\pi, p}$ immediately  shows that 
      $L(\frac{1}{2}, \pi\otimes\phi) / \Omega_{\pi, p}$ is an element in ${\mathcal O}_{{\mathbf C}_p}$ (\cite[Theorem 8.1]{hi94cr}).

The second main result in this paper 
 is summarized as follows: 
\ \\ 

\noindent
{\bf Theorem B} (Corollary \ref{cor:Bess}, $n=0$ case)
{\itshape 
Assume that
\begin{itemize}
  \item $\pi$ is not conjugate self-dual: $\pi^\vee\not\cong \pi^c$; 
  \item  $\Delta_F$  and ${\rm Nr}_{F/{\mathbf Q}}({\mathfrak N})$ are coprime;
  \item $p\nmid 2 \cdot \Delta_F \cdot \sharp \left( F^\times {\mathbf A}^\times \backslash F^\times_{\mathbf A}/ {\mathbf C}^1 \widehat{\mathcal O}^\times_F\right)$, 
             where ${\mathbf C}^1=\left\{  z \in {\mathbf C} \ | \   z \overline{z} =1  \right\}$. 
\end{itemize}
Then we have the following two statements: 
\begin{enumerate}
\item {\rm (Corollary \ref{cor:Bess} \ref{cor:Bess(ii)}) }  
          Let $C$ be an integer satisfying the following condition:  
          \begin{itemize}
          \item[(CF)] $C$ is prime to $N \Delta_F$ and each prime $v$ of ${\mathbf Q}$ dividing $C$ is split in $F$.
          \end{itemize} 
          Then for each $S\in \Lambda_2$ with $\det S = \frac{\Delta_F C^2}{4}$, we have ${\mathbf a}(S)/\Omega_{\pi, p} \in {\mathcal O}_{{\mathbf C}_p}$.    
\item {\rm (Corollary \ref{cor:Bess} \ref{cor:Bess(iii)}) } 
          Assume that the condition {\rm (LR)} in Theorem A. 
         Then there exists a positive integer $C \in {\mathbf Z}$ satisfying the condition {\rm (CF)} in the first statement and 
         \begin{align*}
              p\nmid C \cdot \prod_{ \substack{ \ell:\text{prime} \\  \ell \mid C} }(\ell - 1) 
         \end{align*}
         such that 
         the following assertions are equivalent:          
         \begin{itemize}
         \item For a finite-order Hecke character 
                 $\phi: F^\times {\mathbf A}^\times \backslash F^\times_{\mathbf A} \to {\mathbf C}^\times$ of the conductor $C {\mathcal O}_F$,   
                  we have 
                 $  L(\frac{1}{2}, \pi\otimes\phi) / \Omega_{\pi, p} \in {\mathcal O}^\times_{\mathbf C_p}$.
         \item For an element of  $S\in \Lambda_2$ with $\det S= \frac{\Delta_F C^2}{4}$, we have 
                 $ {\mathbf a}(S) / \Omega_{\pi, p}  \in {\mathcal O}^\times_{\mathbf C_p}$.
         \end{itemize}
\end{enumerate}
}

\begin{rem}
\begin{enumerate}
\item In \cite[Proposition 5.1]{hn17}, we prove that all Fourier coefficients of Yoshida lifts are $p$-adically integral.  
         Hence Theorem B (i) can be considered as a weak analogue of \cite[Proposition 5.1]{hn17}. 
\item Theorem B (ii) reduces a study of non-triviality of HST lifts modulo a prime  
            to a study of the central value of the standard $L$-function of $\pi$.
         Even though results on the non-vanishing of $L(\frac{1}{2}, \pi\otimes \phi)/\Omega_{\pi, p}$ modulo a prime  for anti-cyclotomic characters $\phi$  
          are not known so far,    
         this strategy for a proof of the non-triviality of Fourier coefficients modulo a prime can be find in \cite[Section 13]{su14} for instance. 
         If we consider the Yoshida lifts, which are in the image of theta correspondences from a pair of elliptic modular forms, 
         \cite[Theorem 5.3]{hn17} shows that the Yoshida lifts have Fourier coefficients which are non-vanishing modulo a prime according to this strategy.  
         Hence Theorem B (ii) is expected to be useful to study HST lifts which are non-vanishing modulo a prime. 
\item The proof of Theorem B in this paper is due to an explicit computation of Bessel periods of HST lifts.   
         For each imaginary quadratic field $K$, 
         the Bessel periods of Siegel cusp forms  are torus integrals which are defined by  a map ${\rm Res}_{K/{\mathbf Q}} {\mathbb G}_{m/ K} \to {\rm GSp}_4$.    
        The Waldspurger formula (\cite[Proposition 7]{wa85}  
        ) shows that the square of the Bessel periods can be described by 
              the central value of the Rankin-Selberg type $L$-function of $\pi$, 
              up to local integrals which depends on the choice of Bruhat-Schwartz functions. 
              (See \cite[Section 7]{co17} for this argument.)         
        However, if we consider a special case $K=F$, 
        then an explicit Bessel period formula (Theorem \ref{th:bessel}) in this paper implies that the Bessel periods itself can be 
              described by the standard $L$-function of $\pi$.   
        Although an explicit Bessel periods formula for general imaginary quadratic fields $K$   
                           is desired to obtain more information of Fourier coefficients of HST lifts, 
            we concentrate on the case that $K=F$ in this paper,         
            since it already gives us information on a normalization of $\widetilde{\varphi}$.     
        The integrality of Fourier coefficients of more general type  seems to be important,  
           however we remain it to the future works.                    
\item In \cite[Corollary 28]{be14}, the author constructs HST lifts which have integral Fourier coefficients in the $n=0$ case.  
         Since Berger's method relies on a cohomological interpretation of theta lifts (\cite[Section 3]{be14}) which is different form our formulation in Section \ref{s:hst}, 
          it is necessary to check whether his construction coincides with ours.
          So far, basic properties of HST lifts in \cite{be14} such as inner product formulas and a sufficient condition for non-triviality modulo a prime are not yet worked out. 
          Hence it seems to be one of directions for further study of HST lifts 
           to clarify a relation between results in \cite{be14} and the present paper, 
             and to generalize a result in \cite{be14} in the $n>0$ case by using our fixed Bruhat-Schwartz function $\widetilde{\varphi}$.
\end{enumerate}
\end{rem}

The organization of this paper is as follows.    
We introduce basic notation of automorphic forms on ${\rm GL}_2$ over imaginary quadratic fields and ${\rm GSp}_4$ in  Section \ref{s:notdef}. 
Automorphic forms on orthogonal groups are discussed in Section \ref{sec:orthauto}, 
    where we explain how to extend automorphic forms on ${\rm GL}_2$ over imaginary quadratic fields
              to automorphic forms on ${\rm GO}_{3,1}$ according to \cite{hst93}. 
HST lifts on ${\rm GSp}_4$ are defined in Section \ref{s:hst},
    where we fix a distinguished Bruhat-Schwartz function $\widetilde{\varphi}$.   
In Section \ref{s:InnPrd}, 
we reduce a computation for an explicit inner product formula for HST lifts 
to computations of certain local integrals by using the Rallis inner product formula.    
   The computation of the local integrals is given in Section \ref{sec:complocint}, which is summarized in Theorem \ref{HSTInnprd}. 
The archimedean local integral is computed assuming Conjecture \ref{conj:arint}  which is hold for small weights. 
In Appendix, we explain how to check Conjecture \ref{conj:arint}.
The first main result Theorem A (Theorem \ref{thm:alginnprd}, Corollary \ref{cor:nonvan}) is given in Section \ref{sec:clainnprd}.  
Section \ref{sec:Bessel}  and Section \ref{s:besloc} are devoted to a study of Fourier coefficients of HST lifts.     
The study of Fourier coefficients in this paper is based on an explicit computation of  Bessel periods for HST lifts.  
The definition of Bessel periods is given in Section \ref{s:defbessel}, 
     and an explicit Bessel periods formula is given in Theorem \ref{th:bessel}.      
The proof of Theorem \ref{th:bessel} is given by computations of local integrals in Section \ref{s:besloc}.   
An explicit relation between Bessel periods and Fourier coefficients is discussed in Section \ref{sec:claBess}, 
where we prove the second main result Theorem B (Corollary \ref{cor:Bess}).

\section{Notation and definitions}\label{s:notdef}

\subsection{Basic notation}
Denote by ${\mathbf Q}$ (resp. ${\mathbf Z}, {\mathbf R}, {\mathbf C}$) be the rational number field 
(resp. the ring of the rational integers, the real number field, the complex number field).
The binomial coefficient $\binom{a}{b}$ for $a,b\in {\mathbf Q}$ is defined by
\begin{align*}
   \binom{a}{b} = \begin{cases} \frac{\Gamma(a+1)}{\Gamma(b+1)\Gamma(a-b+1)},  &  (a, b \in {\mathbf Z}),   \\
                                                        0, &  \text{(otherwise)}. \end{cases}
\end{align*}

Let $\Sigma_{\mathbf Q}$ be the set of places of the rational number field ${\mathbf Q}$. 
For each place $v\in \Sigma_{\mathbf Q}$, we denote by ${\mathbf Q}_v$ the completion of ${\mathbf Q}$ at $v$.
Write ${\mathbf A}$ for the ring of adeles of ${\mathbf Q}$ and 
     denote by ${\mathbf A}_{\rm fin}$ (resp. ${\mathbf A}_\infty$) the finite (resp. infinite) part of ${\mathbf A}$.
For an algebraic group $G$ over ${\mathbf Q}$ and a ${\mathbf Q}$-algebra $A$,  
   we denote by $G(A)$ (resp. $Z_G$) the $A$-rational points of $G$ (resp. the center of $G$).
For $x \in G({\mathbf A})$, we write $x_{\rm fin}$ for the projection of $x$ to $G({\mathbf A}_{\rm fin})$.  
Let $[G]$ be the quotient space $G({\mathbf Q})\backslash G({\mathbf A})$.
Let $\psi_{\mathbf Q}=\prod_{v\in \Sigma_{\mathbf Q}}\psi_{\mathbf Q, v}:{\mathbf A}/{\mathbf Q}\to {\mathbf C}^\times$ be the additive character 
with $\psi_{\mathbf Q}(x_\infty)=\exp(2\pi\sqrt{-1}x_\infty)$ for $x_\infty\in {\mathbf R}={\mathbf Q}_\infty$.

For a number field $L$, we denote by ${\mathcal O}_L$ 
(resp.  $L_{\mathbf A}$, ${\mathfrak d}_L$, $\Delta_L$, $\delta_L$) the ring of integers of $L$ 
(resp. the ring of adeles of $L$, the different of $L/{\mathbf Q}$, the discriminant of $L/{\mathbf Q}$, 
a generator of the different ${\mathfrak d}_L$ in $L_{{\mathbf A},{\rm fin}}$).
For each place $v$ of $L$, define $L_v$ to be the completion of $L$ at $v$.  
Put $\widehat{\mathbf Z}=\prod_{v\Sigma_{\mathbf Q}, v<\infty} {\mathbf Z}_v$   
    and $\widehat{\mathcal O}_L={\mathcal O}_L\otimes_{\mathbf Z} \widehat{\mathbf Z}$.  
For each finite place $v$, denote by ${\mathcal O}_{L,v}$ (resp. $\varpi_v$) the ring of integers  (resp. a uniformizer) of $L_v$. 
Put $q_v=\sharp {\mathcal O}_{L,v}/\varpi_v {\mathcal O}_{L,v}$.   
Let ${\rm Tr}_{L/{\mathbf Q}}$ (resp. ${\rm Nr}_{L/{\mathbf Q}}$) be the trace (resp. norm) map of $L/{\mathbf Q}$.  
Define an additive character $\psi_L: L_{\mathbf A}/L \to {\mathbf C}^\times$ to be $\psi_L=\psi_{\mathbf Q}\circ {\rm Tr}_{L/{\mathbf Q}}$.

In this paper, $L$-functions are always complete one. 
For instance, we define the Riemann zeta function $\zeta$ as follows:
\begin{align*}
   \zeta_\infty(s) = \Gamma_{\mathbf R}(s) = \pi^{-\frac{\pi}{2}} \Gamma\left(\frac{s}{2}\right), 
   \quad 
   \zeta_p=\frac{1}{1-p^{-s}}, 
   \quad 
   \zeta(s) = \prod_{v\in \Sigma_{\mathbf Q}} \zeta_v(s). 
\end{align*}

Let $L$ be a number field and $\varpi_v$ a uniformizer of $L_v$ for each finite place $v$ of $L$. 
For each Hecke character $\phi:L^\times\backslash L^\times_{\mathbf A} \to {\mathbf C}^\times$,   
we write $\phi=\prod_v\phi_v$, where $v$ runs over the set of places of $L$.   
Let $\varpi^{c_v}_v{\mathcal O}_{L,v}$ be the conductor of $\phi_v$.
Define the $\epsilon$-factor $\epsilon(s, \phi_v)$ of $\phi_v$ 
with respect to $\psi_{L,v}$ to be 
\begin{align}
   \epsilon (s, \phi_v) = \int_{\varpi^{-c_v-{\rm ord}_v({\mathfrak d}_L)}_v{\mathcal O}^\times_{L, v}}   \phi^{-1}_v(x) \psi_{L,v}(x)  |t|^{-s}_v  {\rm d}t  \label{def:ep}
\end{align}
where ${\rm d} t$ is the Haar measure on $L_v$ which is self-dual with respect to $\psi_{L,v}$.

\subsection{Algebraic representation of ${\rm GL}_2$}\label{algrep}
Let $A$ be a ${\mathbf Z}$-algebra.
Let $A[X,Y]_n$ denote the space of two variable homogeneous polynomials of degree $n$ over $A$.
Suppose that $n!$ is invertible in $A$.
We define the perfect pairing $\langle \cdot , \cdot \rangle_n:A[X,Y]_n\times A[X,Y]_n\to A$ by
\begin{align*}
  \langle X^iY^{n-i}, X^jY^{n-j} \rangle_n
 =\begin{cases}  (-1)^i\binom{n}{i}^{-1},  & ( i+j=n), \\
                       0, & ( i+j\neq n). \end{cases} 
\end{align*}
Denote by $u^\vee$ the dual of $u\in A[X,Y]_n$ with respect the pairing  $\langle \cdot, \cdot \rangle_n$. 
For $\lambda = (n+b,b)\in {\mathbf Z}^2$ with $n, b \in {\mathbf Z}_{\geq 0}$,
put ${\mathcal L}_\lambda(A)=A[X,Y]_n$ 
and let $\rho_\lambda:{\rm GL}_2(A) \to {\rm Aut}_A{\mathcal L}_\lambda(A)$ be the representation given by
\begin{align}\label{eq:pairequiv}
  \rho_\lambda(g) P(X,Y) = P((X,Y)g) \cdot (\det g)^b.  
\end{align}
Then $( \rho_\lambda, {\mathcal L}_\lambda(A) )$ is the algebraic representation  of ${\rm GL}_2(A)$
with the highest weight $\lambda$.
Moreover, it is well-known that the pairing $\langle \cdot,\cdot\rangle_{\mathcal L}:=\langle \cdot,\cdot\rangle_n$ on ${\mathcal L}_\lambda(A)$ satisfies
\begin{align*}
     \langle \rho_\lambda(g)v, \rho_\lambda(g)w \rangle_n
  = (\det g)^{n+2b} \cdot \langle v,w \rangle_n, 
   \quad (g\in {\rm GL}_2(A)).
\end{align*}

For each non-negative integer $k$, we put
\begin{align*}
  ( \tau_k, {\mathcal W}_k(A)) := (\rho_{(k,-k)},   A[X,Y]_{k} ). 
\end{align*}
Then $(\tau_k, {\mathcal W}_k(A) )$ is an algebraic representation of ${\rm PGL}_2(A) = {\rm GL}_2(A)/A^\times$.

\subsection{Automorphic forms on ${\rm GL}_2$ over imaginary quadratic fields}\label{sec:autoGL2}

We recall a definition and basic formulas of automorphic forms on ${\rm GL}_2$ over imaginary quadratic fields 
according to \cite{hi94cr}. 
Some of formulas are also found in \cite{na17}.

Let $F$ be an imaginary quadratic filed, $n$ a non-negative integer and ${\mathfrak N}$ an ideal of the ring of integers ${\mathcal O}_F$ of $F$. 
Put 
\begin{align*}
  U_0({\mathfrak N}) 
    = \left\{   \begin{pmatrix} a & b \delta^{-1}_F \\ c\delta_F & d \end{pmatrix}  
                  | a, b, c, d \in \widehat{\mathcal O}_F,  
                     ad-bc \in \widehat{\mathcal O}^\times_F, 
                     c \equiv 0 \ {\rm mod} \ {\mathfrak N}  \right\}. 
\end{align*}

Let $\pi$ be a unitary irreducible cuspidal automorphic representation of ${\rm GL}_2(F_{\mathbf A})$ satisfying the following conditions:
 \begin{itemize}
 \item $\pi$ has a cohomological weight $n+2$, 
               that is, 
               the Langlands parameter  $W_{\mathbf C} \cong {\mathbf C}^\times \to {\rm GL}_2({\mathbf C})$ of $\pi_\infty$ is given by 
           \begin{align*}
             z \mapsto \begin{pmatrix} (z/\overline{z})^{\frac{n+1}{2}}  &  0  \\ 0 &  (\overline{z}/z)^{\frac{n+1}{2}}    \end{pmatrix},   
           \end{align*} 
           where $W_{\mathbf C}$ is the Weil group of ${\mathbf C}$.
 \item There exists an ideal ${\mathfrak M}$ of ${\mathcal O}_F$ such that $\pi$ has a ${U_0({\mathfrak M})}$-fixed vector.
          Denote  by ${\mathfrak N}$ the minimal ideal of ${\mathcal O}_F$ 
              among a set of ideals of ${\mathcal O}_F$ satisfying this property.  
 \end{itemize}
Recall that $\pi$ is realized as 
     a subspace of a regular representation of moderate growth smooth functions $f:{\rm GL}_2(F_{\mathbf A}) \to {\mathbf C}$. 
Hereafter we denote this realization by ${\mathcal A}_0( \pi )$.   

We put 
\begin{align*}
{\mathcal S}_{n+2}(U_0({\mathfrak N})) = \bigoplus_\pi {\rm Hom}_{{\rm SU}_2({\mathbf R})}({\mathcal W}_{2n+2}({\mathbf C})^\vee,  {\mathcal A}_0(\pi)^{U_0({\mathfrak N})}).
\end{align*}
We consider  a moderate growth smooth function $f:{\rm GL}_2(F_{\mathbf A}) \to {\mathcal W}_{2n+2}({\mathbf C})$ 
as an element in ${\mathcal S}_{n+2}(U_0({\mathfrak N}))$, 
if $f$ satisfies the following condition:
\begin{itemize}
\item Let $D, D_c$ be the Casimir elements of the universal enveloping algebra of ${\rm Lie}({\rm GL}_2({\mathbf C}))\otimes_{\mathbf R}{\mathbf C}$,  
          which is defined in \cite[Section 2.3]{hi94cr}. 
          Then, we have $Df=D_cf=(\frac{n^2}{2}+n)f$.  
\item For each $\gamma \in {\rm GL}_2(F), g\in {\rm GL}_2(F_{\mathbf A})$ and $u_\infty \in {\rm SU}_2({\mathbf R}), u \in U_0({\mathfrak N})$, we have
         \begin{align*}
            f(\gamma g u_\infty u) = \rho_{(2n+2, 0)}(u^{-1}_\infty) f(g).
         \end{align*}
\item For each $g\in {\rm GL}_2(F_{\mathbf A})$ and each unipotent radical  $U(F_{\mathbf A})$ of any proper parabolic subgroup of ${\rm GL}_2(F_{\mathbf A})$, we have 
         \begin{align*}
            \int_{ U(F) \backslash U(F_{\mathbf A}) }   f(u g )    {\rm d}x = 0,
         \end{align*}
         where ${\rm d}x$ is a Haar measure on $U(F_{\mathbf A})$.
\end{itemize}
We call an element of ${\mathcal A}_0(\pi)$ (resp. ${\mathcal S}_{n+2}(U_0({\mathfrak N}))$) 
a cuspidal automorphic (resp. cusp) form on ${\rm GL}_2(F_{\mathbf A})$ of the cohomological weight $n+2$, of level $U_0({\mathfrak N})$.

In this paper, we always assume the following three conditions on $\pi$:
\begin{itemize}
\item $n$ is even; 
\item The central character of $\pi$ is trivial;
\item $\pi^\vee\not\cong\pi^c$, that is, $\pi$ is not conjugate self-dual.
\end{itemize}

We recall some basic properties on the Whittaker model of $\pi$. 
For a cusp form $f\in {\mathcal S}_{n+2}(U_0({\mathfrak N}))$, denote by $\pi$ the automorphic representation which is generated by $f$.
Let ${\mathcal W}(\pi, \psi_F)\cong \otimes^\prime_w {\mathcal W}(\pi_w, \psi_{F,w})$ be the Whittaker model of $\pi$. 
We note that the Whittaker function $W_f: {\rm GL}_2(F_{\mathbf A}) \to {\mathcal W}_{2n+2}({\mathbf C})$ for $f$ is defined to be 
\begin{align*}
    W_f(g) = \int_{[F]} f \left(\begin{pmatrix} 1 & x \\  0 & 1 \end{pmatrix} g \right) \psi_F(-x) {\rm d}x,   
\end{align*}
where ${\rm d}x$ is the self-dual Haar measure with respect to $\psi_F$. 

We denote by $K_s(z)$ the Bessel function of the second kind which is defined to be 
\begin{align}\label{def:2bessel}
   K_s(z) = \frac{1}{2} \int^\infty_0 \exp \left(-\frac{z}{2}  \left( t+\frac{1}{t} \right) \right)  t^{s-1} {\rm d} t, \quad ({\rm Re}(z) > 0, s\in {\mathbf C}).  
\end{align}
Define the element $W_{f,\infty}: {\rm GL}_2(F_{\mathbf A}) \to {\mathcal W}_{2n+2}({\mathbf C})$ in ${\mathcal W}(\pi_\infty, \psi_{F, \infty})$ so that  
\begin{align*}
    W_{f,\infty} \left(\begin{pmatrix} t & 0 \\ 0 & 1  \end{pmatrix} \right) 
         = 2^4\times \sum^{n+1}_{j=-n-1} \sqrt{-1}^j   t^{n+2} K_j(4\pi t)  (X^{n+1+j}Y^{n+1-j})^\vee
\end{align*}
for $0<t\in {\mathbf R}$.
Note that if we fix a Haar ${\rm d}^\times t_\infty$ on $F^\times_\infty={\mathbf C}^\times$ so that 
\begin{align*}
    {\rm d}^\times t_\infty = \frac{ {\rm d} r {\rm d}\theta  }{2\pi r},  
      \quad (t_\infty = r e^{\sqrt{-1}\theta}), 
\end{align*}
then we have 
\begin{align*}
  \int_{\mathbf C^\times} W_{f,\infty} \left(\begin{pmatrix} t_\infty & 0 \\ 0 & 1  \end{pmatrix} \right) |t_\infty|^{s-\frac{1}{2}}_\infty  {\rm d}^\times t_\infty  
  =&  \Gamma_{\mathbf C}\left( s + \frac{n+1}{2} \right)^2 (X^{n+1}Y^{n+1})^\vee \\
  =&  L(s, \pi_\infty) (X^{n+1}Y^{n+1})^\vee.
\end{align*}
For each finite place $w$ of $F$, we also fix a Haar measure ${\rm d}^\times t_w$  on $F^\times_w$ so that 
${\rm vol}({\mathcal O}^\times_{F, w}, {\rm d}^\times t_w) = q^{-\frac{1}{2} {\rm ord}_w(\delta_{F,w})}_w$.   
We fix the element $W_{f,w}$ in ${\mathcal  W}(\pi_v, \psi_{F,w})$ so that  
\begin{align*}
  \int_{F^\times_w}  W_{f,w} \left(\begin{pmatrix} t_w & 0 \\ 0 & 1 \end{pmatrix} \right)  |t_w|^{s-\frac{1}{2}}_w {\rm d}^\times t_w 
  = q^{(s-\frac{1}{2}) {\rm ord}_w(\delta_F) }_w L(s,\pi_w).
\end{align*} 
The multiplicity one theorem on the Whittaker model of $\pi_w$ shows that 
we can normalize $f$ so that the following identity holds:
\begin{align*}
  W_f = \prod_w W_{f,w}. 
\end{align*}

\subsection{Siegel modular forms on ${\rm GSp}_4({\mathbf A})$  }\label{sec:autoSp4}

Let ${\rm GSp}_4$ be the algebraic group defined by
\begin{align*}
  G:=  {\rm GSp}_4 
 = \left\{  g\in {\rm GL}_4 \ | \  
              g\begin{pmatrix} 0_2 & 1_2 \\ -1_2 & 0_2 \end{pmatrix} {}^{\rm t} g 
              = \nu(g) \begin{pmatrix} 0_2 & 1_2 \\ -1_2 & 0_2 \end{pmatrix}  
                 \right\}  
\end{align*}
with the similitude character $\nu: {\rm GSp}_4\to {\mathbb G}_m$.
Let ${\mathfrak H}_2$ be the Siegel upper half place of degree 2. 
Define the automorphy factor $J:{\rm GSp}_4({\mathbf R})^+\times{\mathfrak H}_2 \to {\rm GL}_2({\mathbf C})$ by
\begin{align*}
  J(g,Z) = CZ+D, \quad \left(g\in \begin{pmatrix} A & B \\ C & D \end{pmatrix} \right). 
\end{align*}
Put ${\mathbf i}:=\sqrt{-1} \cdot 1_2\in {\mathfrak H}_2$. 
Let $K_\infty$ be the maximal compact subgroup of ${\rm GSp}_4({\mathbf R})$ which is defined by
\begin{align*}
  K_\infty = \left\{ g\in {\rm GSp}_4({\mathbf R}) \ |\  g{}^{\rm t}g = 1_2 \right\}.  
\end{align*}
Define $U^{(2)}_0(N)$ to be a compact subgroup of ${\rm GSp}_4(\widehat{\mathbf Z})$ ($1\leq N\in {\mathbf Z}$) which is defined by 
\begin{align*}
    U^{(2)}_0(N)
    = \left\{  \begin{pmatrix} A & B \\ C & D \end{pmatrix}  \in {\rm GSp}_4(\widehat{\mathbf Z}) \ | \ 
               C \equiv 0_2 \ {\rm mod}\ N \widehat{\mathbf Z} 
                 \right\}.
\end{align*}
Write $\Gamma^{(2)}_0(N) = U^{(2)}_0(N) \cap {\rm Sp}_4({\mathbf Z})$. 
Let  $\chi:  {\mathbf Q}^\times \backslash {\mathbf A}^\times \to {\mathbf C}^\times$  be a Hecke character of ${\mathbf A}^\times$
      and  $\lambda \in {\mathbf Z}^{\oplus 2}$.
Then a holomorphic Siegel cusp form ${\mathcal F}:{\rm GSp}_4({\mathbf A})\to {\mathcal L}_\lambda({\mathbf C})$  of genus 2 is said to be 
weight $\lambda$, level $U^{(2)}_0(N)$ and type $\chi$ with trivial central character   
if ${\mathcal F}$ satisfies 
\begin{align*}
  & {\mathcal F}( \gamma g k_\infty u_{\rm fin} z )
 = \rho_\lambda(J(k_\infty,{\mathbf i})^{-1} ) {\mathcal F}(g) \chi( \det D ),    \\
   & \quad \left(  \gamma \in {\rm GSp}_4({\mathbf Q}),  
              g \in   {\rm GSp}_4({\mathbf A}),
              k_\infty \in K_\infty, 
              u_{\rm fin} = \begin{pmatrix} A & B  \\ C & D \end{pmatrix} \in U^{(2)}_0(N), 
              z \in {\mathbf A}^\times \right).
\end{align*}

Let $U$ be a unipotent subgroup of $G$ defined by 
\begin{align*}
 U = \left\{ u(X) := \begin{pmatrix} 1_2 & X \\ 0_2 & 1_2  \end{pmatrix} 
                 | \ X={}^{\rm t}X   \right\}.
\end{align*}
Let ${\mathcal S}_2({\mathbf Q})$ be the set of symmetric matrices in ${\rm M}_2({\mathbf Q})$.
For each $S\in {\mathcal S}_2({\mathbf Q})$, 
let $\psi_S: U_{\mathbf Q}\backslash U_{\mathbf A} \to {\mathbf C}^\times$ be the additive character defined by
$\psi_S(u(X))=\psi_{\mathbf Q} ({\rm Tr}(-SX))$.
Define the adelic $S$-th Fourier coefficient 
    ${\mathbf W}_{{\mathcal F},S}:{\rm GSp}_4({\mathbf A}) \to {\mathcal L}_\lambda({\mathbf C})$ of ${\mathcal F}$ to be
\begin{align}
  {\mathbf W}_{{\mathcal F},S} (g) = \int_{[U]} {\mathcal F}(ug)\psi_S(u) {\rm d} u,   \label{def:FWcoeff}
\end{align}
where ${\rm d}u$ is the Haar measure such that  ${\rm vol}( U({\mathbf Q})\backslash U({\mathbf A}), {\rm d}u )=1$.
Then, we have the Fourier expansion
\begin{align}\label{def:Fexp}
 {\mathcal F}(g) = \sum_{S\in {\mathcal S}_2({\mathbf Q}) } {\mathbf W}_{{\mathcal F},S}(g).  
\end{align}
Recall the following well-known formula for later use:
\begin{align} \label{FCtw}
     {\mathbf W}_{{\mathcal F},S} \left(  \begin{pmatrix} \gamma & 0_2 \\ 0_2 & \nu {}^{\rm t}\gamma^{-1}   \end{pmatrix}  g \right) 
    = {\mathbf W}_{{\mathcal F}, \nu^{-1} {}^{\rm t}\gamma S  \gamma } (g),    
\end{align}
where $\gamma \in {\rm GL}_2({\mathbf Q})$ and $\nu\in {\mathbf Q}^\times$

\section{Automorphic forms on orthogonal groups }\label{sec:orthauto}

In this section, we describe automorphic forms on ${\rm GO}_{3,1}({\mathbf A})$ in terms of 
automorphic forms on ${\rm GL}_2(F_{\mathbf A})$.    
We firstly fix a realization of orthogonal space of signature $(3,1)$ and ${\rm GO}_{3,1}$ in Section \ref{secorth}. 
In Section \ref{secrepGO}, we recall a relation between automorphic representations ${\rm GSO}_{3,1}$ and ${\rm GO}_{3,1}$.   
This relation will clarify the representation theoretic aspects of automorphic forms on ${\rm GO}_{3,1}({\mathbf A})$
   which will be provided in Section \ref{secautoGO} according to a description in \cite{hst93}.  
The minimal $K$-type of automorphic representations of ${\rm GO}_{3,1}({\mathbf A})$ is described in Section \ref{sec:minK}, 
which will be used to study the minimal $K$-type of HST lifts in Lemma \ref{HSTwt}.

\subsection{Orthogonal groups}\label{secorth}

Let $F={\mathbf Q}(\sqrt{-\Delta_F})$ be an imaginary quadratic field with the absolute discriminant $\Delta_F >0$. 
Let $c$ be the generator of  ${\rm Gal}(F/{\mathbf Q}) $.
We sometimes write  $x^c=\bar{x}$ for $x \in F$ and extend it to ${\rm M}_2(F_{\mathbf A})$. 
Let $\ast:{\rm M}_2(F_{\mathbf A}) \to {\rm M}_2(F_{\mathbf A})$ be the involution defined by 
\begin{align*}
    \begin{pmatrix} a & b \\ c & d \end{pmatrix}^\ast
    =  \begin{pmatrix} d & -b \\ -c & a \end{pmatrix}.
\end{align*}
Define the quadratic space $(V, {\rm n})$ over ${\mathbf Q}$  to be 
\begin{align*}
  V = \{ x \in {\rm M}_2(F) |\ \bar{x}^\ast = x  \} 
\end{align*}
and ${\rm n}:V \to {\mathbf Q}, x\mapsto {\rm n}(x):= {\rm det}(x)$.
Define $\varrho:  {\rm Res}_{F/{\mathbf Q}} {\rm GL}_{2 /F}\times_{ {\rm Res}_{F/{\mathbf Q}}  {\mathbb G}_{m/{\mathbf Q}} }{\mathbb G}_{m/{\mathbf Q}} 
\stackrel{\sim}{\to} H^0:={\rm GSO}(V)$ to be the isomorphism given by
\begin{align*}
  (h,\alpha) \mapsto \varrho(h,\alpha)z = \alpha^{-1} h z \overline{h}^\ast. 
\end{align*}
Let ${\mathbf t}\in {\rm GO}(V)$ be the element given by the action 
\begin{align}\label{def:t}
  {\mathbf t} \left(\begin{pmatrix} a & b \\ c & \overline{a} \end{pmatrix} \right) 
   = \begin{pmatrix} -\overline{a} & b \\ c & -a \end{pmatrix}.
\end{align}
Then, we have 
\begin{align*}
   H:={\rm GO}(V) = {\rm GSO}(V) \rtimes \{ 1, {\mathbf t}\}, \quad
   {\mathbf t} h {\mathbf t} = h^c.
\end{align*}

\subsection{Automorphic representations of {\rm GSO} and {\rm GO} }\label{secrepGO}

According to \cite[Section 1]{hst93}, \cite[Section 6.1]{tak11} (see also \cite[Section 5.1]{hn}), 
we briefly recall a description of automorphic representations of $H^0={\rm GSO}(V)$ and $H={\rm GO}(V)$.

Let $\pi$ be a unitary irreducible cuspidal automorphic representation of ${\rm GL}_2(F_{\mathbf A})$ with the trivial central character
and ${\mathbf 1}$ the trivial character of ${\mathbf A}^\times$. 
Suppose that $\pi$ is not conjugate self-dual. 
Since we have $H^0({\mathbf Q}) \cong {\rm GL}_2(F)\times_{F^\times} {\mathbf Q}^\times$, 
a representation $\sigma:= \pi\boxtimes {\mathbf 1}$ gives 
an irreducible automorphic representation of $H^0({\mathbf A})$. 
Decompose $\sigma \cong \otimes^\prime_v\sigma_v$.
Note that $\sigma_v\cong \pi_v\boxtimes {\mathbf 1}$ and that $\sigma_v$ is realized in the same representation space with $\pi_v$. 
For each finite place $v$ of ${\mathbf Q}$, $\varrho(U_0({\mathfrak N})_v)$-invariant subspace of $\sigma_v$ is one dimensional by the theory of newform for $\pi$.

We extend $\sigma=\otimes^\prime_v \sigma_v$ to  an irreducible automorphic representation $\widetilde{\sigma}$ of $H({\mathbf A})$ as follows.  
Let $\sigma^\sharp_v := {\rm Ind}^{H({\mathbf Q}_v)}_{H^0({\mathbf Q}_v)}\sigma_v$ be the induced representation. 
Denote by ${\mathcal V}_v$ a representation space of $\pi_v$ and   
   we consider ${\mathcal V}_v$ as a representation space of $\sigma_v$.   
We find that the representation space ${\mathcal V}^\sharp_v$ of $\sigma^\sharp_v$ is ${\mathcal V}^{\oplus 2}_v$, where 
$H({\mathbf Q}_v)$ acts as follows
\begin{itemize}
\item $\sigma^\sharp_v(h) (x,y) = (\sigma_v(h) x, \sigma_v( h^c ) y)$ for $h\in H^0({\mathbf Q}_v)$;
\item $\sigma^\sharp_v({\mathbf t}_v) (x,y) = (y, x)$
\end{itemize}
for $x,y \in {\mathcal V}_v$. 
Let $\delta : \Sigma_{\mathbf Q} \to \{ \pm 1 \}$ be a map  such that $\delta(v)=1$ 
     for all but finitely many $v\in \Sigma_{\mathbf Q}$
     and for $v\in \Sigma_{\mathbf Q}$ such that $\pi_v \not\cong \pi^c_v$.
For the convenience, we sometimes write $\delta(v)=+$ (resp. $-$) to denote $\delta(v)=1$ (resp. $-1$).  
Denote by ${\mathfrak S}\subset \Sigma_{\mathbf Q}$  the subset such that $\pi_v \cong \pi^c_v$, where $\pi^c_v$ is the conjugate of $\pi_v$.  
Note that $\infty\in {\mathfrak S}$.   
For $v\in {\mathfrak S}$, there exists an involution $\xi_v : \pi_v \to \pi_v$
such that $\xi_v\circ \sigma_v(h) = \sigma_v(h^c)\circ \xi_v$  for each $h\in H^0({\mathbf Q}_v)$.
We define  irreducible representations $\widehat{\sigma}_v$ and $\widehat{\sigma}^\pm_v$ of $H({\mathbf Q}_v)$ as follows:
\begin{itemize}
\item If $v \not\in {\mathfrak S}$, then define $\widehat{\sigma}_v$ to be $\sigma^\sharp_v$. Denote by $\widehat{\mathcal V}_v={\mathcal V}^\sharp_v$. 
\item If $v \in {\mathfrak S}$, then define $\widehat{\sigma}^{\delta(v)}_v$ to be the representation whose representation space is given by
         \begin{align*}
             \widehat{{\mathcal V}}^{\delta(v)}_v =  \left\{  (x,  \delta(v) \xi_v x) : x \in {\mathcal V}_v  \right\}. 
         \end{align*} 
         Then, $\widehat{\sigma}^{\delta(v)}_v$ is an irreducible representation.  
\end{itemize} 
Let $\widehat{\sigma}$ to be the automorphic representation 
${\rm Ind}^{H({\mathbf A})}_{H^0({\mathbf A})}\sigma$
of $H({\mathbf A})$.
Then, assuming that $\pi$ is not conjugate self-dual, 
it is known that $\widehat{\sigma}$  is written as (\cite[Proposition 6.2]{tak11})
\begin{align*}
  \widehat{\sigma} = \bigoplus_\delta  \bigotimes_{v\in {\mathfrak S}} \widehat{\sigma}^{\delta(v)}_v
                                   \bigotimes_{v\not\in {\mathfrak S}} \widehat{\sigma}_v.
\end{align*}
We define $\widetilde{\sigma}$ to be 
\begin{align*}
  \widetilde{\sigma} = \widehat{\sigma}^{-}_{\infty} \bigotimes_{ \substack{ v\in {\mathfrak S} \\  v<\infty  } } \widehat{\sigma}^+_v  
                                  \bigotimes_{v\not\in {\mathfrak S}} \widehat{\sigma}_v.
\end{align*}

\begin{rem}\label{rem:del}
Denote by $\widetilde{\sigma}^+$ an extension of $\sigma$ to an irreducible automorphic representation of $H({\mathbf A})$ 
   so that $\widetilde{\sigma}^{+}_\infty = \widehat{\sigma}^+_\infty$.
Then, \cite[Corollary 3]{hst93} shows that the image of each automorphic form in $\widetilde{\sigma}^+$ by the theta correspondence for $({\rm GO}(V), {\rm GSp}_4)$  
can not be holomorphic. 
Since we are interested in a construction of holomorphic Siegel modular forms, 
we concentrate to study the extension $\widetilde{\sigma}$ of $\sigma$ so that $\widetilde{\sigma}_\infty=\widehat{\sigma}^-_\infty$ as above.
In other words, hereafter we always assume that 
\begin{align*}
   \delta(\infty) = -1; \quad \delta (v) = 1 \quad (v<\infty).
\end{align*}
However, we usually write $\delta(v)$ for each $v\in \Sigma_{\mathbf Q}$
to emphasis the dependance of $\delta$.  
\end{rem}

\subsection{Automorphic forms on {\rm GO}}\label{secautoGO}

Let $f:{\rm GL}_2(F_{\mathbf A}) \to {\mathcal W}_{2n+2}({\mathbf C})$ be a normalized newform of weight $n+2$ on ${\rm GL}_2(F_{\mathbf A})$ 
and  $\pi$ the automorphic representation of ${\rm GL}_2(F_{\mathbf A})$ generated by $f$. 
In this subsection, we extend  $f$ to an automorphic forms $\widetilde{\mathbf f}$ on $H({\mathbf A}) ={\rm GO}(V_{\mathbf A})$ according to \cite[Section 1, Section 4]{hst93}.
We will define $\widetilde{\mathbf f}$ so that whose associated automorphic representation is $\widetilde{\sigma}$, which is introduced in Section \ref{secrepGO}.

Denote by ${\mathbf 1}$ to be the trivial character on ${\mathbf A}^\times$.  
We define ${\mathbf f} = f \boxtimes {\mathbf 1}:H^0({\mathbf A}) ={\rm GSO}(V_{\mathbf A}) \to {\mathcal W}_{2n+2}({\mathbf C})$ for $f\in {\mathcal S}_{n+2}(U_0({\mathfrak N}))$.
Since we assumed that the central character of $\pi$ is trivial, ${\mathbf f}$ is well-defined   
and ${\mathbf f}$ is an automorphic form on $H^0({\mathbf A})$. 
Put
\begin{align}\label{def:ufrN}
   {\mathcal U}_{\mathfrak N} 
   = \prod_v {\mathcal U}_{\mathfrak N, v}
   = \left\{  \varrho(g,\alpha) \in H^0( {\mathbf A}_{\rm fin} ) \ | \  g \in U_0({\mathfrak N}), \alpha \in \widehat{\mathbf Z}^\times   \right\}.  
\end{align}
We denote by ${\mathcal M}_{n}(H^0, {\mathcal U}_{\mathfrak N})$ 
the space of the ${\mathbf C}$-linear combinations of ${\mathbf f}=f\boxtimes{\mathbf 1}$ for $f\in {\mathcal S}_{n+2}(U_0({\mathfrak N}))$. 

Define $\delta$ to be the function on the set $\Sigma_{\mathbf Q}$ of places of ${\mathbf Q}$ as in Remark \ref{rem:del}. 
Denote by ${\mathfrak N}$ the conductor of $\pi$ 
and we define $\Sigma_{\rm ram}$ to be the set of places of ${\mathbf Q}$ dividing ${\mathfrak N}\cap {\mathbf Z}$, 
the infinite place of ${\mathbf Q}$ and ramified places in $F/{\mathbf Q}$.

Let $W_f$ be the Whittaker function of $f$ and $\Sigma_1= {\mathfrak S} \cap \Sigma_{\rm ram}$. 
We define $W_{\mathbf f}(\varrho(g,\alpha)) =  W_f(g)$ for $\varrho(g,\alpha)\in {\rm GSO}(V)({\mathbf A})$.  
For ${\mathcal R}\subset \Sigma_{\mathbf Q}$, put $\delta_{\mathcal R}:=\prod_{v\in {\mathcal R}} \delta(v)$  
   and 
\begin{align}\label{def:whR}
W_{ {\mathbf f}, {\mathcal R}}(h):=
\begin{cases} \delta_{\mathcal R} W_{\mathbf f}(h^c_{\mathcal R} h^{\mathcal R}),  &  ( {\mathcal R} \subset \Sigma_1),  \\
                       0,  & ({\mathcal R} \not\subset \Sigma_1), \end{cases} 
\end{align}
  where $h^c_{\mathcal R} h^{\mathcal R} =$ $ \prod_{v\in {\mathcal R}} h^c_v \prod_{v\in \Sigma_{\mathbf Q}\backslash {\mathcal R}} h_v$. 
Define ${\mathbf f}_{\mathcal R}: H^0({\mathbf A}) \to {\mathcal W}_{2n+2}({\mathbf C})$ to be the automorphic form on $H^0({\mathbf A})$ which corresponds to 
  a Whittaker function $W_{{\mathbf f}, \mathcal R}$: 
\begin{align}
     {\mathbf f}_{\mathcal R} (h) 
  = \sum_{\xi \in F^\times } W_{ {\mathbf f}, {\mathcal R} }\left( \begin{pmatrix} \xi & 0 \\ 0 & 1 \end{pmatrix} h  \right).     \label{eq:whittdef}
\end{align}
Recall that the symmetric difference $A\triangle B $ of sets $A$ and $B$ is defined to be  $(A\cup B)\backslash (A\cap B)$.
Then we define $\widetilde{\mathbf f}: H({\mathbf A}) \to {\mathcal W}_{2n+2}({\mathbf C})^{\oplus 2}$ by 
\begin{align}\label{def:extf}
  \widetilde{\mathbf f}(h {\mathbf t}_{\mathcal R})  
 =  ( {\mathbf f}_{\mathcal R^\prime} (h) + {\mathbf f}_{\Sigma_1\backslash \mathcal R^\prime}(h^c),  
                  {\mathbf f}_{\mathcal R^\prime \triangle \{\infty\}} (h) + {\mathbf f}_{(\Sigma_1\backslash \mathcal R^\prime) \triangle \{\infty\} }(h^c) ),  
                  \quad (h\in H^0({\mathbf A})),  
\end{align}
where ${\mathcal R}^\prime = {\mathcal R}\cap \Sigma_1.$
The following proposition is checked in the straight forward way:

\begin{prop}\label{prop:ext}
{\itshape
Let $h\in H^0({\mathbf A}), {\mathcal R}\subset \Sigma_{\mathbf Q}$.
Then, we have the following statements:
\begin{enumerate}
\item $\widetilde{\mathbf f}({\mathbf t} h {\mathbf t}_{\mathcal R})  =  \widetilde{\mathbf f}( h {\mathbf t}_{\mathcal R})$.
\item For $u\in {\rm SU}_2({\mathbf R})$, 
         $  \widetilde{\mathbf f}( h {\mathbf t}_{\mathcal R} u)
         =  (\tau_{2n+2}(u^{-1}), \tau_{2n+2}(\overline{u}^{-1})) \widetilde{\mathbf f}( h {\mathbf t}_{\mathcal R})$.
\item   $  \widetilde{\mathbf f}({\mathbf t} h {\mathbf t}_{\mathcal R} {\mathbf t}_\infty)
    =    (  {\mathbf f}_{\mathcal R^\prime \triangle \{\infty\}} (h) + {\mathbf f}_{(\Sigma_1\backslash \mathcal R^\prime) \triangle \{\infty\} }(h^c), 
             {\mathbf f}_{\mathcal R^\prime} (h) + {\mathbf f}_{\Sigma_1\backslash \mathcal R^\prime}(h^c) )$. 
\item For $v\in \Sigma_{\mathbf Q}\backslash \Sigma_1$, $\widetilde{\mathbf f}({\mathbf t} h {\mathbf t}_{\mathcal R} {\mathbf t}_v)  =  \widetilde{\mathbf f}( h {\mathbf t}_{\mathcal R})$.
\end{enumerate}
}
\end{prop}

Proposition \ref{prop:ext} shows that 
 $\widetilde{\mathbf f}$ gives an automorphic forms on $H({\mathbf A})$ and 
the automorphic representation which is associated with $\widetilde{\mathbf f}$ is $\widetilde{\sigma}$. 

For each finite place $w$ of $F$ dividing ${\mathfrak N}$,  let 
\begin{align}\label{fricke}
   n_w = {\rm ord}_w({\mathfrak N})   ; \quad  \eta_{{\mathfrak N}, w} = \begin{pmatrix} 0 & \delta^{-1}_{F, w} \\ -\varpi^{n_w}_w \delta_{F,w} &  0  \end{pmatrix}.  
\end{align}
Since the central character of $\pi$ is trivial, we have the Atkin-Lehner eigenvalue $\varepsilon(\pi_w)$ of $f$: 
\begin{align*}
   {\mathbf f}(h \eta_{{\mathfrak N}, w}) = \varepsilon(\pi_w)  {\mathbf f}(h),  \quad (h\in H^0({\mathbf A})).
\end{align*}

Let 
\begin{align}
  {\mathcal P}
  = \left\{ \text{prime factors } p \text{ of } N \ | \  p\nmid {\mathfrak N} 
              \right\}.  
    \label{defP}
\end{align}    
To make the explicit inner product formula and the Bessel periods   formula concise, 
    we need the level raising operator ${\mathscr V}_{\mathcal P}$, which is defined as follows. 
     (See also \cite[Section 4.4]{hn}  for the definition of ${\mathscr V}_{\mathcal P}$.)  
For each $p=w\overline{w} \in {\mathfrak N}$ with $w\nmid {\mathfrak N}$ and $\overline{w}\mid {\mathfrak N}$, 
let 
\begin{align}\label{raiseopet}
    \eta_{{\mathfrak N}, w}  = \begin{pmatrix} 0 & \delta^{-1}_{F, w} \\ -\varpi^{n_{\overline{w}}}_w \delta_{F,w}  &  0  \end{pmatrix}, 
    \quad (n_{\overline{w}} = {\rm ord}_{\overline{w}} ({\mathfrak N})  ). 
\end{align}
By using the Atkin-Lehner eigenvalue $\varepsilon(\pi_{\overline{w}})$ of $\pi_{\overline{w}}$, define  
\begin{align}\label{raiseop}
   {\mathscr V}_p(\widetilde{\mathbf f})(h) 
   =  \widetilde{\mathbf f}(h)  +  \varepsilon(\pi_{\overline{w}}) \widetilde{\mathbf f}(h \eta_{{\mathfrak N}, w}).   
\end{align}
Then we define ${\mathcal P}$-stabilized newform $\widetilde{\mathbf f}^\dag$  to be 
\begin{align*}
     \widetilde{\mathbf f}^\dag = {\mathscr V}_{\mathcal P} (\widetilde{\mathbf f}); 
     \quad {\mathscr V}_{\mathcal P} = \prod_{p\in {\mathcal P}}  {\mathscr V}_p.
\end{align*}

\subsection{Minimal $K$-type}\label{sec:minK}

In this subsection, we describe the minimal $K$-type of $\widetilde{\sigma}_\infty$. 
This will be necessary  to compute the minimal $K$-type of an image of a theta correspondence from $\widetilde{\sigma}_\infty$ in Lemma \ref{HSTwt}.  
We start with some basic properties of $\pi_\infty$. 

\begin{lem}\label{l:whitt}
{\itshape 
 Let ${\mathcal W}(\pi_\infty, \psi_{F,\infty})$ be the Whittaker model of $\pi_\infty$ 
         and  $W^j_{\pi,\infty}$ an element of ${\mathcal W}(\pi_\infty, \psi_{F,\infty})$ which is characterized by the following identity:
         \begin{align}
             W^j_{\pi,\infty} \left( \begin{pmatrix} t & 0 \\ 0 & 1 \end{pmatrix}  \right)   \label{eq:expwhitt}  
             =
              2^4 \sqrt{-1}^j t^{n+2}K_j(4\pi t) (X^{n+1+j} Y^{n+1-j})^\vee. 
          \end{align}
          Let 
          \begin{align*}
             W_{\pi,\infty}(g) := W_{\pi,\infty}(g; (X,Y)) =  \sum^{n+1}_{j=-n-1}  W^j_{\pi,\infty} (g)    .
          \end{align*}
          Then, 
          $\left\{ W^j_{\pi,\infty}: j=-n-1,\ldots, n+1\right\}$ is a pair of basis of the minimal ${\rm SU}_2({\mathbf R})$-type of ${\mathcal W}(\pi_\infty, \psi_{F,\infty})$
          and 
          $W_{\pi,\infty} \in {\mathcal W}(\pi_\infty, \psi_{F,\infty})^{{\rm SU}_2({\mathbf R})}$.   
          In particular, the following map  
          $\iota_\infty: {\mathcal W}_{2n+2}({\mathbf C}) = {\mathbf C}[X,Y]_{2n+2}\to {\mathcal V}_{\infty}$ 
          is  ${\rm SU}_2({\mathbf R})$-equivariant
          \begin{align*}
              \iota_\infty(u) = \langle  W_{\sigma, \infty} , u\rangle_{\mathcal W}  
          \end{align*}
          where $W_{\sigma, \infty}: H^0_1({\mathbf R}) \to {\mathbf C}$ is given by
          \begin{align*}
                W_{\sigma, \infty}(\varrho(g,\alpha)) = W_{\pi, \infty}(g), \quad (g\in {\rm GL}_2({\mathbf C}), \alpha \in {\mathbf C}^\times).
          \end{align*}
}
\end{lem}
\begin{proof}
These properties of the Whittaker model of $\pi_\infty$ can be found in \cite[Section 6]{hi94cr}.
\end{proof}

We sometimes abbreviate $W_{\sigma, \infty}(\varrho(g,\alpha))$ to $W_{\sigma, \infty}(g)$.
Let ${\mathcal W}(\sigma_\infty, \psi_{F,\infty})$ be the regular representation of $H^0({\mathbf R})$ which is spanned by $W_{\sigma, \infty}$.
Fix $\xi_\infty: {\mathcal W}(\sigma_\infty, \psi_{F,\infty}) \to {\mathcal W}(\sigma_\infty, \psi_{F,\infty})$ by 
\begin{align*}
   \xi_\infty W_{\sigma, \infty}(h) =  W_{\sigma, \infty}( h^c ).
\end{align*}
The following lemma is useful  to describe the minimal $K$-type of $\widetilde{\sigma}_\infty$ in an explicit manner.  

\begin{lem}\label{l:minwhitt}
{\itshape
We have 
\begin{align*}
   \xi_\infty W_{\sigma, \infty} = (-1)^{n+1} \tau_{2n+2}(w_0)  W_{\sigma, \infty}. 
\end{align*}
}
\end{lem}
\begin{proof}
Let $W^j_{\sigma, \infty}$ be the element in Lemma \ref{l:whitt}. 
Then, as explained in Lemma \ref{l:whitt}, 
$\{ W^j_{\sigma, \infty}: j=-n-1, \ldots, n+1 \}$ gives a pair of basis of the minimal ${\rm SU}_2({\mathbf R})$-type of ${\mathcal V}_\infty$. 
Then the action of  $\xi_\infty$ on  $W^j_{\sigma, \infty}$ is given explicitly as follows: 
\begin{align*}
    \xi_\infty W^j_{\sigma, \infty}(h)  
    =  W^j_{\sigma, \infty}(h^c).
\end{align*}
Let $w_0 = \begin{pmatrix} 0 & 1 \\ -1 & 0 \end{pmatrix}$.  
Since $W_{\sigma,\infty}(hw_0; (X,Y)) = W_{\sigma,\infty}(h; (X,Y)w^{-1}_0)$, we find that 
\begin{align}
    W^j_{\sigma, \infty}(hw_0)  =   (-1)^{n+1+j} W^{-j}_{\sigma,\infty}(h). \label{eq:whittw0}    
\end{align}
By the basic property of the Bessel function $K_s(z)$ (\cite[page 67]{mos66}) and the definition of $W^j_{\sigma, \infty}$, we also find that
\begin{align}
    (-1)^j W^{-j}_{\sigma, \infty}\left(  \begin{pmatrix} t & 0 \\  0 & 1 \end{pmatrix} \right)  =  W^{j}_{\sigma, \infty}\left(  \begin{pmatrix} t & 0 \\  0 & 1 \end{pmatrix} \right).  \label{eq:whitt-j}   
\end{align}
Consider the Iwasawa decomposition 
\begin{align*}
  h= a \begin{pmatrix} t & x \\ 0 & 1  \end{pmatrix}u, 
    \quad (a\in {\mathbf C}^\times, t\in {\mathbf R}^\times, x \in {\mathbf C}, u\in {\rm SU}_2({\mathbf R})).
\end{align*}
Then, we find the following two identities: 
\begin{align*}
  \xi_\infty W_{\sigma, \infty}\left( h \right) 
  =  \sum^{n+1}_{j=-n-1}   \xi_\infty W^j_{\sigma, \infty} (h) u^\vee_j   
  = & \sum^{n+1}_{j=-n-1}   W^j_{\sigma, \infty}\left( \begin{pmatrix}  t & \overline{x} \\  0 & 1  \end{pmatrix} \bar{u} \right) u^\vee_j   \\  \displaybreak[0]
  = & \sum^{n+1}_{j=-n-1}  \psi_{F, \infty}(x) W^j_{\sigma, \infty}\left( \begin{pmatrix}  t & 0 \\  0 & 1  \end{pmatrix}   \right)  \tau_{2n+2}( \overline{u}^{-1} ) u^\vee_j,   \\ \displaybreak[0]
W_{\sigma, \infty}\left( h w_0 \right)
  =   \sum^{n+1}_{j=-n-1}    W^j_{\sigma, \infty} (hw_0) u^\vee_j   
  = & \sum^{n+1}_{j=-n-1}   W^j_{\sigma, \infty}\left( \begin{pmatrix}  t & x \\  0 & 1  \end{pmatrix} w_0 \bar{u} \right) u^\vee_j   \\   \displaybreak[0]
  = & \sum^{n+1}_{j=-n-1}  \psi_{F, \infty}(x) W^j_{\sigma, \infty}\left( \begin{pmatrix}  t & 0 \\  0 & 1  \end{pmatrix}  w_0 \right)  \tau_{2n+2}( \overline{u}^{-1} ) u^\vee_j,   \\  \displaybreak[0]
  \stackrel{(\ref{eq:whittw0})}{= } 
     & \sum^{n+1}_{j=-n-1}  \psi_{F, \infty}(x) (-1)^{n+1+j} W^{-j}_{\sigma, \infty}\left( \begin{pmatrix}  t & 0 \\  0 & 1  \end{pmatrix}  \right)  \tau_{2n+2}( \overline{u}^{-1} ) u^\vee_j   \\  \displaybreak[0]
  \stackrel{(\ref{eq:whitt-j})}{= } 
    & \sum^{n+1}_{j=-n-1}  \psi_{F, \infty}(x) (-1)^{n+1} W^{j}_{\sigma, \infty}\left( \begin{pmatrix}  t & 0 \\  0 & 1  \end{pmatrix}  \right)  \tau_{2n+2}( \overline{u}^{-1} ) u^\vee_{j}.  
\end{align*}
By comparing these two, we obtain
\begin{align*}
       \xi_\infty W^j_{\sigma, \infty}(h) = (-1)^{n+1}W^j_{\sigma, \infty} (hw_0). 
\end{align*}
This proves the lemma.
\end{proof}

Let $w_0=\begin{pmatrix} 0 & 1 \\ -1 & 0 \end{pmatrix}$. 
We define a representation $( \widetilde{\tau}_{2n+2}, \widetilde{\mathcal W}_{2n+2}({\mathbf C}))$ of ${\rm SU}_2({\mathbf R})  \rtimes \{ 1, {\mathbf t}\} $ as follows:
\begin{align*}
 & \widetilde{W}_{2n+2}({\mathbf C}) 
     = \left\{  (u,     (-1)^{n+1}  \delta(\infty) \tau_{2n+2}(w_0) u) 
              \in    {\mathcal W}_{2n+2}({\mathbf C})^{\oplus 2} 
           | \    u \in {\mathcal W}_{2n+2}({\mathbf C})   \right\},   \\
 & \quad \widetilde{\tau}_{2n+2}(h) (u, (-1)^{n+1} \delta(\infty) \tau_{2n+2}(w_0) u) = (\tau_{2n+2}(g) u, (-1)^{n+1} \delta(\infty) \tau_{2n+2}(g^c) \tau(w_0) u)  \\
      & \quad \quad \quad \quad \quad \quad \quad \quad \quad \quad \quad \quad \quad \quad \quad \quad 
          \quad \ 
          = (\tau_{2n+2}(g) u, (-1)^{n+1} \delta(\infty) \tau(w_0) \tau_{2n+2}(g)  u), \\
    & \quad \quad \quad \quad \quad \quad \quad \quad \quad \quad \quad \quad \quad \quad \quad \quad 
          \quad \quad \quad \quad 
           (h=\varrho(g, 1) \in H^0({\mathbf R}), g\in {\rm SU}_2({\mathbf R})), \\   
 & \quad   \widetilde{\tau}_{2n+2}({\mathbf t}_\infty) (u, (-1)^{n+1} \delta(\infty) \tau_{2n+2}(w_0) u)  
      =  ( (-1)^{n+1} \delta(\infty) \tau_{2n+2}(w_0) u,   u).
\end{align*}
Since the above formula implies 
\begin{align*}
     \widetilde{\tau}_{2n+2}({\mathbf t}_\infty) \widetilde{\tau}_{2n+2}(\varrho(g, 1))
  =  \widetilde{\tau}_{2n+2}(\varrho(g^c, 1)) \widetilde{\tau}_{2n+2}({\mathbf t}_\infty),  
\end{align*}
$( \widetilde{\tau}_{2n+2}, \widetilde{\mathcal W}_{2n+2}({\mathbf C}))$ is well-defined.   
By Proposition \ref{prop:ext} and Lemma \ref{l:minwhitt}, the regular representation of ${\rm SU}_2({\mathbf R})  \rtimes \{ 1, {\mathbf t}\} $
   which is generated by $\widetilde{\mathbf f}$ is isomorphic to $( \widetilde{\tau}_{2n+2}, \widetilde{\mathcal W}_{2n+2}({\mathbf C}))$. 
Hence we consider $\widetilde{\mathbf f}$  as a $\widetilde{\mathcal W}_{2n+2}({\mathbf C})$-valued function
    $\widetilde{\mathbf f}:H({\mathbf A}) \to \widetilde{\mathcal W}_{2n+2}({\mathbf C})$.
Denote by ${\mathcal M}_{2n+2}(H, {\mathfrak N})$ the ${\mathbf C}$-span of automorphic forms $\widetilde{\mathbf f}$ which are defined as above.

 It is also convenient to realize $\widetilde{\sigma}$ (resp. $\sigma$)
 as a space of moderate growth smooth functions $H({\mathbf A}) \to {\mathbf C}$
 (resp. $H^0({\mathbf A}) \to {\mathbf C}$), 
which we denote by ${\mathcal A}(\widetilde{\sigma})$ (resp. ${\mathcal A}(\sigma)$).   
W briefly explain a relation between this realization and the space of automorphic forms 
$\widetilde{\mathbf f}: H({\mathbf A}) \to  \widetilde{\mathcal W}_{2n+2}({\mathbf C} )$ which are defined as above.
Define $\langle \cdot, \cdot\rangle_{ \widetilde{\mathcal W}} : ({\mathcal W}_{2n+2}({\mathbf C})^{\oplus 2})^{\otimes 2}\to {\mathbf C}$ to be the pairing given by
\begin{align*}
   \langle (u_1, v_1), (u_2, v_2) \rangle_{\widetilde{\mathcal W}}
   = \langle u_1, u_2 \rangle_{2n+2}
      + \langle v_1, v_2 \rangle_{2n+2}. 
\end{align*}
We denote the restriction of $\langle\cdot, \cdot \rangle_{\widetilde{\mathcal W}}$ to $\widetilde{W}_{2n+2}({\mathbf C})$ by the same notation. 
Then the pairing $\langle\cdot, \cdot \rangle_{\widetilde{\mathcal W}}$ on $\widetilde{W}_{2n+2}({\mathbf C})^{\otimes 2}$ is ${\rm SU}_2({\mathbf R})  \rtimes \{ 1, {\mathbf t}\}$-equivariant. 
Let $\widetilde{u} = (u,  (-1)^{n+1} \delta(\infty) \tau(w_0) u) \in \widetilde{\mathcal W}_{2n+2}({\mathbf C})$  for $u\in {\mathcal W}_{2n+2}({\mathbf C})$. 
Then, the embedding 
\begin{align*}
   \widetilde{\mathcal W}_{2n+2}({\mathbf C}) \to \langle  \widetilde{\sigma}(h) \widetilde{\mathbf f} : h\in H({\mathbf A}) \rangle_{\mathbf C};  
    \widetilde{u} \mapsto \langle \widetilde{\mathbf f}, \widetilde{u} \rangle_{\widetilde{\mathcal W}}
\end{align*}
is ${\rm SU}_2({\mathbf R})  \rtimes \{ 1, {\mathbf t}\}$-equivariant by Proposition \ref{prop:ext} and Lemma \ref{l:minwhitt}. 
Hence ${\mathcal A}(\widetilde{\sigma})$ is the regular representation of $H({\mathbf A})$ which is generated by $\langle \widetilde{\mathbf f}, \widetilde{u} \rangle_{\widetilde{\mathcal W}}$ 
where $\widetilde{\mathbf f}:H({\mathbf A}) \to \widetilde{\mathcal W}_{2n+2}({\mathbf C})$ is as above and $u \in {\mathcal W}_{2n+2}({\mathbf C})$.

\section{HST lifts}\label{s:hst}

In this section, we give an explicit construction of HST lifts 
   by fixing a distinguished Bruhat-Schwartz function $\widetilde{\varphi}$.   
Firstly,  we recall the Schr\"odinger realization of the Weil representation in Section \ref{sec:weil} 
and the definition of theta series in Section \ref{sectheta}.   
The definition of $\widetilde{\varphi}$ and basic properties of HST lifts are given in Section \ref{sec:test}.  
The Fourier expansion of HST lifts is given in adelic language in Section \ref{secFC}, 
which will be studied in classical language in Section \ref{sec:claBess} 
    after we provide an explicit Bessel periods formula for HST lifts.

\subsection{Weil representation on ${\rm O}(V)\times{\rm Sp}_4$}\label{sec:weil}
Let $(V,{\rm n})$ be a four dimensional quadratic space over the rational number field ${\mathbf Q}$ 
and let $(\cdot, \cdot): V\times V \to {\mathbf Q}$ be the bilinear form defined by
$(x,y) = {\rm n}(x+y) - {\rm n}(x) - {\rm n}(y)$. 
Denote by ${\rm GO}(V)$ the orthogonal similitude group with the similitude character 
$\nu:{\rm GO}(V) \to {\mathbb G}_m$.
Let ${\mathbf X} = V\oplus V$.
For each $x=(x_1,x_2)\in {\mathbf X}$, we put
\begin{align}\label{def:Sx}
  S_x= \begin{pmatrix} {\rm n}(x_1) & \frac{1}{2}(x_1,x_2)  \\
                                \frac{1}{2}(x_1,x_2) & {\rm n}(x_2)   \end{pmatrix}. 
\end{align}
Let $v$ be a place of ${\mathbf Q}$ and $|\cdot|_v$ be the normalized absolute value on ${\mathbf Q}_v$.
Let $V_v=V\otimes_{\mathbf Q}{\mathbf Q}_v$ and ${\mathbf X}_v={\mathbf X}\otimes_{\mathbf Q}{\mathbf Q}_v$.
We denote by ${\mathcal S}({\mathbf X}_v)$ the space of ${\mathbf C}$-valued Bruhat-Schwartz functions on ${\mathbf X}_v$. 
Define $S_x$ for $x=(x_1,x_2)\in {\mathbf X}_v=V_v\oplus V_v$ in the same way with (\ref{def:Sx}).

Let $\chi_{V_v}:{\mathbf Q}^\times_v\to {\mathbf C}^\times$ be the quadratic character attached to $V_v$.
We recall the Schr\"odinger realization of the Weil representation 
$\omega_{V_v}:{\rm Sp}_4({\mathbf Q}_v)\to {\rm Aut}_{\mathbf C}{\mathcal S}({\mathbf X}_v)$ according to \cite[Section 4.2]{ic05}.
For $\varphi\in {\mathcal S}({\mathbf X}_v)$, we have
\begin{align}
\begin{aligned}
\omega_{V_v} \left(  \begin{pmatrix} a & 0_2 \\ 0_2 & {}^{\rm t}a^{-1} \end{pmatrix} \right) \varphi(x)
 = & \chi_{V_v}(\det a ) |\det a|^2_v\cdot \varphi(xa),  \\
\omega_{V_v} \left(  \begin{pmatrix} 1_2 & b \\ 0_2 & 1_2 \end{pmatrix} \right) \varphi(x)
 = & \psi_{{\mathbf Q},v}({\rm Tr}(S_xb)) \cdot \varphi(x),  \\
\omega_{V_v} \left(  \begin{pmatrix} 0_2 & 1_2 \\ -1_2 & 0_2 \end{pmatrix} \right) \varphi(x)
 = & \gamma^{-2}_{V_v} \cdot \widehat{\varphi}(x),  
\end{aligned} \label{eq:weilsch}
\end{align}
where $\gamma_{V_v}$ is the Weil index of $V_v$, and $\widehat{\varphi}$ is the Fourier transform of $\varphi$
with respect to the self-dual Haar measure ${\rm d}\mu$ on $V_v\oplus V_v$:
\begin{align*}
 \widehat{\varphi}(x) := \int_{{\mathbf X}_v} \varphi(y) \psi_{ {\mathbf Q}, v}((x,y)) {\rm d}\mu(y).  
\end{align*}
The Weil representation 
$\omega_v:R({\rm GO}(V_v)\times {\rm GSp}_4({\mathbf Q}_v))\to {\rm Aut}_{\mathbf C}{\mathcal S}({\mathbf X}_v)$
is given by
\begin{align*}
    \omega_p(h,g)\varphi(x)
  = |\nu(h)|^{-2}_v(\omega_{V_v}(g_1)\varphi)(h^{-1}x), 
\quad \left( g_1 = \begin{pmatrix} 1_2 & 0_2 \\ 0_2 & \nu(g)^{-1} 1_2 \end{pmatrix}g \right). 
\end{align*}
Let ${\mathcal S}({\mathbf X}_{\mathbf A}) = \otimes_v{\mathcal S}({\mathbf X}_v)$.
We define
$\omega_V = \otimes_v\omega_{V_v}: {\rm Sp}_4({\mathbf A}) \to {\rm Aut}_{\mathbf C}{\mathcal S}({\mathbf X}_{\mathbf A})$ 
and 
$\omega=\otimes_v\omega_v: R({\rm GO}(V)_{\mathbf A}\times {\rm GSp}_4({\mathbf A}))
\to {\rm Aut}_{\mathbf C}{\mathcal S}({\mathbf X}_{\mathbf A})$.

\begin{rem}
In this paper,  we always take $V$ to be $\left\{  x \in {\rm M}_2(F) : \overline{x}^\ast = x \right\}$ which is introduced in Section \ref{secorth}. 
Define ${\rm n}(x) = {\rm det}(x)$ for $x \in V$. Then, the quadratic space ($V, {\rm n}$) has the signature $(3,1)$. 
In this case, the bilinear form $(\cdot, \cdot)$ is given by the following formula: 
\begin{align}\label{eq:bitr}
   (x,y) = {\rm Tr}(x y^\ast). 
\end{align} 
\end{rem}

\subsection{Theta lifts of Harris-Soudry-Taylor}\label{sectheta}

We put $\lambda=(n+2, 2)$. 
Denote by ${\mathcal L}_\lambda({\mathbf C})$ the algebraic representation of ${\rm GL}_2({\mathbf C})$ of the highest weight $\lambda$ which is introduced in Section \ref{algrep}. 
Let $V$ be the four dimensional orthogonal space which was introduced in Section \ref{secorth}, 
and $\widetilde{\mathbf f}$ an extension o ${\mathbf f}$ which was introduced in \ref{secautoGO}.
We will choose a distinguished Bruhat-Schwartz function 
$\widetilde{\varphi} \in {\mathcal S}({\mathbf X}_{\mathbf A})\otimes \widetilde{\mathcal W}_{2n+2}({\mathbf C})^\vee\otimes {\mathcal L}_\lambda({\mathbf C})$ in the next subsection, 
and consider the theta kernel $\theta(h,g; \widetilde{\varphi} )$:
\begin{align*}
   \theta(h,g; \widetilde{\varphi})
   =\sum_{x\in {\mathbf X} } \omega(h,g) \widetilde{\varphi} (x).
\end{align*}
Define the theta lift $\theta( \widetilde{\varphi} , \tilde{\mathbf f}): {\rm GSp}_4({\mathbf A}) \to {\mathcal L}_\lambda({\mathbf C})$ to be
\begin{align*}
   \theta(\widetilde{\varphi}, \widetilde{\mathbf f})(g)  
   =  \theta(g;  \widetilde{\varphi},  \tilde{\mathbf f} ) 
   = \int_{[{\rm O}(V)]} \langle \theta(g,hh_1; \widetilde{\varphi} ), \tilde{\mathbf f}(hh_1) \rangle_{\widetilde{\mathcal W}} {\rm d}h, 
       \quad (\nu(h_1) = \nu(g)), 
\end{align*}
where ${\rm d}h$ is the Tamagawa measure on $H_1({\mathbf A}):={\rm O}(V)({\mathbf A})$.
We call $\theta(\widetilde{\varphi},  \tilde{\mathbf f} )$ HST lifts, since it was studied closely in \cite{hst93}. 

\subsection{The choice of test functions}\label{sec:test}

We introduce our choice of the distinguished test function  
$\widetilde{\varphi} = \widetilde{\varphi}_\infty\otimes_{v<\infty} \varphi_v:  
                                   V^{\oplus 2}_{\mathbf A} \to \widetilde{\mathcal W}_{2n+2}({\mathbf C})^\vee \otimes {\mathcal L}_\lambda({\mathbf C})$  
and we prove basic properties of $\widetilde{\varphi}$ in Lemma \ref{HSTlev} and Lemma \ref{HSTwt}.

Let $N$ be a positive integer such that  $N{\mathbf Z}={\mathfrak N}\cap {\mathbf Z}$.
Let ${\mathfrak  d}_F$ be the different of $F/{\mathbf Q}$ and let $\Delta_F = {\rm Nr}_{F/{\mathbf Q}}({\mathfrak d}_F)$ be the discriminant of $F$.
For each finite place $v\in \Sigma_{\mathbf Q}$, we let 
\begin{align}
  V^\prime({\mathcal O}_{F,v}) 
         = \left\{  \begin{pmatrix} a & b\delta^{-1}_F \\ c\delta_F  & \bar{a}  \end{pmatrix} \in V_v 
                    \ | \  a  \in {\mathcal O}_{F,v}, b\in {\mathbf Z}_v,   c \in N {\mathbf Z}_v  \right\}.  \label{def:lattice}
\end{align}  
Define a distinguished Bruhat-Schwartz function $\varphi_v$ for each finite place $v\in \Sigma_{\mathbf Q}$ to be 
the characteristic function of $V^\prime({\mathcal O}_{F,v})^{\oplus2} $ on $V^{\oplus 2}_v$.

\begin{lem}\label{HSTlev}
{\itshape 
Let $v\in \Sigma_{\mathbf Q}$ be a finite place.  
Denote by $N_F$   the least common multiple ${\rm l.c.m.}(N, \Delta_F)$  of $N$ and $\Delta_F$.
Then, for $g = \begin{pmatrix} a & b \\ c &d \end{pmatrix} \in U^{(2)}_0(N_F)\cap {\rm Sp}_4(\widehat{\mathbf Z})$, we have 
\begin{align*}
   \omega_{V_v}(g)\varphi_v = \chi_{V_v}({\rm det} a) \varphi_v. 
\end{align*} 
}
\end{lem}
\begin{proof}
Put $V^\prime({\mathcal O}_{F,v})^\vee = \{ x\in V_v | (x,y) \in {\mathbf Z}_v \text{ for all } y \in V^\prime({\mathcal O}_{F,v}) \}$.
Let $l_v$ be the non-negative integer such that $p^{-l_v}{\mathbf Z}_v = \langle {\rm n}(x)  | x \in V^\prime({\mathcal O}_{F,v})^\vee \rangle_{{\mathbf Z}_v}$.
Then, by \cite[Lemma 2.1]{yo84}, it is enough to prove that $l_v={\rm ord}_v(N_F)$. 
Since it is easy to see that $V^\prime({\mathcal O}_{F,v})^\vee 
    = V_v \cap \left\{  \begin{pmatrix} {\mathfrak d}^{-1}_F {\mathcal O}_{F,v} &  N^{-1} {\mathcal O}_{F,v} \\  
                                                           {\mathcal O}_{F,v} & {\mathfrak d}^{-1}_F{\mathcal O}_{F,v}  \end{pmatrix} \right\}$, 
 the lemma follows immedeately.
\end{proof}   

Now, we proceed to define an archimedean test function 
 \begin{align*}  
   \widetilde{\varphi}_\infty \in {\mathcal S}({\mathbf X}_\infty) \otimes \widetilde{\mathcal W}_{2n+2}({\mathbf C})^\vee \otimes {\mathcal L}_\lambda({\mathbf C}).
 \end{align*}      
We firstly prepare a Bruhat-Schwartz function $\varphi_\infty  \in {\mathcal S}({\mathbf X}_\infty)\otimes {\mathbf C}[X,Y]_{2n+2} \otimes {\mathbf C}[X,Y]_n$. 
Let
\begin{align*}
  V^0 = \left\{ \begin{pmatrix} a & b \\ b & c \end{pmatrix} \in {\rm M}_2({\mathbf C})  \right\};  \quad  
  {\mathbf X}^0_\infty= (V_\infty\cap V^0)^{\oplus 2}. 
\end{align*}
Let ${\mathbf p} : V^0 \to  {\mathbf C}[X,Y]_2$ be the isomorphism given by
\begin{align*}
   {\mathbf p}\left( \begin{pmatrix} a & b \\ b & c \end{pmatrix}  \right)  =  a X^2 + 2 b XY + c Y^2. 
\end{align*}
For each integer $0\leq \alpha \leq n$, define $P^\alpha : {\mathbf X}^0_\infty \to {\mathbf C}[X,Y]_{2n+2}$ by 
\begin{align*}
     P^\alpha(x_1, x_2) = {\mathbf p}\left(\frac{1}{2} (x_1w_0x_2  +{}^{\rm t} (x_1 w_0 x_2) ) \right) 
                                       {\mathbf p}(x_1)^\alpha{\mathbf p}(x_2)^{n-\alpha}, 
                                        \quad \left(w_0=\begin{pmatrix} 0 & 1 \\  -1 & 0 \end{pmatrix} \right).
\end{align*}
Explicitly, $P^\alpha$ is written as follows: 
\begin{align*}
    & P^\alpha  \left( \left(  \begin{pmatrix} z_1 & \sqrt{-1} t_1 \\  \sqrt{-1} t_1 & \bar{z}_1  \end{pmatrix}, 
                                \begin{pmatrix} z_2 & \sqrt{-1} t_2 \\ \sqrt{-1} t_2 & \bar{z}_2 \end{pmatrix} \right) \right)  \\ 
 = &{\mathbf p}\left(  \begin{pmatrix}  \sqrt{-1} ( z_1 t_2- t_1z_2)  & \frac{1}{2}( z_1 \bar{z}_2  -  \bar{z}_1 z_2)     \\   
                                                           \frac{1}{2}( z_1 \bar{z}_2  -  \bar{z}_1 z_2 ) &   -\sqrt{-1}  ( \bar{z_1}t_2   -  t_1\bar{z_2})  \end{pmatrix} \right)  \\
   & \quad \times  {\mathbf p}\left( \begin{pmatrix} z_1 & \sqrt{-1}  t_1 \\  \sqrt{-1} t_1 & \bar{z}_1 \end{pmatrix} \right)^\alpha
    {\mathbf p}\left( \begin{pmatrix} z_2 & \sqrt{-1} t_2 \\  \sqrt{-1} t_2 & \bar{z}_2 \end{pmatrix} \right)^{n-\alpha}.
\end{align*}

\begin{lem}\label{pluri} 
{\itshape 
$P^\alpha$ is a pluri-harmonic polynomial.
}
\end{lem}
\begin{proof}
It suffices to prove that $P^\alpha$ is annihilated by the following three differential operators:
\begin{align*}
\frac{\partial^2}{\partial t^2_1} + 4 \frac{\partial^2}{\partial z_1 \partial \bar{z}_1}, \quad
\frac{\partial^2}{\partial t_1 \partial t_2} + 2 \left( \frac{\partial^2}{\partial z_1 \partial \bar{z}_2} + \frac{\partial^2}{\partial z_2 \partial \bar{z}_1} \right), \quad
\frac{\partial^2}{\partial t^2_2} + 4 \frac{\partial^2}{\partial z_2 \partial \bar{z}_2}.
\end{align*}
The computation is straightforward, hence we omit it.
\end{proof}

We define $\varphi^\alpha_\infty \in {\mathcal S}({\mathbf X}_\infty)\otimes{\mathbf C}[X,Y]_{2n+2}$ by
\begin{align*}
   \varphi^\alpha_\infty (x_1, x_2)
  = P^\alpha\left(  \frac{1}{2} (x_1+{}^{\rm t}x_1),    \frac{1}{2} (x_2+{}^{\rm t}x_2) \right) e^{-\pi {\rm Tr}( x_1  {}^{\rm t}\overline{ x_1} + x_2  {}^{\rm t}\overline{ x_2} ) }.    
\end{align*}
We also define $\varphi_\infty \in {\mathcal S}({\mathbf X}_\infty)\otimes {\mathbf C}[X,Y]_{2n+2} \otimes {\mathbf C}[X,Y]_n$ by 
\begin{align}\label{def:phiinf}
   \varphi_\infty(x) = \sum^n_{\alpha=0} \varphi^\alpha_\infty(x) \otimes \binom{n}{\alpha}X^\alpha Y^{n-\alpha}.    
\end{align}
Then, fix the archimedean test function $\widetilde{\varphi}_\infty$ as follows:
\begin{align*}
  \widetilde{\varphi}_\infty  = (\varphi_\infty,  (-1)^{n+1} \delta(\infty)\tau_{2n+2}(w_0)\varphi_\infty  )
      &\in {\mathcal S}({\mathbf X}_\infty) \otimes \widetilde{\mathcal W}_{2n+2}({\mathbf C})^\vee \otimes {\mathcal L}_\lambda({\mathbf C}).  
\end{align*}

The following lemma determines the minimal $K$-type of HST lifts, which is compatible with \cite[Lemma 12]{hst93}.

\begin{lem}\label{HSTwt}
{\itshape
We embed ${\rm U}_2({\mathbf R}) \hookrightarrow {\rm Sp}_4({\mathbf R}) $ by $A+B\sqrt{-1} \mapsto \begin{pmatrix} A & B \\ -B & A \end{pmatrix}$.
Then, we have the following statements:
\begin{enumerate}
\item\label{HSTwt(i)} For $(u,g) \in {\rm SU}_2({\mathbf R}) \times {\rm U}_2({\mathbf R})$, we have
\begin{align*}
  \omega_\infty(u,g)\varphi_\infty = \rho_{(2n+2,0)}(u^{-1})\otimes \rho_{(n+2,2)}({}^{\rm t}g)\varphi_\infty.    
\end{align*}
\item\label{HSTwt(ii)} 
Assume that $n$ is even and $\delta(\infty)=-1$. 
We have 
 \begin{align*}
     \omega_\infty( {\mathbf t}, 1) \widetilde{\varphi}_\infty 
     =  \widetilde{\tau}_{2n+2}({\mathbf t}) \widetilde{\varphi}_\infty = \widetilde{\tau}_{2n+2}(w_0) \widetilde{\varphi}_\infty.
 \end{align*}
\end{enumerate}
}
\end{lem}
\begin{proof}
We prove the statement \ref{HSTwt(i)}. 
First, we prove that $\omega_\infty(u, 1) \varphi_\infty= \rho_{(2n+2,0)}(u^{-1})\varphi_\infty$ for $u\in {\rm SU}_2({\mathbf R})$.
This is enough to prove the following identity:
          \begin{align*}
               P^\alpha\left( u^{-1} x_1  (\bar{u}^\ast)^{-1}, 
                                     u^{-1} x_2 (\bar{u}^\ast)^{-1} ; (X, Y)\right)
             = P^\alpha\left( x_1, x_2 ; (X, Y)u^{-1} \right),    
          \end{align*}
where $(x_1,x_2) \in {\mathbf X}^0_\infty$.
Since $\bar{u}^\ast={}^{\rm t}u$ and $w^{-1}_0 u w_0 = \bar{u}$, 
this is proved by the straightforward computation.

We prove that $\omega_\infty(1, g) \varphi_\infty= \rho_{(n+2,2)}({}^{\rm t} g)\varphi_\infty$ for $g\in {\rm U}_2({\mathbf R})$.
We put $g=A+\sqrt{-1}B$, where $A, B\in {\rm M}_2({\mathbf R})$.   
We may assume that $\det B\neq 0$.
Then we find that
\begin{align*}
\begin{pmatrix}  A & B \\
                     -B & A  \end{pmatrix}
=\begin{pmatrix}  1_2 & -AB^{-1} \\
                           0_2 & 1_2  \end{pmatrix}
  \begin{pmatrix}  {}^{\rm t} B^{-1} & 0_2  \\
                           0_2 & B  \end{pmatrix}
  \begin{pmatrix}  0_2  & 1_2 \\
                     -1_2 & 0_2  \end{pmatrix}
  \begin{pmatrix}  1_2 & -B^{-1}A \\
                            0_2  & 1_2  \end{pmatrix}.
\end{align*}
Fix an isomorphism $\iota:{\rm M}_{2,3}({\mathbf R}) \to {\mathbf X}^0_\infty $ by 
\begin{align*}
    \begin{pmatrix} t_1 & x_1  & y_1   \\
                            t_2 & x_2 & y_2  \end{pmatrix}
  \mapsto 
    \left(  \begin{pmatrix}  z_1 & \sqrt{-1} t_1 \\  \sqrt{-1} t_1 &   \bar{z}_1   \end{pmatrix}, \begin{pmatrix}  z_2 & \sqrt{-1}t_2 \\ \sqrt{-1}t_2 & \bar{z}_2   \end{pmatrix}      \right),     
       \quad (z_i:=x_i+\sqrt{-1}y_i (i=1, 2 )).
\end{align*}
We also define $\iota^\prime(s_1,s_2)$ for $(s_1,s_2)\in {\mathbf R}^{\oplus 2}$ to be   
\begin{align*}
    \iota^\prime(s_1,s_2)
   =  \left( \begin{pmatrix} 0 & \sqrt{-1} s_1 \\ -\sqrt{-1} s_1 & 0   \end{pmatrix}, 
                         \begin{pmatrix} 0 & \sqrt{-1} s_2 \\ -\sqrt{-1} s_2 & 0   \end{pmatrix}   \right).
\end{align*}
Then we write each $\widetilde{X}, \widetilde{Y} \in {\mathbf X}_\infty$ as follows: 
\begin{align*}
    \widetilde{X} = \iota^\prime(\xi_1,\xi_2)
                      + \iota(X),  \quad
    \widetilde{Y} 
    = ( \widetilde{Y}_1,  \widetilde{Y}_2) 
    = \iota^\prime(\eta_1,\eta_2)
                      + \iota(Y),                    
                       \quad (\xi_1, \xi_2, \eta_1, \eta_2  \in {\mathbf R}, X, Y\in {\rm M}_{2,3}({\mathbf R}) ).
\end{align*}
Let ${\rm d}\mu(y)$ be the self-dual measure on ${\mathbf X}_\infty$ with respect to 
the Fourier transform:
\begin{align*}
   \widehat{\varphi}(x) =   \int_{{\mathbf X}_\infty}  \varphi(y) \psi_{{\mathbf Q}, \infty}((x,y))  {\rm d}\mu(y),
     \quad (\varphi \in {\mathcal S}({\mathbf X}_\infty)).  
\end{align*}
By using the definition of the bilinear form $(\cdot, \cdot)$ on $V^{\otimes 2}$ (see also (\ref{eq:bitr})), we have 
\begin{align*}
    (\widetilde{X}, \widetilde{Y}) 
      =& {\rm Tr}(\widetilde{X}_1\widetilde{Y}^\ast_1 + \widetilde{X}_2\widetilde{Y}^\ast_2) 
     =    - 2 (\xi_1\eta_1+ \xi_2\eta_2) + 2{\rm Tr}({}^{\rm t}XY).   
\end{align*}
Hence the self-dual measure on ${\rm M}_{2,3}({\mathbf R})$ (resp. ${\mathbf R}^{\oplus 2}$) 
    which is induced by ${\rm d}\mu(y)$  via $\iota$   
    is given by $2^3 {\rm d}Y$ (resp. $2{\rm d}\eta$) for $Y\in {\rm M}_{2,3}({\mathbf R})$ (resp. $\eta \in {\mathbf R}^{\oplus 2}$). 

Recall that  $\gamma^2_{V_\infty}$ is $-1$  (\cite[Proposition A.10]{ra93}, \cite[Lemma A.1]{ic05})  and that 
\begin{align*} 
       {\rm Tr}(   S_{ \widetilde{Y} } T ) 
        =&  - \begin{pmatrix} \eta_1 & \eta_2 \end{pmatrix}  T  \begin{pmatrix} \eta_1 \\ \eta_2 \end{pmatrix}  
                 + {\rm Tr}({}^{\rm t}Y T Y),  
            \quad (T\in {\rm M}_2({\mathbf C})).  
\end{align*}
 Then, the definition of the Schr\"{o}dinger model (\ref{eq:weilsch}) of the Weil representation $\omega_V$ shows that 
\begin{align*}
 & \omega_V \left( \begin{pmatrix}  0_2 & 1_2 \\
                     -1_2 & 0_2  \end{pmatrix}
  \begin{pmatrix}  1_2 & -B^{-1}A \\
                     0_2 & 1_2  \end{pmatrix}  \right)\varphi^\alpha_\infty( \widetilde{X} )   \\  \displaybreak[0]
= & (-1) \times \int_{{\mathbf X}_\infty} 
      \psi_{{\mathbf Q}, \infty}( {\rm Tr}(S_{\widetilde{Y}} (-B^{-1}A) ) )
      P^\alpha( \frac{1}{2} ( \widetilde{Y}   + {}^{\rm t}\widetilde{Y}   ) ) 
            e^{-\pi{\rm Tr}(  \widetilde{Y}_1 {}^{\rm t}\overline{ \widetilde{Y}_1} + \widetilde{Y}_2 {}^{\rm t}\overline{ \widetilde{Y}_2 }  ) }
      \psi_{{\mathbf Q}, \infty}(   ( \widetilde{X},  \widetilde{Y} )   )
       {\rm d}\widetilde{Y} \\    \displaybreak[0]
= & (-1) \times  \int_{{\rm M}_{2,3}({\mathbf R})} 
          e^{ 2\pi\sqrt{-1}{\rm Tr}(  {}^{\rm t}Y  (-B^{-1}A )Y )  }
          P^\alpha(\iota(Y)) 
      e^{ -2\pi {\rm Tr}( {}^{\rm t}Y Y )  }  
      e^{ 4\pi\sqrt{-1} {\rm Tr}( {}^{\rm t}X Y )  }  \times 2^3 {\rm d}Y  \\    \displaybreak[0]
   & \times \int_{{\mathbf R}^2}  
       e^{ - 2\pi\sqrt{-1} (\eta_1,\eta_2) (-B^{-1}A)  \left( \begin{smallmatrix} \eta_1 \\ \eta_2 \end{smallmatrix} \right)   } 
       e^{-2\pi(  \eta^2_1 + \eta^2_2  )} 
        e^{  -4\pi\sqrt{-1} (\xi_1\eta_1+\xi_2\eta_2)  }  \times 2 {\rm d}\eta  \\   \displaybreak[0]
= & (-1) \times  \int_{{\rm M}_{2,3}({\mathbf R})} 
          e^{ 4\pi\sqrt{-1} {\rm Tr}({}^{\rm t}X Y )  }   
          e^{2\pi\sqrt{-1}{\rm Tr}(  {}^{\rm t}Y  ( \sqrt{-1} 1_2 - B^{-1}A )Y )  }   
          P^\alpha(\iota(Y)) \times 2^3 {\rm d}Y  \\    \displaybreak[0]
   & \times \int_{{\mathbf R}^2}  
        e^{-4\pi\sqrt{-1} (\xi_1\eta_1+\xi_2\eta_2)  } 
               e^{2\pi\sqrt{-1} (\eta_1,\eta_2) ( \sqrt{-1} 1_2 + B^{-1}A)  \left( \begin{smallmatrix} \eta_1 \\ \eta_2 \end{smallmatrix} \right)   }   \times 2 {\rm d}y   \\    \displaybreak[0]
= & (-1) \times  \int_{{\rm M}_{2,3}({\mathbf R})} 
          e^{ \sqrt{-1} {\rm Tr}({}^{\rm t}(4\pi X) Y )  }   
          e^{ \frac{\sqrt{-1}}{2} {\rm Tr}(  {}^{\rm t}Y  ( 4\pi( \sqrt{-1} 1_2 - B^{-1}A ) )Y )  }   
          P^\alpha(\iota(Y))  \times 2^3 {\rm d}Y  \\   \displaybreak[0]
   & \times \int_{{\mathbf R}^2}  
        e^{\sqrt{-1} ( -4\pi( \xi_1\eta_1+\xi_2\eta_2) )   } 
               e^{ \frac{\pi\sqrt{-1}}{2} (\eta_1,\eta_2) (4\pi ( \sqrt{-1} 1_2 + B^{-1}A) )  \left( \begin{smallmatrix} \eta_1 \\ \eta_2 \end{smallmatrix} \right)   }  \times 2 {\rm d}\eta. 
\end{align*}
Since $P^\alpha$ is pluri-harmonic by  Lemma \ref{pluri}, \cite[Lemma 4.5]{kv78} yields that 
the above identity is equal to 
\begin{align*}
 & (-1) \times   (2\pi)^\frac{2\cdot 3}{2} \left( \det \frac{   4\pi( \sqrt{-1} 1_2 - B^{-1}A )   }{\sqrt{-1}} \right)^{-\frac{3}{2}}  
        e^{  \frac{\sqrt{-1}}{2} {\rm Tr}(  {}^{\rm t} (4\pi X)(     -    \frac{1}{4\pi} ( \sqrt{-1} 1_2 -  B^{-1}A )^{-1}         ) (4\pi X) )   }  \times 2^3   \\
   & \quad \quad \times         P^\alpha(\iota(    -    \frac{1}{4\pi} ( \sqrt{-1} 1_2 -  B^{-1}A )^{-1} \times  4\pi X   ))   \\
   & \times (2\pi)^\frac{1\cdot 2}{2} 
         \left( \det \frac{  4\pi ( \sqrt{-1} 1_2  +  B^{-1}A   }{\sqrt{-1}} \right)^{-\frac{1}{2}}  
          e^{   \frac{\sqrt{-1}}{2}   (4\pi \xi_1, 4\pi \xi_2)( -  \frac{1}{4\pi}(\sqrt{-1} 1_2 + B^{-1}A)^{-1}    ) \left( \begin{smallmatrix} 4\pi \xi_1 \\ 4\pi \xi_2 \end{smallmatrix} \right)    } \times 2  \\
 =& (-1) \times     \left( \det (  1_2  +  \sqrt{-1}  B^{-1}A )  \right)^{-\frac{3}{2}}  
        e^{  -2\pi \sqrt{-1} {\rm Tr}(  {}^{\rm t} X      ( \sqrt{-1} 1_2 -  B^{-1}A )^{-1}          X )   }    \\
   & \quad \quad \times         P^\alpha(\iota(   -    ( \sqrt{-1} 1_2 -  B^{-1}A )^{-1}   X)   )   \\
   & \times  
         \left( \det  ( 1_2 - \sqrt{-1} B^{-1}A   ) \right)^{-\frac{1}{2}}  
          e^{   - 2\pi\sqrt{-1}  ( \xi_1,  \xi_2) (\sqrt{-1}  1_2 + B^{-1}A)^{-1}     \left( \begin{smallmatrix}  \xi_1 \\  \xi_2 \end{smallmatrix} \right)    }.
\end{align*}
Here we take a blanch of $\det(1_n+\sqrt{-1}S)^{\frac{m}{2}}$ for $S={}^{\rm t} S\in {\rm M}_n({\mathbf R})$ and $m\in {\mathbf Z}$ as follows.
For $z=re^{\sqrt{-1}\theta}\in {\mathbf C}$ for $r>0$ and $-\pi<\theta\leq \pi$, we define $ z^{\frac{1}{2}}  = r^{\frac{1}{2}}  e^{\sqrt{-1}\frac{\theta}{2}}$.
Since $S$ is a symmetric matrix, we find a $T\in {\rm O}_n({\mathbf R})$ such that
\begin{align*}
    {}^{\rm t} T(1_n+\sqrt{-1}S)T = 1_n + \sqrt{-1}\begin{pmatrix} d_1 && \\ & \ddots & \\ & & d_n  \end{pmatrix},
\end{align*}
where $d_1,\ldots, d_n \in {\mathbf R}$.
Define
\begin{align*}
   \det(1_n+\sqrt{-1}S)^{\frac{m}{2}}   = \left( |\det T|^{-1}\prod^n_{i=1}(1+\sqrt{-1}d_i)^{\frac{1}{2}} \right)^m .
\end{align*}
Summarizing the above computation, we obtain
\begin{align*}
 & \omega_V \left( \begin{pmatrix}  0_2 & 1_2 \\
                     -1_2 & 0_2  \end{pmatrix}
  \begin{pmatrix}  1_2 & -B^{-1}A \\
                           0_2 & 1_2  \end{pmatrix}  \right)\varphi^\alpha_\infty(\widetilde{X}) \\
= & (-1) \times  \det ( 1_2 + \sqrt{-1}B^{-1}A)^{-\frac{3}{2}}  
                           e^{  -2\pi \sqrt{-1} {\rm Tr}(  {}^{\rm t} X      ( \sqrt{-1} {\bf 1}_2 -  B^{-1}A )^{-1}          X )   }  
             \times     P^\alpha(  \iota(  -    ( \sqrt{-1} 1_2 -  B^{-1}A )^{-1})  X  )  )    \\
   & \times  \det ( 1_2 -  \sqrt{-1}B^{-1}A)^{-\frac{1}{2}}  
                   e^{   - 2\pi\sqrt{-1}  ( \xi_1,  \xi_2) (\sqrt{-1} {\bf 1}_2 + B^{-1}A)^{-1}     \left( \begin{smallmatrix}  \xi_1 \\  \xi_2 \end{smallmatrix} \right)    } .
\end{align*}
Note that
\begin{align*}
    &  \det ( 1_2 + \sqrt{-1}B^{-1}A)^{-\frac{3}{2}}  \det ( 1_2 - \sqrt{-1}B^{-1}A)^{-\frac{1}{2}} \\
 = &  \det ( 1_2  +  \sqrt{-1}B^{-1}A)^{-1}   (  \det ( 1_2 +  \sqrt{-1}B^{-1}A)^{-\frac{1}{2}}   \det ( 1_2  -  \sqrt{-1}B^{-1}A))^{-\frac{1}{2}} \\
= & \det(\sqrt{-1}B^{-1})^{-1} \det(A-\sqrt{-1}B)^{-1}   \\
   & \quad      \times \left\{   \det(  \sqrt{-1}B^{-1}) \det( - \sqrt{-1}B^{-1})   
                                \det (A-\sqrt{-1}B)  \det (A+\sqrt{-1}B)      \right\}^{-\frac{1}{2}}    \\
= & -\det B  \det g \times |\det B|.
\end{align*}
We also find that  
\begin{align*}
    P^\alpha(  \iota(  -    ( \sqrt{-1} 1_2 -  B^{-1}A )^{-1}  X)  )  
    =&    P^\alpha(  - \iota(   X  )    {}^{\rm t}( \sqrt{-1} 1_2 -  B^{-1}A )^{-1}   )      
   =  P^\alpha(      \iota( X  )    {}^{\rm t}  B   g ) . 
\end{align*}
By using (\ref{eq:weilsch}) again,  $\omega_V(g)\varphi^\alpha_\infty$ is given as follows:
\begin{align*}
       \omega_V\left( \begin{pmatrix}  A & B \\ -B & A  \end{pmatrix}  \right)\varphi^\alpha_\infty(\widetilde{X})  
  = &  \psi_{{\mathbf Q}, \infty} ( S_{\widetilde{X}} (-AB^{-1})) \times \chi_{V_\infty}(\det B) |\det B^{-1}|^{\frac{4}{2}} \\ 
     &    \times (-1) \times  (-\det B |\det B| \det g) \\
     &  \times    e^{  -2\pi \sqrt{-1} {\rm Tr}(  {}^{\rm t} (B^{-1}X)      ( \sqrt{-1} {\bf 1}_2 -  B^{-1}A )^{-1}          B^{-1}X )   }  
                     P^\alpha( \iota(X) {}^{\rm t} B^{-1} {}^{\rm t} B g  )    \\
      & \times  e^{   - 2\pi\sqrt{-1}  ( \xi_1,  \xi_2) {}^{\rm t}B^{-1} ( (\sqrt{-1} {\bf 1}_2 + B^{-1}A)^{-1}    )  B^{-1} \left( \begin{smallmatrix}  \xi_1 \\  \xi_2 \end{smallmatrix} \right)   }.
\end{align*}
In the straight forward way, we find that
\begin{align*}
   &   e^{  -2\pi \sqrt{-1} {\rm Tr}(  {}^{\rm t} (B^{-1}X)      ( \sqrt{-1} 1_2 -  B^{-1}A )^{-1}          B^{-1}X )   }  
       \times  e^{   - 2\pi\sqrt{-1}  ( \xi_1,  \xi_2) {}^{\rm t}B^{-1} ( (\sqrt{-1} 1_2 + B^{-1}A)^{-1}    )  B^{-1} \left( \begin{smallmatrix}  \xi_1 \\  \xi_2 \end{smallmatrix} \right)   }   \\
  =&  e^{-\pi{\rm Tr}(\widetilde{X}_1  {}^{\rm t}\overline{\widetilde{X}_1}+\widetilde{X}_2 {}^{\rm t}\overline{\widetilde{X}_2} )}
         \times e^{2\pi\sqrt{-1} {\rm Tr}(S_{\widetilde{X}}( {}^{\rm t}B^{-1} {}^{\rm t}A) )}. 
\end{align*}
Since $AB^{-1} = {}^{\rm t}B^{-1} {}^{\rm t}A$, we obtain 
\begin{align*}
       \omega_V\left( \begin{pmatrix}  A & B \\ -B & A  \end{pmatrix}  \right)\varphi^\alpha_\infty(\widetilde{X})  
  =   \det g P^\alpha(\iota(X) g)  
       e^{-\pi{\rm Tr}(\widetilde{X}_1  {}^{\rm t}\overline{\widetilde{X}_1}+\widetilde{X}_2 {}^{\rm t}\overline{\widetilde{X}_2} )}.
\end{align*}
Define ${\mathbf q}(x_1, x_2), {\mathbf p}^\prime(x) (x_1, x_2, x\in V)$ to be 
\begin{align*}
       {\mathbf q}(x_1, x_2) =& {\mathbf p}\left( \frac{1}{8}\left(        (x_1+{}^{\rm t}x_1) w_0 (x_2+{}^{\rm t}x_2)  
                                                                                               + {}^{\rm t}(  (x_1+{}^{\rm t}x_1) w_0 (x_2+{}^{\rm t}x_2) )  
                                                                                         \right)  
                                                                         \right),   \\
       {\mathbf p}^\prime(x)  =& {\mathbf p}\left(\frac{1}{2}(x+{}^{\rm t}x) \right).
\end{align*}
Note that $\varphi^\alpha_\infty(x_1,x_2)
                                      = {\mathbf q}(x_1, x_2) {\mathbf p}^\prime(x_1)^\alpha {\mathbf p}^\prime(x_2)^{n-\alpha} e^{-\pi {\rm Tr}(x_1 {}^{\rm t}\overline{x_1} + z_2 {}^{\rm t}\overline{x_2})}$.
Write $g = \begin{pmatrix} a & b \\ c & d \end{pmatrix} \in {\rm U}_2({\mathbf R})$.   
Then we find that 
\begin{align*}
    {\mathbf q}(    a x_1 + c x_2  ,  b x_1 + d x_2 ) = \det g {\mathbf q}(x_1, x_2).  
\end{align*}
This proves that 
\begin{align*}
    P^\alpha( \widetilde{X} g) = \det g {\mathbf q}(\widetilde{X})
                                                (a{\mathbf p}^\prime(\widetilde{X}_1) + c{\mathbf p}^\prime(\widetilde{X}_1))^\alpha 
                                                (b{\mathbf p}^\prime(\widetilde{X}_1) + d{\mathbf p}^\prime(\widetilde{X}_2))^{n-\alpha},
\end{align*}
and hence we obtain 
\begin{align*}
    &   \omega_V \left( \begin{pmatrix}  A & B \\ -B & A  \end{pmatrix}  \right)\varphi^\alpha_\infty(\widetilde{X}) \\
 = &   \sum^n_{\alpha=0}   \det g^2 {\mathbf q}(\widetilde{X})  
                                   (a{\mathbf p}^\prime(\widetilde{X}_1) + c{\mathbf p}^\prime(\widetilde{X}_2))^\alpha 
                                   (b{\mathbf p}^\prime(\widetilde{X}_1) + d{\mathbf p}^\prime(\widetilde{X}_2))^{n-\alpha}
          e^{-\pi{\rm Tr}(\widetilde{X}_1  {}^{\rm t}\overline{\widetilde{X}_1}+\widetilde{X}_2 {}^{\rm t}\overline{\widetilde{X}_2} )} \otimes
            \binom{n}{\alpha} X^\alpha Y^{n-\alpha}  \\
 = &   \sum^n_{\alpha=0}    P^\alpha(\widetilde{X}) 
         e^{-\pi{\rm Tr}(\widetilde{X}_1  {}^{\rm t}\overline{\widetilde{X}_1}+\widetilde{X}_2 {}^{\rm t}\overline{\widetilde{X}_2} )}
                                  \otimes     \binom{n}{\alpha}\cdot \rho_{(n+2, 2)}({}^{\rm t} g)(X^\alpha Y^{n-\alpha}). 
\end{align*}
This proves the first statement.

We prove the statement \ref{HSTwt(ii)}. 
Put
\begin{align*}
  x= \begin{pmatrix} z & \sqrt{-1}t \\ \sqrt{-1} t & \overline{z}  \end{pmatrix}, x_1, x_2 \in V_\infty \cap V^0; \quad  
  {\mathbf p}(x; (X,Y)) = z X^2 + 2 \sqrt{-1} t XY + \overline{z} Y^2.
\end{align*}
By the definition (\ref{def:t}) of ${\mathbf t} \in H({\mathbf R})$, we find that
\begin{align*}
  {\mathbf p}({\mathbf t}x; (X,Y))    
  =& -\overline{z} X^2 + 2 \sqrt{-1} t XY -z Y^2    
  = - {\mathbf p}( x ; (X,Y) w_0), \\
     {\mathbf p}\left(  \frac{1}{2} ( {\mathbf t}x_1w_0{\mathbf t}x_2   +{}^{\rm t}( {\mathbf t}x_1w_0{\mathbf t}x_2  )     ); (X,Y) \right)    
  =& {\mathbf p}\left(  \frac{1}{2} ( x_1w_0 x_2   +{}^{\rm t}( x_1w_0x_2  )     ); (X,Y)w_0 \right). 
\end{align*}
These identities and the definition of $P^\alpha$ show that
\begin{align*}
  \omega_\infty({\mathbf t}, 1)\varphi_\infty(x)
  = (-1)^{n} \tau_{2n+2}(w_0)\varphi_\infty(x). 
\end{align*}
Since $n$ is even and $\delta(\infty)=-1$, the definition of $\widetilde{\tau}_{2n+2}$ which is introduced in Section \ref{sec:minK} yields that 
\begin{align*}
  \omega_\infty({\mathbf t}, 1) \widetilde{\varphi}_\infty(x)
  =&  (  (-1)^{n} \tau_{2n+2}(w_0)\varphi_\infty(x),  (-1)\delta(\infty) \varphi_\infty(x)  )    \\
  =& ( (-1)^{n+1} \delta(\infty) \tau_{2n+2}(w_0)\varphi_\infty(x), \varphi_\infty(x) )   \\
  =&  \widetilde{\tau}_{2n+2}(w_0)\widetilde{\varphi}_\infty(x). 
\end{align*}
This shows the second statement. 
\end{proof}

Hereafter, we always fix the Bruhat-Schwartz function 
   $\widetilde{\varphi} = \widetilde{\varphi}_\infty\otimes_{v<\infty}\varphi_v 
       \in {\mathcal S}({\mathbf X}_{\mathbf A})\otimes \widetilde{\mathcal W}_{2n+2}({\mathbf C})^\vee \otimes {\mathcal L}_\lambda({\mathbf C})$, 
where $\widetilde{\varphi}_\infty$ and $\varphi_v$ are introduced as above. 
For later use, we put $\widetilde{\varphi}_v = \varphi_v$ for each finite $v\in \Sigma_{\mathbf Q}$.
Then, by Lemma \ref{HSTlev}, Lemma \ref{HSTwt} and \cite[Proposition 3]{hst93}, 
     we see that $\theta(\widetilde{\varphi}, \widetilde{\mathbf f}) :{\rm GSp}_4({\mathbf A}) \to {\mathcal L}_\lambda({\mathbf C}) $
     is a Siegel cusp form of weight $\lambda=(n+2,2)$ and of level $\Gamma_0(N_F)$. 
Denote by $\Pi$ the automorphic representation of ${\rm GSp}_4({\mathbf A})$ which is generated by $\theta(\widetilde{\varphi}, \widetilde{\mathbf f})$.   
Let $L(s, \Pi, {\rm spin}) = \prod_{v\in \Sigma_{\mathbf Q} } L(s, \Pi_v, {\rm spin})$ be the spinor $L$-function of $\Pi$. 
Then \cite[Lemma 10, Lemma 11]{hst93} yield that for each $v\in \Sigma_{\mathbf Q}\backslash \Sigma_{\rm ram}$,  
\begin{align*}
   L(s, \Pi_v, {\rm spin} ) = L(s,\pi_v). 
\end{align*}

\subsection{Fourier expansion of theta series}\label{secFC}

In this section, we give a description (\ref{thetaFC}) of adelic Fourier coefficients of HST lifts $\theta(\widetilde{\varphi}, \widetilde{\mathbf f})$
  in terms of the initial automorphic form $\widetilde{\mathbf f}$ and the Bruhat-Schwartz function $\widetilde{\varphi}$.  
For the simplicity, we write $\theta = \theta(\widetilde{\varphi}, \widetilde{\mathbf f})$ in this subsection.  

We write $S=S_x$ for some $x\in V({\mathbf Q})^{\oplus 2}$. 
For the unipotent subgroup $U({\mathbf A})$ of ${\rm Sp}({\mathbf A})$, 
let $\psi_S:U({\mathbf A})\to {\mathbf C}^\times$ be the character which is introduced in Section \ref{sec:autoSp4}.
Recall that  the adelic $S$-th Fourier coefficient  ${\mathbf W}_{\theta,S}(g)$ for $g\in {\rm GSp}^+_4({\mathbf A})$ of $\theta$ is defined to be
\begin{align}\label{def:FWc}
   {\mathbf W}_{\theta, S}(g) = \int_{[U]} \theta(ug) \psi_S(u) {\rm d}u,  
\end{align}
where ${\rm d}u$ is the Haar measure which is normalized so that ${\rm vol}(U({\mathbf Q}) \backslash U({\mathbf A}), {\rm d}u)=1$.  

For $g\in {\rm GSp}^+_4({\mathbf A})$, we take $h_1\in {\rm GO}(V)({\mathbf A})$ such that $\nu(g)=\nu(h_1)$. 
Then, we find that 
\begin{align}\label{wittFC}
\begin{aligned}[c]
   {\mathbf W}_{\theta, S}(g) 
   =& \int_{[U]} {\rm d}u \int_{[H_1 ]} {\rm d}h 
         \langle \theta(hh_1, ug), \widetilde{\mathbf f}(hh_1)  \rangle_{\widetilde{\mathcal W}} \psi_S(u)   \\
   =& \int_{[U]} {\rm d}u \int_{[H_1 ]} {\rm d}h
        \sum_{x\in{\mathbf X} } 
         \langle \omega(1, u) \omega(hh_1, g)\widetilde{\varphi}(x), \widetilde{\mathbf f}(hh_1)  \rangle_{\widetilde{\mathcal W}} \psi_S(u)   \\
   =&  \int_{[H_1 ]} {\rm d}h
        \sum_{ \substack{  x\in{\mathbf X}  \\   S=S_x } } 
         \langle \omega_V(g_1) \widetilde{\varphi} (h^{-1}_1h^{-1}x), \widetilde{\mathbf f}(hh_1)  \rangle_{\widetilde{\mathcal W}},   
\end{aligned}
\end{align}
where we put $g_1 = \begin{pmatrix} 1_2 & 0_2 \\  0_2 & \nu(g)^{-1} 1_2 \end{pmatrix} g$.   
For $x\in {\mathbf X} = V^{\oplus 2} $, define 
\begin{align*}
  H_x = \left\{  h\in H_1  \  | \  h\cdot x = x \right\}. 
\end{align*}
Then, Witt's theorem shows that
\begin{align} \label{thetaFC}   
\begin{aligned}[c]
   {\mathbf W}_{\theta, S}(g) 
   =&   \int_{[H_1]} {\rm d}h
        \sum_{ h^\prime \in H_x({\mathbf Q}) \backslash H_1({\mathbf Q}) } 
         \langle \omega_V(g_1) \widetilde{\varphi} (h^{-1}_1h^{-1} (h^\prime)^{-1} x), \widetilde{\mathbf f}( h^\prime hh_1)  \rangle_{\widetilde{\mathcal W}}   \\     
   =&   \int_{ H_x({\mathbf Q})  \backslash  H_1({\mathbf A}) } {\rm d}h
         \langle \omega_V(g_1) \widetilde{\varphi} (h^{-1}_1h^{-1}  x), \widetilde{\mathbf f}(  hh_1)  \rangle_{\widetilde{\mathcal W}}   \\  
   =&   \int_{ H_x({\mathbf A})  \backslash  H_1({\mathbf A}) } {\rm d}h
           \int_{ [H_x] } {\rm d}s
         \langle \omega_V(g_1) \widetilde{\varphi} (h^{-1}_1h^{-1}  s^{-1} x), \widetilde{\mathbf f}(  shh_1)  \rangle_{\widetilde{\mathcal W}}   \\    
   =&   \int_{ H_x({\mathbf A})  \backslash  H_1({\mathbf A}) } {\rm d}h
           \int_{ [H_x] } {\rm d}s
         \langle \omega_V(g_1) \widetilde{\varphi}(h^{-1}_1h^{-1}   x), \widetilde{\mathbf f}(  shh_1)  \rangle_{\widetilde{\mathcal W}}.   
\end{aligned}
\end{align}

\section{Inner product formula}\label{s:InnPrd}

In this section, we prove an explicit inner product formula for HST lifts. 
See Theorem \ref{thm:alginnprd} for the result. 
In Section \ref{pairH}, we  introduce a local Hermitian pairing for spaces of automorphic representations of ${\rm GO}_{3,1}({\mathbf Q}_v)$ for each $v\in \Sigma_{\mathbf Q}$.   
By using these local Hemitian pairings, 
 we decompose a global Hermitian pairing for automorphic forms on ${\rm GO}_{3,1}({\mathbf A})$ to a product of local pairings.       
To fix a decomposition of a global pairing is necessary to apply the Rallis inner product formula for HST lifts.  
In Section \ref{sec:Rinn}, we reduce a computation of an inner product of HST lifts to a computation of local integrals on ${\rm GSO}_{3,1}({\mathbf Q}_v)$    
  and we give an explicit inner product formula in Theorem \ref{HSTInnprd}  by using adelic language.  
The inner product formula in classical language 
  is given in Theorem \ref{thm:alginnprd} in Section \ref{sec:clainnprd}.

\subsection{Hermitian pairing}\label{pairH}

We define Hermitian pairings on the several representation spaces of automorphic representations of orthogonal groups.  
These pairings will be used in Section \ref{sec:Rinn}  to introduce the Rallis inner product formula. 
In this subsection, we use the same notation in Section \ref{secrepGO} and \ref{secautoGO}. 

Recall that ${\mathcal U}_{\mathfrak N}$ is a subgroup of $H^0( {\mathbf A}_{\rm fin} )$ which is defined in Section \ref{secautoGO}. 
For each finite place $v$ of ${\mathbf Q}$, the ${\mathcal U}_{\mathfrak N, v}$-invariant subspace $\sigma^{{\mathcal U}_{\mathfrak N, v}}_v$ of $\sigma_v$     
    is one dimensional by the theory of newform. 
We fix a non-zero element $f^0_v$ in $\sigma^{{\mathcal U}_{\mathfrak N, v}}_v$. 
We may assume that 
\begin{align*}
   \xi_v  f^0_v  = f^0_v
\end{align*}   
 for each finite place $v\in {\mathfrak S}$. 
Fix an ${\rm SU}_2({\mathbf R})$-equivariant map $\iota_\infty: {\mathcal W}_{2n+2}({\mathbf C}) \to {\mathcal V}_\sigma$, which is unique up to constant.
For each $u\in {\mathcal W}_{2n+2}({\mathbf C})$, define 
\begin{align}
    f^0_{\infty, u} = \iota_\infty( u  ) \in {\mathcal V}_\infty.   \label{eq:InfModel}
\end{align} 
Then we may fix an isomorphism $j: \otimes^\prime_v \sigma_v \to \sigma$ so that 
\begin{align}
   j(  f^0_{\infty, u} \otimes^\prime_{v<\infty} f^0_v) = \langle {\mathbf f},  u \rangle_{\mathcal W}. 
  \label{LocGrobIsoGSO}
\end{align}
We also define
\begin{align*}
   \widetilde{f}^{0}_{\infty,u} =  (f^0_{\infty, u},  \delta(\infty) \xi_\infty f^0_{\infty, u}) \in \widetilde{\mathcal V}_\infty; \quad 
   \widetilde{f}^{0}_{v}
        =   (f^{0}_{v}, f^{0}_{v})  
                                      \in \widetilde{\mathcal V}_v, \quad (v<\infty).
\end{align*}
Write $\widetilde{u} = (u, (-1)^{n+1} \delta(\infty) \tau_{2n+2}(w_0)u)\in \widetilde{\mathcal W}_{2n+2}({\mathbf C})$   
 for each $u\in {\mathcal W}_{2n+2}({\mathbf C})$.
Then, we fix an isomorphism $\widetilde{j}: \otimes^\prime_v \widetilde{\sigma}_v \to \widetilde{\sigma}$ so that 
\begin{align}
  \widetilde{j}(  \widetilde{f}^0_{\infty, u} \otimes^\prime_{v<\infty} \widetilde{f}^0_v) = \langle \widetilde{\mathbf f}, \widetilde{u} \rangle_{\widetilde{\mathcal W}}.
  \label{LocGrobIsoGO}
\end{align}

Let ${\mathcal B}_{\mathcal W}: {\mathcal W}_{2n+2}({\mathbf C})^{\otimes 2} \to {\mathbf C}$  
be the ${\rm SU}_2({\mathbf R})$-equivariant  Hermitian pairing which is given by 
\begin{align*}
   {\mathcal B}_{\mathcal W}(u_1, u_2) 
      = \langle u_1, \tau_{2n+2}(w_0) \cdot \bar{u}_2  \rangle_{\mathcal W},   
      \quad \left(w_0=\begin{pmatrix} 0 & 1 \\ -1 & 0 \end{pmatrix},  
                  u_1, u_2 \in {\mathcal W}_{2n+2}({\mathbf C}) \right).    
\end{align*}
Define ${\mathcal B}_{\sigma_\infty}$ to be the Hermitian pairing ${\mathcal V}^{\otimes 2}_\infty \to {\mathbf C} $   
such that  
\begin{align}
    {\mathcal B}_{\sigma_\infty}( f^0_{\infty, u_1}, f^0_{\infty, u_2})  
      = {\mathcal B}_{\mathcal W}(u_1, u_2)    \label{eq:BInf}
\end{align}
for each $u_1, u_2 \in {\mathcal W}_{2n+2}({\mathbf C})$.   
See Lemma \ref{l:pairB0} for the explicit form of the pairing ${\mathcal B}_{\sigma_\infty}$.  
For each finite place $v$ of ${\mathbf Q}$, 
we define the Hermitian pairing ${\mathcal B}_{\sigma_v}:{\mathcal V}^{\otimes 2}_v  \to {\mathbf C}$ 
so that ${\mathcal B}_{\sigma_v}(f^0_v, f^0_v)=1$.

For each place $v$ of ${\mathbf Q}$, the Hemitian pairing ${\mathcal B}_{\sigma_v}$ is extended to the pair on $\widehat{\sigma}_v$ and $\widehat{\sigma}^\pm_v$  as follows.
Let ${\mathcal B}_{\sigma^\sharp_v}: {\mathcal V}^{\sharp \otimes 2}_v \to {\mathbf C}$    
be the Hermitian pairing which is defined to be
\begin{align*}
   {\mathcal B}_{\sigma^\sharp_v}((f_1,f^\prime_1), (f_2,f^\prime_2))
   = \frac{1}{2} \left(  {\mathcal B}_{\sigma_v}(f_1, f_2) + {\mathcal B}_{\sigma_v}(f^\prime_1, f^\prime_2) 
                                 \right).
\end{align*}
Define ${\mathcal B}_{\widehat{\sigma}^\pm_v}={\mathcal B}_{\sigma^\sharp_v}|_{{\mathcal V}^\pm_v}$ if $v\in {\mathfrak S}$.   
We also define
\begin{align*}
 {\mathcal B}_{\widetilde{\sigma}_\infty} = {\mathcal B}_{\widehat{\sigma}^-_\infty}; \quad   
 {\mathcal B}_{\widetilde{\sigma}_v} = \begin{cases} {\mathcal B}_{\widehat{\sigma}^+_v},   &  (v\in {\mathfrak S}, v < \infty),  \\    
                                                                                    {\mathcal B}_{\widehat{\sigma}^\sharp_v},  &   (v\not\in {\mathfrak S}).  
                                                            \end{cases}       
\end{align*}

The following lemma is not used in Section \ref{s:InnPrd}. 
However, this lemma verifies the existence of the pairing ${\mathcal B}_{\sigma_\infty}$ in an explicit way 
and this is necessary for the explicit computation of the inner product of HST lifts in Section \ref{s:ArchInt}. 

\begin{lem}\label{l:pairB0}
{\itshape 
Let $W_{\sigma, \infty}$ be the element in Lemma \ref{l:whitt}.     
Put $W^j_{\sigma, \infty} =\langle  W_{\sigma, \infty}, X^{n+1+j} Y^{n+1-j}\rangle_{\mathcal W}$
          and 
          \begin{align*}
               C_\infty = \frac{  2^{-4} {\rm dim}{\mathcal W}_{2n+2}({\mathbf C})^2 \binom{2n+2}{n+1}  }{\Gamma_{\mathbf C}(2n+4)},  
               \quad (\Gamma_{\mathbf C}(s) := 2(2\pi)^{-s} \Gamma(s)). 
          \end{align*} 
        Define ${\mathcal B}^0_{\sigma_\infty}: {\mathcal V}^{\otimes 2}_\infty  \to {\mathbf C}$ by 
        \begin{align*}
                 {\mathcal B}^0_{\sigma_\infty}(W^i_{\sigma, \infty}, W^j_{\sigma, \infty})    
           = C_\infty \times 
                  \int_{{\rm SU}_2({\mathbf R})} 
                   \int^\infty_0
                      W^i_{\sigma, \infty}\left( \begin{pmatrix} t &  0 \\ 0 & 1 \end{pmatrix} u \right)  
                              \overline{  W^j_{\sigma, \infty}\left( \begin{pmatrix} t &  0 \\ 0 & 1 \end{pmatrix} u  \right)  }  
                              \frac{{\rm d}t }{t}
                              {\rm d}u, 
        \end{align*}
        where ${\rm d}u$ is  the Haar measure on ${\rm SU}_2({\mathbf R})$ such that ${\rm vol}({\rm SU}_2({\mathbf R}), {\rm d}u)=1$.
          Then, ${\mathcal B}^0_{\sigma_\infty}$ is an ${\rm SU}_2({\mathbf R})$-equivariant pairing  
          and we have
          \begin{align*}
             {\mathcal B}^0_{\sigma_\infty}(W^i_{\sigma, \infty}, W^j_{\sigma, \infty})  
          =& {\mathcal B}_{\mathcal W}( X^{n+1+i} Y^{n+1-i}, X^{n+1+j} Y^{n+1-j} ) 
           = \begin{cases}   \binom{2n+2}{n+1+i}^{-1},    &  (i=j),  \\
                                      0,   &  (i\neq j).   \end{cases}
        \end{align*}
        In particular, ${\mathcal B}^0_{\sigma_\infty}$ gives an explicit form of ${\mathcal B}_{\sigma_\infty}$ in (\ref{eq:BInf}).
}
\end{lem}
\begin{proof}
Put $u_i=X^{n+1+i}Y^{n+1-i}$.   
It suffices to compute ${\mathcal B}^0_{\sigma, \infty}(W^i_{\sigma, \infty}, W^j_{\sigma, \infty})$ explicitly. 
Since we have 
\begin{align*}
  W^i_{\sigma, \infty}(gu) = \langle W_{\sigma, \infty}(g), \tau_{2n+2}(u) u_i  \rangle_{\mathcal W},   
      \quad (u\in {\rm SU}_2({\mathbf R})),  
\end{align*}
we find that 
\begin{align*}
     & {\mathcal B}^0_{\sigma, \infty}(W^i_{\sigma, \infty}, W^j_{\sigma, \infty})   \\
  = & \int_{{\rm SU}_2({\mathbf R})}  {\rm d}u
      \int^\infty_0 \frac{{\rm d}t}{t}  \\
     & \quad \left\langle  \sum^{n+1}_{k=-n-1} W^k_{\sigma, \infty} \left(\begin{pmatrix} t &  0 \\ 0 & 1 \end{pmatrix} \right) u^\vee_k, \tau_{2n+2}(u) u_i  \right\rangle_{\mathcal W}
      \left\langle  \sum^{n+1}_{l=-n-1} \overline{ W^l_{\sigma, \infty} \left(\begin{pmatrix} t &  0 \\ 0 & 1 \end{pmatrix} \right)}  u^\vee_l, \tau_{2n+2}(\overline{u}) u_j  \right\rangle_{\mathcal W}   \\
 =&  \sum_{k,l}    
      \int^\infty_0 
       W^k_{\sigma, \infty}\left(\begin{pmatrix} t &  0 \\ 0 & 1 \end{pmatrix}\right)  \overline{ W^l_{\sigma, \infty}\left( \begin{pmatrix} t &  0 \\ 0 & 1 \end{pmatrix} \right) }
       \frac{{\rm d}t}{t}
       \int_{{\rm SU}_2({\mathbf R})}
        \left\langle u^\vee_k, \tau_{2n+2}(u) u_i  \right\rangle_{\mathcal W}
      \left\langle    u^\vee_l, \tau_{2n+2}(\overline{u}) u_j  \right\rangle_{\mathcal W}
       {\rm d}u. 
\end{align*}
The Schur's orthogonality relation (\cite[Corollary 4.10]{kn02}) yields 
\begin{align*}
   & \int_{{\rm SU}_2({\mathbf R})}
        \langle u^\vee_k, \tau_{2n+2}(u) u_i  \rangle_{\mathcal W}
      \langle    u^\vee_l, \tau_{2n+2}(\overline{u}) u_j  \rangle_{\mathcal W}
       {\rm d}u  \\   \displaybreak[0]
  =& \int_{{\rm SU}_2({\mathbf R})}
        \langle u^\vee_k,   \tau_{2n+2}(w_0) \tau_{2n+2}(\overline{u}) \tau_{2n+2}(w^{-1}_0)  u_i  \rangle_{\mathcal W}
        \overline{ \langle     u^\vee_l,  \tau_{2n+2}(w_0) \tau_{2n+2}( \overline{u} ) \tau_{2n+2}(w^{-1}_0)  u_j  \rangle_{\mathcal W}  }
       {\rm d}u  \\    \displaybreak[0]
  =& \int_{{\rm SU}_2({\mathbf R})}
        {\mathcal B}_{\mathcal W}(u^\vee_k, \tau_{2n+2}(u)\tau_{2n+2}(w^{-1}_0)u_i)
        \overline{ {\mathcal B}_{\mathcal W}( u^\vee_l, \tau_{2n+2}(u)\tau_{2n+2}(w^{-1}_0) u_j ) }
       {\rm d}u  \\   \displaybreak[0]
       =& \frac{1}{{\rm dim}{\mathcal W}_{2n+2}({\mathbf C})}
         {\mathcal B}_{\mathcal W}( u^\vee_k,   u^\vee_l )
         \overline{ {\mathcal B}_{\mathcal W}(   \tau_{2n+2}(w^{-1}_0) u_i,\tau_{2n+2}(w^{-1}_0) u_j   ) } \\ 
       =& \frac{\delta_{k,l} \delta_{i,j} }{{\rm dim}{\mathcal W}_{2n+2}({\mathbf C})}  
              \binom{2n+2}{n+1+k}
            \binom{2n+2}{n+1+i}^{-1},     
\end{align*}
where $\delta_{s,t}$ is the Kronecker's delta.
We compute the following integral:
\begin{align*}
      \int^\infty_0 
       W^k_{\sigma, \infty} \left(\begin{pmatrix} t &  0 \\ 0 & 1 \end{pmatrix} \right)  \overline{ W^{k}_{\sigma, \infty} \left(\begin{pmatrix} t &  0 \\ 0 & 1 \end{pmatrix} \right)}
       \frac{{\rm d}t}{t}  
  =&  \int^\infty_0   
         2^4 \sqrt{-1}^{k} t^{n+2} K_{k}(4\pi t) 
         \times \overline{  2^4\sqrt{-1}^{k} t^{n+2} K_{k}(4\pi t) } 
                              \frac{{\rm d}t}{t} \\      
  =& 2^8 \times \int^\infty_0   
          t^{2n+3}  K_k(4\pi t) K_{k}(4\pi t)   
        {\rm d}t.      
\end{align*}
By \cite[page 101]{mos66},  we find that 
\begin{align*}
 &  \int^\infty_0   
          t^{2n+3}  K_k(4\pi t) K_{k}(4\pi t)   
        {\rm d}t   \\
=& 2^{2n+3-2} (4\pi)^{-2n-3-1} \times \frac{1}{\Gamma(2n+4)}  \\
  &  \times \Gamma \left(\frac{1+2k+2n+3}{2} \right)  \Gamma\left(\frac{1+2n+3}{2} \right) 
                 \Gamma \left(\frac{1+2n+3}{2} \right) \Gamma\left(\frac{1-2k+2n+3}{2}\right) \\
  & \times {}_2F_1\left(\frac{1+2k+2n+3}{2}, \frac{1+2n+3}{2}; 2n+4;0 \right)  \\
=& 2^{-4} \cdot 2 \cdot  (2\pi)^{-2n-4} \Gamma(2n+4) \times \frac{(n+1+k)!(n+1)!(n+1)!(n+1-k)!}{ (2n+3)! (2n+3)!}  \\
=& 2^{-4} \cdot \frac{\Gamma_{\mathbf C}(2n+4)}{{\rm dim}{\mathcal W}_{2n+2}({\mathbf C})^2 \binom{2n+2}{n+1} } \times \binom{2n+2}{n+1+k}^{-1}.
\end{align*}
Therefore we obtain 
\begin{align*}
  {\mathcal B}^0_{\sigma, \infty}( W^i_{\sigma, \infty}, W^j_{\sigma, \infty} )
  =& \frac{\delta_{i,j}}{{\rm dim}{\mathcal W}_{2n+2}({\mathbf C})}
  \cdot \binom{2n+2}{n+1+i}^{-1}
  \sum_{k}  \binom{2n+2}{n+1+k}  \\
  &\times 2^8 \times   2^{-4} \cdot \frac{\Gamma_{\mathbf C}(2n+4)}{{\rm dim}{\mathcal W}_{2n+2}({\mathbf C})^2 \binom{2n+2}{n+1} } \times \binom{2n+2}{n+1+k}^{-1}  \\
  =& \delta_{i,j} \times 
       2^{4} \cdot \frac{\Gamma_{\mathbf C}(2n+4)}{{\rm dim}{\mathcal W}_{2n+2}({\mathbf C})^2 \binom{2n+2}{n+1} }  
       \times  \binom{2n+2}{n+1+i}^{-1}.
\end{align*}
This proves the lemma.
\end{proof}

For later use, we prepare the following lemma which immediately follows  from Lemma \ref{l:minwhitt}. 
\begin{lem}\label{l:minf0}
{\itshape 
We have 
\begin{align*}
   \xi_\infty f^0_{\infty, u} = (-1)^{n+1} \tau_{2n+2}(w_0) f^0_{\infty, u}.   
\end{align*}
In particular, we have 
\begin{align*}
   \xi_\infty f^0_{\infty, j} = (-1)^j f^0_{\infty,-j} 
\end{align*}
for $f^0_{\infty, j}:= f^0_{\infty, u_j} (u_i=X^{n+1+j}Y^{n+1-j})$. 
}
\end{lem}

We denote by ${\rm d}\widetilde{h}$ (resp. ${\rm d}h_0$) the Tamagawa measure on $Z_H({\mathbf A})\backslash H({\mathbf A})$ 
(resp. $Z_{H^0}({\mathbf A})\backslash H^0({\mathbf A})$), 
 ${\rm d}\epsilon_v$  the Haar measure on $\mu_2({\mathbf Q}_v)$ such that ${\rm vol}(\mu_2({\mathbf Q}_v),  {\rm d}\epsilon_v)=1$
and 
${\rm d}\epsilon$ the product measure $\prod_v {\rm d}\epsilon_v$ 
on $\mu_2({\mathbf A})$.
Note that we have
\begin{align*}
  \int_{ Z_H({\mathbf A}) H({\mathbf Q}) \backslash H({\mathbf A}) }   f(\widetilde{h}) {\rm d}\widetilde{h}
  =   \int_{\mu_2({\mathbf Q}) \backslash \mu_2({\mathbf A})  }  
        \int_{ Z_{H^0}({\mathbf A}) H^0({\mathbf Q}) \backslash H^0({\mathbf A}) }
        f(h_0\epsilon) {\rm d}h_0 {\rm d}\epsilon 
\end{align*}
for each $f\in L^1( Z_H({\mathbf A}) H({\mathbf Q}) \backslash H({\mathbf A}) )$.
Define an Hermitian pairing ${\mathcal B}_{\widetilde{\sigma}} : {\mathcal A}({\widetilde{\sigma}})^{\otimes 2} \to {\mathbf C}$ to be 
\begin{align*}
  {\mathcal B}_{\widetilde{\sigma}}(f_1, f_2) 
   =  \int_{ Z_H({\mathbf A}) H({\mathbf Q}) \backslash H({\mathbf A}) }   f_1(\widetilde{h}) \bar{f_2(\widetilde{h})} {\rm d}\widetilde{h},   
   \quad (f_1, f_2 \in {\mathcal A}(\widetilde{\sigma})). 
\end{align*}
We also define an Hermitian pairing on $\langle\cdot, \cdot\rangle_{H^0}: {\mathcal M}_n(H^0, {\mathcal U}_{\mathfrak N})^{\otimes 2} \to {\mathbf C}$ by 
\begin{align*}
\langle {\mathbf f}_1, {\mathbf f}_2 \rangle_{H^0}
   = \int_{ Z_{H^0}({\mathbf A}) H^0({\mathbf Q}) \backslash H^0({\mathbf A}) }   \langle {\mathbf f}_1( h_0 ),  \tau_{2n+2}(w_0) \bar{{\mathbf f}_2( h_0 )}  \rangle_{\mathcal W} {\rm d}h_0,  
    \quad ({\mathbf f}_1, {\mathbf f}_2 \in {\mathcal M}_n(H^0, {\mathcal U}_{\mathfrak N})). 
\end{align*}

The relation of a global pairing ${\mathcal B}_{\widetilde{\sigma}}$  
     and a local pairings ${\mathcal B}_{ \widetilde{\sigma}_v }$ for each $v\in \Sigma_{\mathbf Q}$ is given as follows:  

\begin{lem}
{\itshape  We have
\begin{align*}
   {\mathcal B}_{\widetilde{\sigma}}
   = \frac{ \langle {\mathbf f}, {\mathbf f} \rangle_{H^0} }{ \dim {\mathcal W}_{2n+2}({\mathbf C}) } \cdot \prod_v{\mathcal B}_{\widetilde{\sigma}_v} 
\end{align*}
under the fixed isomorphism $j:\otimes_v \widetilde{\sigma}_v \to \widetilde{\sigma}$ in (\ref{LocGrobIsoGO}).
}
\end{lem}
\begin{proof}
This lemma is proved in the same way as \cite[Lemma 5.2]{hn}
\end{proof}

\subsection{Rallis inner product formula}\label{sec:Rinn}

In this subsection, we reduce a computation of the inner product of HST lifts to computations of certain local integrals by using the Rallis inner product formula. 
Combined with the results of Section \ref{sec:complocint}, we state the explicit inner product formula (Theorem \ref{HSTInnprd}) in representation theoretic language. 

For each place $v$ of ${\mathbf Q}$ and  Bruhat-Schwartz functions $\varphi, \varphi^\prime\in {\mathcal S}({\mathbf X}_v )$, we define
\begin{align*}
       {\mathcal B}_{\omega_v}(\varphi, \varphi^\prime)
  = & \int_{ {\mathbf X}_v }   \varphi (x) \bar{ \varphi^\prime(x) }   {\rm d}x, 
\end{align*}
where ${\rm d}x$ is the self dual Haar measure on ${\mathbf X}_v$ with respect to $\psi_{{\mathbf Q}, v}$.
For $\varphi  \in {\mathcal S}({\mathbf X}_{\mathbf A})$ and $f \in {\mathcal A}(\widetilde{\sigma})$, 
define $\theta( \varphi, f ): {\rm GSp}^+_4({\mathbf Q}) \backslash {\rm GSp}^+_4({\mathbf A}) \to {\mathbf C}$ by
\begin{align*}
   \theta(\varphi, f)(g)
   =\theta(g;  \varphi, f)
   := \int_{[H_1]}  \sum_{x\in {\mathbf X}} \omega(h_1h^\prime,g)\varphi(x) f(h_1h^\prime)     {\rm d} h_1,   
   \quad (\nu(h^\prime)=\nu(g)),  
\end{align*}
where ${\rm d}h_1 = \prod_v {\rm d}h_{1,v}$ is the Tamagawa measure on $H_1({\mathbf A})$.
We extend $\theta( \varphi, f)$ to the automorphic form on ${\rm GSp}_4({\mathbf A})$ so that 
the support of the extension is in ${\rm GSp}_4({\mathbf Q}) {\rm GSp}^+_4({\mathbf A})$
and denote it by the same notation.

To state the Rallis inner product formula,  
we recall the definition of the Asai $L$-function $L(s, {\rm As}^+(\pi))= \prod_{v\in \Sigma_{\mathbf Q}} L(s, {\rm As}^+(\pi_v))$
according to \cite[Section 3.1, 3.2]{hn} and \cite[Section 4]{gh99}.   
For each finite $v\in \Sigma_{\mathbf Q}$, $L(s, {\rm As}^+(\pi_v))$ is defined in the same way with \cite[Definition 3.1]{hn}.    
Define $L(s, {\rm As}^+(\pi_\infty))$ to be 
\begin{align*}
   \Gamma_{\mathbf C}(s+n+1) \Gamma_{\mathbf R}(s+1)^2, 
\end{align*}
where $\Gamma_{\mathbf R} (s) = \pi^{-\frac{s}{2}} \Gamma(\frac{s}{2})$, and $\Gamma_{\mathbf C}(s) = 2 (2\pi)^{-s}  \Gamma(s)$. 
Some of the analytic properties of $L(s, {\rm As}^+(\pi))$ is summarized in \cite[Theorem 3.3]{hn}.  
In particular, we recall that $L(1, {\rm As}^+(\pi))\neq 0$ under the condition that $\pi$ is not conjugate self-dual.  

\begin{prop}\label{p:Rallis}(Rallis inner product formula)
{\itshape 
Let $\varphi_1=\otimes_v \varphi_{1,v}, \varphi_2=\otimes_v \varphi_{2,v} \in {\mathcal S}({\mathbf X}_{\mathbf A})=\otimes_v {\mathcal S}({\mathbf X}_v)$ 
and $f_1=\otimes_v f_{1,v}, f_2=\otimes_v f_{2,v} \in \otimes_v{\mathcal A}(\widetilde{\sigma}_v) \cong \otimes_v\widetilde{\mathcal V}_v$. 
Then, 
\begin{align*}
  \langle  \theta(\varphi_1, f_1), \theta(\varphi_2, f_2) \rangle
  :=& \int_{Z_G({\mathbf A}) G({\mathbf Q}) \backslash G({\mathbf A})}  \theta(\varphi_1, f_1)(g) \overline{ \theta(\varphi_2, f_2)(g) }  {\rm d}g  \\
   =& \frac{\langle {\mathbf f}, {\mathbf f} \rangle_{H^0}}{\dim {\mathcal W}_{2n+2}({\mathbf C})}
     \cdot  \frac{L(1,{\rm As}^+(\pi))}{\zeta(2)\zeta(4)}
      \prod_v {\mathcal Z}^\ast_v(\varphi_{1,v}, \varphi_{2,v}, f_{1,v}, f_{2,v}   ),  
\end{align*}
where ${\mathcal Z}^\ast_v(\varphi_{1,v}, \varphi_{2,v}, f_{1,v}, f_{2,v}   )$ is defined to be 
\begin{align*}
  \frac{\zeta_v(2)\zeta_v(4)}{L(1,{\rm As}^+(\pi_v))} 
  \int_{H_1({\mathbf Q}_v)} 
      {\mathcal B}_{\omega_v}(\omega_v(h_{1,v})\varphi_{1,v}, \varphi_{2,v}  ) 
      {\mathcal B}_{\widetilde{\sigma}_v}(\widetilde{\sigma}_v(h_{1,v})f_{1,v}, f_{2,v}  )  
      {\rm d} h_{1,v}.
\end{align*}
}
\end{prop}
\begin{proof}
Apply \cite[Proposition 11.2, Theorem 11.3]{gqt} for $n=m=4, r=1, \epsilon_0=-1$ in the notation in \cite{gqt}.
\end{proof}

Define $\langle\cdot, \cdot \rangle_{\mathcal L}: {\mathcal L}_\lambda({\mathbf C}) \otimes {\mathcal L}_\lambda({\mathbf C}) \to {\mathbf C}$
to be the pairing $\langle\cdot,\cdot\rangle_{n}$ which is introduced in Section \ref{algrep}.
For vector-valued Siegel cusp forms ${\mathcal F}_1, {\mathcal F}_2: G({\mathbf A}) \to {\mathcal L}_\lambda({\mathbf C})$, 
define an Hermitian pairing to be 
\begin{align*}
  ({\mathcal F}_1, {\mathcal F}_2)_G
  =\int_{Z_G({\mathbf A}) G({\mathbf Q}) \backslash G({\mathbf A}) } 
      \langle {\mathcal F}_1(g), \overline{{\mathcal F}_2(g)}  \rangle_{\mathcal L}
      {\rm d} g. 
\end{align*}
One of main purposes of this paper is to compute 
$(\theta(\widetilde{\varphi}, \widetilde{\mathbf f}^\dag),  \theta(\widetilde{\varphi}, \widetilde{\mathbf f}^\dag) )_G$ in an explicit way.  
To resume computation, we prepare some notation. 
Let $H^0_1={\rm SO}(V)$.  
For $i=1, 2$ and 
\begin{align*}
    & \varphi^\alpha_i  \in {\mathcal S}({\mathbf X}_\infty)  \otimes {\mathbf C}[X,Y]_{2n+2},     \\
    & \varphi_i = \sum^n_{\alpha=0}\varphi^\alpha_i \binom{n}{\alpha}  X^{n-\alpha}Y^\alpha    \in {\mathcal S}({\mathbf X}_\infty)\otimes {\mathbf C}[X,Y]_{2n+2}\otimes {\mathbf C}[X,Y]_n,  \\
    & f_i \in {\mathcal A}(\sigma_\infty)\otimes {\mathbf C}[X,Y]_{2n+2},
\end{align*}
we define
\begin{align*}
  & {\mathcal B}_{{\mathcal W}\otimes{\mathcal L}} (\varphi_1, \varphi_2, f_1, f_2)
  = \sum^n_{\alpha=0} 
      {\mathcal B}_{\sigma_\infty}(   {\mathcal B}_{\mathcal W}(  \varphi^\alpha_1, f_1 ), 
                                                       {\mathcal B}_{\mathcal W}(  \varphi^{n-\alpha}_2 , f_2  )  )  (-1)^\alpha \binom{n}{\alpha}.
\end{align*}
Let ${\rm d}h=\prod_v {\rm d}h_v$ be the Tamagawa measure on $H^0_1({\mathbf A})$   
and define
\begin{align*}
  {\mathcal Z}_\infty(\varphi_\infty, f^0_\infty)     
              =&  \int_{H^0_1({\mathbf R})} \int_{{\mathbf X}_\infty } 
                  {\mathcal B}_{{\mathcal W}\otimes{\mathcal L}} (  \omega_\infty(h_\infty) \varphi_\infty, \varphi_\infty,  \sigma_\infty(h_\infty)f^0_\infty, f^0_\infty)  
                  {\rm d}x  {\rm d}h_\infty,    \\ 
  {\mathcal Z}_v(\varphi_v, f^{0,\dag}_v) 
             =& \int_{H^0_1({\mathbf Q}_v)} 
                   {\mathcal B}_{\omega_v}(\omega_v(h_v)\varphi_v, \varphi_v)
                   {\mathcal B}_{\widetilde{\sigma}_v}(\sigma_v(h_v)f^{0,\dag}_v, f^{0,\dag}_v )
                   {\rm d}h_v,    
                   \quad (v\in \Sigma_{\mathbf Q}, v< \infty),     
\end{align*}
where $f^0_\infty, f^{0,\dag}_v$ are defined in Section \ref{pairH}.

The following proposition reduces a computation of the inner product of $\theta(\widetilde{\varphi}, \widetilde{\mathbf f}^\dag)$
to computations of local integrals on $H^0_1({\mathbf Q}_v)$. 
 
\begin{prop}\label{prop:inntolocint}
{\itshape
We have
\begin{align*}
       \frac{  ( \theta(\widetilde{\varphi}, \widetilde{\mathbf f}^\dag), \theta(\widetilde{\varphi}, \widetilde{\mathbf f}^\dag )  )_G }{\langle {\mathbf f}, {\mathbf f}  \rangle_{H^0}} 
  =  \frac{1}{ \dim {\mathcal W}_{2n+2}({\mathbf C})} 
     \cdot  \frac{L(1,{\rm As}^+(\pi))}{\zeta(2)\zeta(4)}  
       \cdot \prod_{v  \in \Sigma_{\mathbf Q} } {\mathcal Z}^\ast_v(\varphi_v),    
\end{align*}
where
\begin{align*}
  {\mathcal Z}^\ast_v(\varphi_v) = \frac{\zeta_v(2)\zeta_v(4)}{L(1,{\rm As}^+(\pi_v))} 
  \times \begin{cases}
                {\mathcal Z}_\infty(\varphi_\infty, f^0_\infty), & (v= \infty),  \\ 
             {\mathcal Z}_v(\varphi_v, f^{0,\dag}_v), & (v< \infty).
              \end{cases}
\end{align*}
}
\end{prop}
\begin{proof}
Firstly, we reduce a computation of inner products to computations of local integrals on $H_1({\mathbf Q}_v)$ by using the Rallis inner product formula (Proposition \ref{p:Rallis}).  

We fix some notation. 
We identify the ${\rm SU}_2({\mathbf R})$-representation space ${\mathcal W}_{2n+2}({\mathbf C})$ with ${\mathbf C}[X,Y]_{2n+2}$. 
Let $u_j = X^{n+1+j} Y^{n+1-j}$ for $-n-1\leq j \leq n+1$, 
which gives a pair of basis of ${\mathcal W}_{2n+2}({\mathbf C})$. 
Denote by $u^\vee_j$ the dual basis of $u_j$ with respect to $\langle \cdot, \cdot  \rangle_{\mathcal W}$.
Let 
\begin{align*}
\widetilde{u}_j = (u_j,  (-1)^{n+1+j} \delta(\infty) \tau_{2n+2}(w_0)u_j)
             \in \widetilde{W}_{2n+2}({\mathbf C}).
\end{align*}
We also identify the ${\rm U}_2({\mathbf R})$-representation space ${\mathcal L}_\lambda({\mathbf C})$ with ${\mathbf C}[X,Y]_n$.
We recall and prepare some notation:
\begin{align*}
  \varphi_\infty (x) =&    \sum^{n}_{\alpha=0} \varphi^\alpha_\infty(x) \binom{n}{\alpha} X^{n-\alpha} Y^\alpha; \quad    
  \varphi^\alpha_\infty (x) = \sum^{n+1}_{j=-n-1} \varphi^\alpha_{\infty, j} (x)   u_j
                   = \sum^{n+1}_{j=-n-1} \langle \varphi^\alpha_\infty(x), u^\vee_j  \rangle_{\mathcal W}u_j;   \\  
 \widetilde{\varphi}^\alpha_\infty 
     =& (\varphi^\alpha_\infty,  (-1)^{n+1+j}  \delta(\infty)\tau_{2n+2}(w_0)\varphi^\alpha_\infty   ) 
               \in {\mathcal S}({\mathbf X}_\infty)\otimes \widetilde{\mathcal W}_{2n+2}({\mathbf C})^\vee; \\  
   \widetilde{u}^\vee_j =& 
                  \frac{1}{2} ( u^\vee_j,  (-1)^{n+1+j} \delta(\infty)  \tau_{2n+2}(w_0) u^\vee_j );  \\    
     \widetilde{\varphi}^\alpha_{\infty, j} 
        =& \langle \widetilde{\varphi}^\alpha_\infty,  \widetilde{u}^\vee_j   \rangle_{\widetilde{\mathcal W}}
        =    \varphi^\alpha_{\infty,j}  
             \in {\mathcal S}({\mathbf X}_\infty). 
\end{align*}
Let 
\begin{align*}
  \varphi^\alpha  =& \varphi^\alpha_{\infty}\bigotimes_{v<\infty}\varphi_v \in {\mathcal S}({\mathbf X}_{\mathbf A})\otimes {\mathcal W}_{2n+2}({\mathbf C})^\vee; \quad  
  \varphi^\alpha_j  = \varphi^\alpha_{\infty, j}\bigotimes_{v<\infty}\varphi_v \in {\mathcal S}({\mathbf X}_{\mathbf A}); \\
  \widetilde{\varphi}^\alpha  
  =& (\varphi^\alpha,  (-1)^{n+1} \delta(\infty) \tau_{2n+2}(w_0) \varphi^\alpha   ) 
       \in {\mathcal S}({\mathbf X}_{\mathbf A})\otimes \widetilde{\mathcal W}_{2n+2}({\mathbf C})^\vee.
\end{align*}
For $\widetilde{\mathbf f}^\dag \in {\mathcal M}_{2n+2}(H, {\mathfrak N})$, 
we put
\begin{align*}
   \widetilde{\mathbf f}^\dag_j(h) =& \langle  \widetilde{\mathbf f}^\dag   , \widetilde{u}_j \rangle_{\widetilde{\mathcal W}}.
\end{align*}
By using the above notation, we find that
\begin{align*}
   \theta(\varphi^\alpha, \widetilde{\mathbf f}^\dag)(g)
   :=& \int_{[H_1]}  \langle  \sum_{x\in {\mathbf X}}  \omega(h_1 h^\prime, g) \widetilde{\varphi}^\alpha(x), \widetilde{\mathbf f}^\dag(h_1) \rangle_{\widetilde{\mathcal W} }   {\rm d}h_1   
    = \sum^{n+1}_{j=-n-1}  \theta( \widetilde{\varphi}^\alpha_j, \widetilde{\mathbf f}^\dag_j)(g)
                                         \binom{n}{\alpha} X^{n-\alpha} Y^\alpha.  
\end{align*}

The definition of the pairings $\langle\cdot, \cdot\rangle$ and $(\cdot, \cdot)_G$ show that  
\begin{align*}
   (\theta(\widetilde{\varphi}, \widetilde{\mathbf f}^\dag), \theta(\widetilde{\varphi}, \widetilde{\mathbf f}^\dag))_G
   =&\sum^{n}_{\alpha=0} \int_{[Z_G \backslash G]} 
              \theta(\widetilde{\varphi}^\alpha, \widetilde{\mathbf f}^\dag) 
              \overline{\theta(\widetilde{\varphi}^{n-\alpha}, \widetilde{\mathbf f}^\dag) }
              (-1)^\alpha\binom{n}{\alpha} {\rm d}g \\
   =&\sum_\alpha\sum^{n+1}_{i,j=-n-1}   
         \int_{[Z_G \backslash G]} 
              \theta(\widetilde{\varphi}^\alpha_i, \widetilde{\mathbf f}^\dag_i) (g)
              \overline{\theta(\widetilde{\varphi}^{n-\alpha}_j, \widetilde{\mathbf f}^\dag_j) }(g)
              (-1)^\alpha\binom{n}{\alpha} {\rm d}g   \\
   =&\sum_{\alpha, i, j }  
             \langle \theta(\widetilde{\varphi}^\alpha_i, \widetilde{\mathbf f}^\dag_i), 
              \overline{\theta(\widetilde{\varphi}^{n-\alpha}_j, \widetilde{\mathbf f}^\dag_j) } \rangle
              (-1)^\alpha\binom{n}{\alpha} .  
\end{align*}
The Rallis inner product formula (Proposition \ref{p:Rallis}) yields that 
\begin{align*}
\langle \theta(\widetilde{\varphi}^\alpha_i, \widetilde{\mathbf f}^\dag_i), 
              \overline{\theta(\widetilde{\varphi}^{n-\alpha}_j, \widetilde{\mathbf f}^\dag_j) } \rangle
= \frac{1}{ \dim {\mathcal W}_{2n+2}({\mathbf C})} \cdot 
   \frac{L(1,{\rm As}^+(\pi))}{\zeta(2)\zeta(4)} \cdot \widetilde{\mathcal Z}^\ast_{ij} \prod_{v<\infty} \widetilde{\mathcal Z}^\ast_v, 
\end{align*}
where 
\begin{align*}
  \widetilde{\mathcal Z}_{ij} =&  
                                   \int_{H_1({\mathbf R})} 
                                  {\mathcal B}_{\omega_\infty} (\omega_\infty(h_{1,\infty})\widetilde{\varphi}^\alpha_{\infty,i}, \widetilde{\varphi}^{n-\alpha}_{\infty, j}  ) 
                                  {\mathcal B}_{\widetilde{\sigma}_\infty}(\widetilde{\sigma}_\infty(h_{1,\infty}) \widetilde{f}^{0}_{\infty, i}, \widetilde{f}^{0}_{\infty, j} ) 
                                  {\rm d}h_{1,\infty},
                                  \quad (\widetilde{f}^{0}_{\infty, i}:=\widetilde{f}^{0 }_{\infty, u_i}), \\
  \widetilde{\mathcal Z}_v =&   \int_{H_1({\mathbf Q}_v)} 
                                  {\mathcal B}_{\omega_v} (\omega_v(h_{1,v})\widetilde{\varphi}_v, \widetilde{\varphi}_v  ) 
                                  {\mathcal B}_{\widetilde{\sigma}_v}(\widetilde{\sigma}_v(h_{1,v}) \widetilde{f}^{0,\dag}_v, \widetilde{f}^{0,\dag}_v ) 
                                  {\rm d}h_{1,v}.   \\
  \widetilde{\mathcal Z}^\ast_v =& \frac{\zeta_v(2)\zeta_v(4)}{L(1,{{\rm As}^+(\pi_v)})} 
                                                        \cdot \widetilde{\mathcal Z}_v.
\end{align*}

To prove the proposition, 
it suffices to write down local integrals  $\widetilde{\mathcal Z}_{ij}$ and $\widetilde{\mathcal Z}_v$  on $H_1({\mathbf Q}_v)$ in terms of local integrals on $H^0_1({\mathbf Q}_v)$.

For each finite place $v\in \Sigma_{\mathbf Q}$, we find that 
\begin{align*}
   \widetilde{\mathcal Z}_v = {\mathcal Z}_v,  
\end{align*}
which is proved in the same way with \cite[(5.13)]{hn}. 

We compute local integrals $\widetilde{\mathcal Z}_{ij}$. 
We put 
\begin{align*}
     {\mathcal Z}_{ij}
  =&  \int_{H^0_1({\mathbf R})} 
                                  {\mathcal B}_{\omega_\infty} (\omega_\infty(h_\infty) \varphi^\alpha_i, \varphi^{n-\alpha}_j  ) 
                                  {\mathcal B}_{\sigma_\infty}(\sigma_\infty(h_\infty) f^0_{\infty, i}, f^0_{\infty, j} ) 
                                  {\rm d}h_{\infty}, 
       \quad ( f^0_{\infty, i}:= f^0_{\infty, u_i} ).
\end{align*}
Recall that 
\begin{align*}
  \widetilde{\sigma}_\infty(h_\infty) \widetilde{f}^0_{\infty,i}
  =& (\sigma(h_\infty) f^0_{\infty,i},   \sigma(h^c_\infty) \delta(\infty) \xi_\infty f^0_{\infty,i} ),  
\end{align*}
for $h_\infty \in H^0({\mathbf R})$.   
By the definition of ${\mathcal B}_{\widetilde{\sigma}_\infty}$, we have 
\begin{align*}
     {\mathcal B}_{\widetilde{\sigma}_\infty} 
         (\widetilde{\sigma}_\infty(h_\infty) \widetilde{f}^0_{\infty, i}, \widetilde{f}^0_{\infty, j} ) 
 =&  \frac{1}{2} \left\{  
          {\mathcal B}_{\sigma_\infty} (\sigma(h_\infty) f^0_{\infty,i}, f^0_{\infty,j})
       + {\mathcal B}_{\sigma_\infty} ( \sigma(h^c_\infty)  \delta(\infty) \xi_\infty f^0_{\infty,i},  \delta(\infty)  \xi_\infty f^0_{\infty,j})
          \right\}   \\ 
 =&  \frac{1}{2} \left\{  
          {\mathcal B}_{\sigma_\infty} (\sigma(h_\infty) f^0_{\infty,i}, f^0_{\infty,j})
       + {\mathcal B}_{\sigma_\infty} ( \xi_\infty \sigma(h_\infty)  f^0_{\infty,i}, \xi_\infty f^0_{\infty, j})
          \right\}   \\ 
 =&  \frac{1}{2} \left\{  
          {\mathcal B}_{\sigma_\infty} (\sigma(h_\infty) f^0_{\infty,i}, f^0_{\infty,j})
       + {\mathcal B}_{\sigma_\infty} (\sigma(h_\infty)  f^0_{\infty, i},  f^0_{\infty,j})
          \right\}  \\
  =&  {\mathcal B}_{\sigma_\infty} (\sigma(h_\infty) f^0_{\infty,i}, f^0_{\infty,j}),   \\ 
 {\mathcal B}_{\widetilde{\sigma}_\infty} 
         (\widetilde{\sigma}_\infty( h_\infty {\mathbf t}) \widetilde{f}^0_{\infty, i}, \widetilde{f}^0_{\infty, j} ) 
 =&  \frac{1}{2} \left\{  
          {\mathcal B}_{\sigma_\infty} (  \sigma(h_\infty) \delta(\infty) \xi_\infty f^0_{\infty,i}, f^0_{\infty,j})
       + {\mathcal B}_{\sigma_\infty} (  \sigma(h^c_\infty)f^0_{\infty,i}, \delta(\infty)\xi_\infty f^0_{\infty,j})
          \right\} \\
 =& \frac{\delta(\infty)}{2} \left\{  
          {\mathcal B}_{\sigma_\infty} ( \sigma(h_\infty)  \xi_\infty f^0_{\infty,i}, f^0_{\infty,j})
       + {\mathcal B}_{\sigma_\infty} (  \xi_\infty \sigma(h^c_\infty)  f^0_{\infty,i},  f^0_{\infty,j})
          \right\}   \\
=& (-1)^i \delta(\infty)  {\mathcal B}_{\sigma_\infty} (\sigma(h_\infty)  f^0_{\infty,-i}, f^0_{\infty,j}).   
\end{align*}
In the last identity, we used Lemma \ref{l:minf0}.
Since  
\begin{align*}  
   \omega_{\infty}({\mathbf t}) \widetilde{\varphi}^\alpha_\infty 
  =& \widetilde{\tau}_{2n+2}({\mathbf t})\widetilde{\varphi}^\alpha_\infty  
  =   ( (-1)^{n+1}  \delta(\infty) \tau_{2n+2}(w_0) \varphi^\alpha_\infty,
               \varphi^\alpha_\infty   )
\end{align*}
by Lemma \ref{HSTwt}, we find that 
\begin{align*}
\omega_\infty({\mathbf t}) \widetilde{\varphi}^\alpha_{\infty,i}
:=\langle  \omega_{\infty}({\mathbf t}) \widetilde{\varphi}^\alpha_\infty  , \widetilde{u}_i \rangle_{\mathcal W}
= (-1)^i \delta(\infty)   \widetilde{\varphi}^\alpha_{\infty,-i}.                               
\end{align*}
Since ${\rm vol}(\mu_2({\mathbf R}), {\rm d}\epsilon_\infty)=1$, we obtain
\begin{align*}
\widetilde{\mathcal Z}_{ij}
=& \frac{1}{2}
      \int_{H^0_1({\mathbf R})}   
       {\mathcal B}_{\omega_\infty}(\omega_\infty( h_\infty ) \varphi^\alpha_i, \varphi^{n-\alpha}_j  )
         {\mathcal B}_{\widetilde{\sigma}_\infty}( \widetilde{\sigma}( h_\infty ) \widetilde{f}^0_{\infty, i}, \widetilde{f}^0_{\infty, j} )    \\   \displaybreak[0]
   & \quad \quad 
     + {\mathcal B}_{\omega_\infty}(\omega_\infty(h_\infty{\mathbf t}) \varphi^\alpha_i, \varphi^{n-\alpha}_j  )
         {\mathcal B}_{\widetilde{\sigma}_\infty}( \widetilde{\sigma}(h_\infty{\mathbf t}) \widetilde{f}^0_{\infty, i}, \widetilde{f}^0_{\infty, j} ) 
      {\rm d} h_\infty   \\   \displaybreak[0]
=& \frac{1}{2}
      \int_{H^0_1({\mathbf R})}   
       {\mathcal B}_{\omega_\infty}(\omega_\infty( h_\infty ) \varphi^\alpha_i, \varphi^{n-\alpha}_j  )
         {\mathcal B}_{ \sigma_\infty}( \sigma( h_\infty ) f^0_{\infty, i}, f^0_{\infty, j} )    \\   \displaybreak[0]
   & \quad \quad 
       +  {\mathcal B}_{\omega_\infty}(\omega_\infty(h_\infty) (-1)^i  \delta(\infty) \varphi^\alpha_{-i}, \varphi^{n-\alpha}_j  )
        (-1)^i \delta(\infty) {\mathcal B}_{\sigma_\infty}( \sigma(h_\infty )  f^0_{\infty, -i}, f^0_{\infty, j} ) 
      {\rm d} h_\infty   \\   \displaybreak[0]
=& \frac{1}{2}
     ({\mathcal Z}_{ij} + {\mathcal Z}_{-ij} ). 
\end{align*}
Therefore the definitions of ${\mathcal Z}_{ij}$ and ${\mathcal B}_{{\mathcal W}\otimes{\mathcal L}}$ show that 
\begin{align*}
    &  \sum^n_{\alpha=0}  \sum^{n+1}_{i,j=-n-1} \widetilde{\mathcal Z}_{ij}  \\
  =& \frac{1}{2} \sum^n_{\alpha=0}  \sum^{n+1}_{i,j=-n-1} 
          ( {\mathcal Z}_{ij} + {\mathcal Z}_{-ij}  )  (-1)^\alpha\binom{n}{\alpha}   \\
 =  &  \sum^n_{\alpha=0}  \sum^{n+1}_{i,j=-n-1} 
                         \left\{ \int_{{\mathbf X}_\infty}  \omega_\infty(h_\infty) \varphi^\alpha_i(x) \overline{ \varphi^{n-\alpha}_j(x) }    {\rm d}x   \right\}   
     \times     {\mathcal B}_{\sigma_\infty}(\sigma_\infty(h_\infty) f^0_{\infty, i}, f^0_{\infty, j} ) 
       (-1)^\alpha\binom{n}{\alpha}    \\
  =&  \sum^n_{\alpha=0}  \sum^{n+1}_{i,j=-n-1}  
       \int_{{\mathbf X}_\infty}
       {\mathcal B}_{\sigma_\infty}( \omega_\infty(h_\infty) \varphi^\alpha_i(x) \sigma_\infty(h_\infty) f^0_{\infty, i},  \varphi^{n-\alpha}_j(x) f^0_{\infty, j} ) 
       (-1)^\alpha\binom{n}{\alpha}   {\rm d}x  \\  
  =&  \sum^n_{\alpha=0}   
       \int_{{\mathbf X}_\infty}
       {\mathcal B}_{\sigma_\infty}(    {\mathcal B}_{\mathcal W}(  \omega_\infty(h_\infty) \varphi^\alpha(x),  \sigma_\infty(h_\infty) f^0_{\infty} ), 
                                                         {\mathcal B}_{\mathcal W}(   \varphi^{n-\alpha}(x), f^0_{\infty}  )  ) 
       (-1)^\alpha\binom{n}{\alpha}   {\rm d}x  \\  
   =&   \int_{{\mathbf X}_\infty}
          {\mathcal B}_{{\mathcal W}\otimes {\mathcal L}} 
            ( \omega_\infty(h_\infty) \widetilde{\varphi}, \widetilde{\varphi},  \sigma_\infty(h_\infty) f^0_\infty,f^0_\infty) 
          {\rm d}x. 
\end{align*}
This proves the proposition.
\end{proof}

To state an explicit inner product formula (Theorem \ref{HSTInnprd}), we fix a decomposition ${\rm d}h=\prod_{v\in \Sigma_{\mathbf Q}} {\rm d} h_v$ of the Tamagawa measure on $H^0_1({\mathbf A})$ 
to a product of Haar measures ${\rm d}h_v$ on $H^0_1({\mathbf Q}_v)$ as follows. 
For $g\in {\rm SL}_2({\mathbf C})$, write 
\begin{align*}
  g = \begin{pmatrix} t^{\frac{1}{2}}   &  t^{\frac{1}{2}}z  \\ 0 &  t^{-\frac{1}{2}} \end{pmatrix} u, 
      t \in {\mathbf R}_{>0}, z=x+\sqrt{-1}y  \in {\mathbf C},   u\in {\rm SU}_2({\mathbf R}). 
\end{align*}
Then let ${\rm d}g$ be the Haar measure on ${\rm SL}_2({\mathbf C})$ which is given by
\begin{align*}
    {\rm d}g = \frac{{\rm d}x \wedge {\rm d}y \wedge {\rm d}t }{t^3} \wedge {\rm d}u, \quad 
    \left(  g = \begin{pmatrix} t^{\frac{1}{2}}   &  t^{\frac{1}{2}}z  \\ 0 &  t^{-\frac{1}{2}} \end{pmatrix} u, 
      t \in {\mathbf R}_{>0}, z=x+\sqrt{-1}y  \in {\mathbf C},   u\in {\rm SU}_2({\mathbf R})  \right), 
\end{align*}
where ${\rm vol}({\rm SU}_2({\mathbf R}),  {\rm d}u )=1$.
We fix ${\rm d}h=\prod_{v\in \Sigma_{\mathbf Q}} {\rm d} h_v$ so that 
 ${\rm d}h_\infty$ is induced by ${\rm d}g$ via $\varrho$. 
We define a Haar measure ${\rm d}h_{\rm fin}$ on $H^0_1({\mathbf A}_{\rm fin})$ to be $\prod_{ v\in \Sigma_{\mathbf Q}, v<\infty  } {\rm d} h_v$.

Recall $N{\mathbf Z} = {\mathfrak N} \cap {\mathbf Z}$ and  let 
\begin{align}\label{def:un1}
{\mathcal U}_{N} =   \left\{   \varrho(g, \alpha) \in H^0( {\mathbf A}_{\rm fin} ) \ | \ g \in U_0(N), \alpha \in \widehat{\mathbf Z}^\times   \right\},   \quad 
{\mathcal U}_{N,1} = {\mathcal U}_{N} \cap H^0_1( {\mathbf A}_{\rm fin}). 
\end{align}

Summarizing Proposition \ref{prop:inntolocint} and the results in Section \ref{sec:complocint}, we obtain the following corollary:

\begin{thm}\label{HSTInnprd}
{\itshape Assume that
\begin{itemize}
  \item $n$ is even;
  \item $\pi$ is not conjugate self-dual; 
  \item $\delta(\infty)=-1$ and $\delta(v)=1$ for each finite $v\in \Sigma_{\mathbf Q}$;  
  \item ${\mathfrak N}$ is square-free; 
  \item $\Delta_F$ and ${\rm Nr}_{F/{\mathbf Q}}({\mathfrak N})$ are coprime;
  \item Conjecture \ref{conj:arint} on the archimedean local integral, which is hold for $n=0,2,4,6,8$.
\end{itemize}
Let  ${\mathcal P}$ be a set of rational primes which is defined in (\ref{defP}). 
Put  
\begin{align*}
   r_{F,2} = \begin{cases} 1,  &  ( 2\mid \Delta_F), \\
                                         0, &  ( 2\nmid \Delta_F).   \end{cases}  
\end{align*}
For each finite unramified places $v\in \Sigma_{\mathbf Q}$ with the residue characteristic $p$, let $\varepsilon_p$ be as in Lemma \ref{lem:nonarint}. 
Then, we have 
\begin{align*}
  \frac{( \theta(\widetilde{\varphi}, \widetilde{\mathbf f}^\dag, \theta(\widetilde{\varphi}, \widetilde{\mathbf f}^\dag ) )_G}{\langle  {\mathbf f}, {\mathbf f}  \rangle_{H^0}}  
   =& \frac{   L(1,{\rm As}^+(\pi)) }{\zeta(2)\zeta(4)}
     \cdot \frac{ (-1)^{\frac{n}{2}} {\rm vol}({\mathcal U}_{N, 1}, {\rm d}h_{\rm fin}) (2n+3) 2^{\sharp{\mathcal P}} 
                        }{ 2^{n+9}  (n+1) N^2_F \Delta^3_F 2^{-4r_{F,2}}}   \\
     & \times \frac{\zeta_{N_F}(4)}{\zeta_{N_F}(1)} 
                  \cdot \prod_{p\mid  N }   (1+\varepsilon_p)
                  \cdot \prod_{p\mid \Delta_F} (1+p^{-1}).     
\end{align*}
}
\end{thm}
\begin{proof}
The statement follows from Proposition \ref{prop:inntolocint}, Lemma \ref{lem:nonarint} and Lemma \ref{lem:arintInn}.
\end{proof}

\subsection{Classical form of inner product formula}\label{sec:clainnprd}

In this subsection, we introduce a classical form of Theorem \ref{HSTInnprd}. 
We follow the same line with \cite[Section 5.3]{hn}. 

Define the classical form $\theta^\ast_{ \widetilde{\mathbf f}^\dag}: {\mathfrak H}_2 \to {\mathcal L}_\lambda({\mathbf C})$ 
of $\theta(\widetilde{\varphi}, \widetilde{\mathbf f}^\dag)$ to be 
\begin{align*}
        \theta^\ast_{\widetilde{ \mathbf f}^\dag}(Z)
        = \frac{1}{ {\rm vol}( {\mathcal U}_{N,1} , {\rm d}h_{\rm fin}) } 
           \varrho_\lambda(J(g_\infty, {\mathbf i}))
           \theta(\widetilde{\varphi}, \widetilde{\mathbf f}^\dag)(g_\infty), 
           \quad (g_\infty \in {\rm Sp}_4({\mathbf R}), g_\infty\cdot {\mathbf i} = Z \in {\mathfrak H}_2). 
\end{align*}

Let ${\mathcal B}_{\mathcal L}: {\mathcal L}_\lambda({\mathbf C})^{\otimes 2}  \to {\mathbf C}$
be an ${\rm SU}_2({\mathbf R})$-equivariant and positive definite Hermitian pairing which is defined by 
\begin{align*}
  {\mathcal B}_{\mathcal L}(v_1, v_2) 
  = \langle  v_1, \varrho_\lambda(w_0) \overline{v_2}  \rangle_{\mathcal L}, \quad 
 \left( w_0 = \begin{pmatrix} 0 & 1 \\ -1 & 0  \end{pmatrix}, 
  v_1, v_2 \in {\mathcal L}_\lambda({\mathbf C}) \right).  
\end{align*}
Then define the Petersson norm of $\theta^\ast_{\widetilde{\mathbf f}^\dag}$
to be
\begin{align*}
       \langle \theta^\ast_{ \widetilde{\mathbf f}^\dag}, \theta^\ast_{ \widetilde{\mathbf f}^\dag} \rangle_{\mathfrak H_2}  
  =  \int_{\Gamma^{(2)}_0(N_F) \backslash {\mathfrak H}_2}   
        {\mathcal B}_{\mathcal L}( \theta^\ast_{ \widetilde{ \mathbf f}^\dag}(Z), \theta^\ast_{ \widetilde{ \mathbf f}^\dag}(Z) )
        ({\rm det} Y)^{\frac{n}{2}+2}
        \frac{ {\rm d}X {\rm d} Y  }{  ( {\rm det} Y)^3},  
\end{align*}
where $Z=X+\sqrt{-1}Y\in {\mathfrak H}_2$ 
and ${\rm d}X = \prod_{j\leq l} {\rm d}x_{jl}, {\rm d}Y=\prod_{j\leq l} {\rm d}y_{jl}$ for
$X=(x_{jl})$ and $Y=(y_{jl})$.

Recall ${\rm d}h_0$ is the Tamagawa measure on $ Z_{H^0}({\mathbf A})  \backslash  H^0({\mathbf A})$. 
Decompose ${\rm d}h_0 = \prod_{v\in \Sigma_{\mathbf Q}} {\rm d}h_{0, v}$ so that 
  ${\rm d}h_{0,\infty}$ is the Haar measure on $Z_{H^0}({\mathbf R})  \backslash  H^0({\mathbf R})$ 
  which is induced by the Haar measure ${\rm d}\epsilon_\infty$ on $\mu_2({\mathbf R})$ in Section \ref{pairH}
                      and the Haar measure ${\rm d}h_\infty$ on $H^0_1({\mathbf R})$ in Section \ref{sec:Rinn}.
Write ${\rm d}h_{0, {\rm fin}} = \prod_{v\in \Sigma_{\mathbf Q}, v<\infty} {\rm d}h_{0, v}$.  
Let $\overline{ {\mathcal U}_{N}}$  be the image of ${\mathcal U}_{N}$ 
     in $ Z_{H^0}({\mathbf A})  \backslash  H^0({\mathbf A})$. 
Then, define
\begin{align*}
       \langle {\mathbf f}_1, {\mathbf f}_2 \rangle_{\mathscr H}  
  = \frac{1}{{\rm vol}(   \overline{ {\mathcal U}_{N }}, {\rm d}h_{0, \rm fin})}
     \langle  {\mathbf f}_1, {\mathbf f}_2 \rangle_{H^0}
\end{align*}
for ${\mathbf f}_1, {\mathbf f}_2 \in {\mathcal M}_n( H^0,   {\mathcal U}_{\mathfrak N} )$.

\begin{thm}\label{thm:alginnprd}
{\itshape Let ${\mathcal P}, r_{F, 2}, \varepsilon_p$ be as in Theorem \ref{HSTInnprd}.
Define $r_F$ to be the number of the rational primes which is ramified in $F$. 
Put 
\begin{align*}
  \beta = \sharp {\mathcal P} + 4 r_{F,2} - 2n-9  -r_F.
\end{align*}
Assume the conditions in Theorem \ref{HSTInnprd}.   
Then we have 
\begin{align*}
    \frac{  \langle  \theta^\ast_{\widetilde{\mathbf f}^\dag}, \theta^\ast_{\widetilde{\mathbf f}^\dag}   \rangle_{\mathfrak H_2}
              }{\langle  {\mathbf f}, {\mathbf f} \rangle_{\mathscr H} }
 =  2^\beta N_F \Delta^{-3}_F  
      \times   L(1, {\rm As}^+(\pi)) 
                  \prod_{p\mid N } (1+\varepsilon_p) 
                  \cdot \prod_{p\mid \Delta_F}  (1+p^{-1}). 
\end{align*}
}
\end{thm}
\begin{proof}
The proof is done in the same way given in \cite[Theorem 5.7]{hn}. 
Define a pairing $\langle\!\langle  \cdot, \cdot  \rangle\!\rangle: 
                              {\mathcal L}_\lambda({\mathbf C}) \otimes {\mathcal L}_\lambda({\mathbf C}) \to {\mathbf C}$ to be 
\begin{align*}
    \langle\!\langle  v_1, v_2  \rangle\!\rangle
    = \int_{\rm SU_2({\mathbf R})}   \left\langle  v_1,   \overline{ \rho_\lambda(u)  v_2}  \right\rangle_{\mathcal L}   {\rm d}u,   
      \quad ({\rm vol}( {\rm SU}_2({\mathbf R}),  {\rm d}u) = 1, v_1, v_2 \in {\mathcal L}_\lambda({\mathbf C}) ).
\end{align*} 
In the same way with the proof of \cite[Lemma 5.6]{hn}, 
the pairing $\langle\!\langle  \cdot, \cdot  \rangle\!\rangle$ satisfies 
\begin{align*}
    \langle\!\langle  \cdot, \cdot  \rangle\!\rangle
    = \frac{\sqrt{-1}^n }{n+1} {\mathcal B}_{\mathcal L}(\cdot,, \cdot).
\end{align*}
Hence the following facts:
\begin{itemize}
\item the Tamagawa number of ${\rm PGSp}_4$ is 2;
\item ${\rm vol} ({\rm Sp}_4({\mathbf Z}) \backslash {\mathfrak H}_2, \frac{{\rm d}X {\rm d}Y}{ (\det Y)^3 }  )   = 2 \zeta(2) \zeta(4)   $ (\cite[Theorem 11]{si43});
\item $[{\rm Sp}_4({\mathbf Z}) : \Gamma_0(N_F) ]  = N^3_F \prod_{p\mid N_F}    \frac{ 1-p^{-4} }{ 1-p^{-1} }   $  (\cite[page 114, (1)]{kl59});
\end{itemize}
show that 
\begin{align*}
    \frac{  (\theta(\widetilde{\varphi}, \widetilde{\mathbf f}^\dag, \theta(\widetilde{\varphi}, \widetilde{\mathbf f}^\dag ) )_G  
              }{ {\rm vol}( {\mathcal U}_{N,1}, {\rm d}h_{\rm fin})^2  }
   = \frac{\sqrt{-1}^n }{n+1} 
      \cdot N^{-3}_F
      \prod_{p\mid N_F} \frac{ 1-p^{-1} }{ 1-p^{-4} }
      \cdot \frac{ \langle  \theta^\ast_{\mathbf f} ,\theta^\ast_{\mathbf f}  \rangle_{{\mathfrak H}_2}  }{\zeta(2)\zeta(4)}.   
\end{align*}
Then by Theorem \ref{HSTInnprd}, we find that  
\begin{align*}
      \frac{  \langle  \theta^\ast_{\widetilde{\mathbf f}^\dag}, \theta^\ast_{\widetilde{\mathbf f}^\dag}   \rangle_{\mathfrak H_2} 
                }{\langle   {\mathbf f}, {\mathbf f}   \rangle_{\mathscr H} }
  =&  \frac{  {\rm vol}( \overline{{\mathcal U}_{N}}, {\rm d}h_{0, \rm fin}) 
                  }{   {\rm vol}( {\mathcal U}_{N, 1} , {\rm d}h_{\rm fin})  }
      \cdot \frac{ N_F(2n+3) 2^{\sharp {\mathcal P}}  
                        }{ \dim {\mathcal W}_{2n+2}({\mathbf C})  2^{n+9} \Delta^3_F 2^{r_{F,2}}}  \\
    &\quad  
      \times   L(1, {\rm As}^+(\pi)) 
                  \prod_{p\mid N } (1+\varepsilon_p) 
                  \cdot \prod_{p\mid \Delta_F}  (1+p^{-1}).  
\end{align*}
By the proof of \cite[Theorem 5.7]{hn}, we have 
\begin{align*}
{\rm vol}( \overline{{\mathcal U}_{N}}, {\rm d}h_{0, \rm fin}) 
= 2^{-r_F}  {\rm vol}( {\mathcal U}_{N, 1} , {\rm d}h_{\rm fin}).
\end{align*} 
This shows the theorem. 
\end{proof}

The following corollary immediately follows from the non-vanishing of $L(1, {\rm As}^+(\pi))$:  

\begin{cor}\label{cor:nonvan}
{\itshape 
Assume conditions in Theorem \ref{thm:alginnprd} and 
\begin{itemize}
\item[(LR)] for each $p\mid N$, $\varepsilon_p=1$.  
\end{itemize}
Then, $\theta^\ast_{ \widetilde{\mathbf f}^\dag }$ is non-zero.
}
\end{cor}
\begin{proof}
By the assumptions in the statement, the non-vanishing of $\theta^\ast_{ \widetilde{\mathbf f}^\dag }$ is equivalent to 
the non-vanishing of $L(1, {\rm As}^+(\pi))$ which follows from \cite[Theorem 4.3]{gs15}.
This proves the corollary.     
\end{proof}

\section{Proof of Theorem \ref{HSTInnprd}} \label{sec:complocint}

A proof of Theorem \ref{HSTInnprd} is divided two parts: 
a computation of non-archimedean local integrals and an archimedean local integral.
In the non-archimedean case, the computation is exactly the same with \cite[Section 6.3]{hn}.
Hence the archimedean computation is the main theme of this section.  
We do it assuming Conjecture \ref{conj:arint} in Section \ref{s:ArchInt}.    
Conjecture \ref{conj:arint} is hold for small weights and 
we give a proof of Conjecture \ref{conj:arint} for $n=0$ in Appendix. 

\subsection{The non-archimedean local integral}\label{s:NArchInt}

In this subsection we give explicit formulas for the non-archimedean local integrals ${\mathcal Z}_v(\varphi_v, f^{0,\dag}_v)$ for each finite place $v\in \Sigma_{\mathbf Q}$. 
The calculation of these values is done in the exactly same way with \cite[Section 6.3]{hn}. 
We summarize the result of the computation as follows:

\begin{lem}\label{lem:nonarint}
{\itshape Assume that ${\mathfrak N}$ is square-free 
                and that ${\rm Nr}_{F/{\mathbf Q}}({\mathfrak N})$ is prime to $\Delta_F$. 
              For each finite $v\in \Sigma_{\mathbf Q}$, let $p$ be the residue characteristic of $v$.   
\begin{enumerate}
\item Suppose that $v\in \Sigma_{\mathbf Q}$ is unramified in $F$. 
         Let $\varepsilon_{w_i}$ (resp. $\varepsilon_v$) be the root number of $\pi_{w_i}$  for $i=1,2$ (resp. $\pi_v$)  
          if $v= w_1 w_2$ is split in $F$ (resp. $v$ is inert in $F$).          
         Put
         \begin{align*}
            \varepsilon_p=\begin{cases} \varepsilon_{w_1} \varepsilon_{w_2},   &  (v:\text{split in }F),  \\
                                                          \varepsilon_v,                      &  (v:\text{inert in }F). \end{cases}
         \end{align*}    
         Then, we have 
         \begin{align*}
            {\mathcal Z}_v(\varphi_v, f^{0,\dag})
            = {\rm vol}({\mathcal U}_{N, 1, v}, {\rm d}h_v) \cdot \frac{L(1,{\rm As}^+(\pi_v))}{\zeta_v(2)\zeta_v(4)} 
               \times \begin{cases}   1,  &  (p\nmid  {\rm Nr}_{F/{\mathbf Q}} ({\mathfrak N})),  \\
                                                  p^{-2} (1+\varepsilon_p),  &   (p\mid  {\rm Nr}_{F/{\mathbf Q}} ({\mathfrak N})).    \end{cases}    
         \end{align*}
\item Suppose that $v\in \Sigma_{\mathbf Q}$ is ramified in $F$. 
         Then, we have 
         \begin{align*}
            {\mathcal Z}_v(\varphi_v, f^{0,\dag})
            =  {\rm vol}({\mathcal U}_{N, 1, v}, {\rm d}h_v)
                 \cdot |2^{-4}\Delta^3_F|_v \cdot (1+p^{-1}) \cdot \frac{L(1,{\rm As}^+(\pi_v))}{\zeta_v(1)\zeta_v(2)}.             
         \end{align*}
\end{enumerate}
}
\end{lem}

\subsection{The archimedean local integral}\label{s:ArchInt}

In this subsection, we compute the archimedean local integral. 
For this purpose, we firstly prepare a lemma on the Haar measure ${\rm d}h_\infty$ on $H^0_1({\mathbf R})$, which is fixed in Section \ref{sec:Rinn}.

\begin{lem}\label{l:cptint}
{\itshape The following set is a set of representatives of $H^0_1({\mathbf R})$:
  \begin{align*}
       \left\{  \varrho(u_1, u_2, a, \varepsilon) := \varrho \left(   u_1 \begin{pmatrix} a^{\frac{1}{2}}  & 0 \\  0 & a^{-\frac{1}{2}} \end{pmatrix} u_2, \varepsilon \right)    | \  
                         0<a<1, u_1, u _2 \in \left\{\pm 1\right\}\backslash {\rm SU}_2({\mathbf R}),  
                         \varepsilon \in \{ \pm 1\}      \right\}.  
  \end{align*} 
Moreover, for $f\in L^1(H^0_1({\mathbf R}))$, we have 
\begin{align*}
   \int_{H^0_1({\mathbf R})}  f(h) {\rm d} h_\infty  
   = \pi \sum_{\varepsilon=\pm 1}   
       \int_{{\rm SU}_2({\mathbf R})/\{\pm 1\}}
       \int_{{\rm SU}_2({\mathbf R})/\{\pm 1\}}
       \int^1_0
       f(\varrho(u_1, u_2, a, \varepsilon)) 
       \frac{(a-a^{-1})^2 {\rm d}a}{a} {\rm d}u_1 {\rm d}u_2.  
\end{align*}
}
\end{lem}
\begin{proof}
The first statement follows from the Cartan decomposition.  
Hence it suffices to change variables of the integrals on $H^0_1({\mathbf R})$. 
Let 
\begin{align*}
   u(\psi,\theta,\varphi) =& \begin{pmatrix} \alpha & \beta \\ -\bar{\beta} & \bar{\alpha}  \end{pmatrix}, 
         \alpha = \cos \psi e^{\sqrt{-1}\theta}, \beta = \sin\psi e^{\sqrt{-1}\varphi}, 0<\theta, \varphi<2\pi, 0< \psi <\frac{\pi}{2}.   
\end{align*}
Denote by ${\rm d}u$ the Haar measure on ${\rm SU}_2({\mathbf R})$ such that ${\rm vol}({\rm SU}_2({\mathbf R}), {\rm d}u  )=1$. 
Then ${\rm d}u$ is explicitly given by the following formula: 
\begin{align}
  \int_{{\rm SU}_2({\mathbf R})}  f(u) {\rm d}u 
     = \int^{\frac{\pi}{2}}_0 \sin2\psi {\rm d}\psi
      \int^{2\pi}_0 \frac{{\rm d}\theta}{2\pi}
      \int^{2\pi}_0 \frac{{\rm d}\varphi}{2\pi}      
        f(u(\psi,\theta,\varphi)), 
        \quad ( f\in L^1({\rm SU}_2({\mathbf R})) ).   \label{eq:haar1}
\end{align}
For $0<a$, define
\begin{align*}
    N(\alpha, \beta, a)
    =& a^{-1}|\alpha|^2 + a|\beta|^2. 
\end{align*}
Then the following gives the Iwasawa decomposition of $\begin{pmatrix} \alpha & \beta \\ -\bar{\beta} & \bar{\alpha}  \end{pmatrix} \begin{pmatrix} a^{\frac{1}{2}} & 0 \\  0 & a^{-\frac{1}{2}}  \end{pmatrix}$:
\begin{align*}
    & \begin{pmatrix} 1 & \frac{(a^{-1}-a) \alpha \beta }{N(\alpha, \beta, a)} \\ 0 & 1  \end{pmatrix} 
          \begin{pmatrix} N(\alpha,\beta,a)^{-1} & 0 \\ 0 & 1  \end{pmatrix} 
           \begin{pmatrix} a^{-\frac{1}{2}}\alpha & a^\frac{1}{2} \beta \\ -a^\frac{1}{2} \bar{\beta} & a^{-\frac{1}{2}}\bar{\alpha}  \end{pmatrix}    \\
   = & \begin{pmatrix} 1 & \frac{(a^{-1}-a) \alpha \beta }{N(\alpha, \beta, a)} \\ 0 & 1  \end{pmatrix} 
          \begin{pmatrix} N(\alpha,\beta,a)^{-\frac{1}{2}} & 0 \\ 0 & N(\alpha,\beta,a)^{\frac{1}{2}}  \end{pmatrix} 
          \cdot N(\alpha,\beta,a)^{-\frac{1}{2}}
           \begin{pmatrix} a^{-\frac{1}{2}}\alpha & a^\frac{1}{2} \beta \\ -a^\frac{1}{2} \bar{\beta} & a^{-\frac{1}{2}}\bar{\alpha}  \end{pmatrix}.
\end{align*}
Let 
\begin{align*}
  z=&  \frac{(a^{-1}-a) \alpha \beta }{N(\alpha, \beta, a)}, \quad 
  t= N(\alpha, \beta, a)^{-1}. 
\end{align*}
Note that $|\alpha|^2 +|\beta|^2=1$. Assuming 
\begin{align*}
   \alpha=\bar{\alpha} \quad (\theta=0), 
\end{align*}
we find that 
\begin{align}\label{eq:haar2}
\begin{aligned}
   & \frac{ {\rm d}z \wedge {\rm d}t \wedge {\rm d}\bar{z} }{t^3} = 2\sqrt{-1} \frac{ {\rm d}x \wedge {\rm d}y \wedge {\rm d}t }{t^3}  \\
  =&   -2\alpha \beta^{-1}  \cdot   \frac{ (a-a^{-1})^2  }{a}  {\rm d}a \wedge {\rm d}\alpha \wedge {\rm d}\beta 
  =   \sqrt{-1} \sin(2\psi) \cdot   \frac{ (a-a^{-1})^2  }{a}  {\rm d}a \wedge {\rm d}\phi \wedge {\rm d}\psi.
\end{aligned}
\end{align}
This identity (\ref{eq:haar1}) and (\ref{eq:haar2}) prove the lemma.
\end{proof}

\begin{conj}\label{conj:arint}
{\itshape For even non-negative integer $n$, the value 
\begin{align*}
     I_n  :=&   \int^1_0 \frac{(a-a^{-1})^2 {\rm d}a}{a}      
     \int_{{\mathbf X}_\infty}  {\rm d}x
          \int^\infty_0  \frac{{\rm d} t}{t}      \\
     &  \quad
             \left\langle
                       \left\langle  \varphi_\infty \left( \begin{pmatrix}  a^{-\frac{1}{2}}  & 0 \\  0 &  a^{\frac{1}{2}}   \end{pmatrix}  x \right),  
                                               W_{\sigma, \infty}\left( \begin{pmatrix} t &  0 \\ 0 & 1 \end{pmatrix}     
                                                                                   \begin{pmatrix}  a^{\frac{1}{2}}  & 0 \\  0 &  a^{-\frac{1}{2}}   \end{pmatrix} \right)   \right\rangle_{\mathcal W},  
                                   \overline{ \left\langle  \varphi_\infty(x),    W_{\sigma, \infty}\left( \begin{pmatrix} t &  0 \\ 0 & 1 \end{pmatrix}   \right)      \right\rangle_{\mathcal W} }
              \right\rangle_{\mathcal L}
\end{align*}
is given by 
\begin{align*}
   I_n= (-1)^{\frac{n}{2}} \cdot 2^{-4n-8} \cdot (n+1)!^2 \cdot n! \cdot \pi^{-3n-6}. 
\end{align*}
}
\end{conj}

\begin{rem}
We can check by using Mathematica that Conjecture \ref{conj:arint} is hold for $n=0, 2,4, 6, 8$. 
See Appendix for a proof of Conjecture \ref{conj:arint} in the $n=0$ case.
\end{rem}

Recall that the $\Gamma$-factor of ${\rm As}^+(\pi)$ is given by
\begin{align*}
  L(s, {\rm As}^+(\pi_\infty) ) =  \Gamma_{\mathbf C}(s+n+1)  \Gamma_{\mathbf R}(s+1)^2.  
\end{align*}
The purpose of the rest of this subsection is to prove the following lemma:

\begin{lem}\label{lem:arintInn}
{\itshape
Assume Conjecture \ref{conj:arint}. 
Then we have
\begin{align}
      {\mathcal Z}_\infty(\varphi_\infty, f^0_\infty)     
              =& \frac{ (-1)^{\frac{n}{2}} 2^{-n-9} (2n+3) }{n+1} 
        \cdot \frac{  L(1, {\rm As}^+(\pi_\infty)) }{ \Gamma_{\mathbf R}(2) \Gamma_{\mathbf R}(4)  }. \label{eq:BWL}
\end{align}
}
\end{lem}
\begin{proof}
Let $C_\infty$ be the constant in Lemma \ref{l:pairB0}.
By Lemma \ref{l:pairB0}, the left-hand side of (\ref{eq:BWL}) is equal to 
\begin{align*}
       & \int_{H^0_1({\mathbf R})} \int_{{\mathbf X}({\mathbf R})} 
                  {\mathcal B}_{{\mathcal W}\otimes{\mathcal L}} (  \omega_\infty(h_\infty) \varphi_\infty, \varphi_\infty,  \sigma_\infty(h_\infty)f^0_\infty, f^0_\infty)  
                  {\rm d}x  {\rm d}h_\infty   \\   \displaybreak[0]
    =&   C_\infty \cdot 
          \int_{H^0_1({\mathbf R})}  
          \int_{{\mathbf X}_\infty}
          \int^\infty_0
          \int_{{\rm SU}_2({\mathbf R})}   \\  \displaybreak[0]
      &\quad   \left\langle
                      \left\langle  \varphi_\infty(h^{-1}_\infty x),  W_{\sigma, \infty}\left( \begin{pmatrix} t &  0 \\ 0 & 1 \end{pmatrix} u g \right)   \right\rangle_{\mathcal W},  
                       \overline{ \left\langle  \varphi_\infty(x),    W_{\sigma, \infty}\left( \begin{pmatrix} t &  0 \\ 0 & 1 \end{pmatrix} u  \right)      \right\rangle_{\mathcal W}  }
                     \right\rangle_{\mathcal L}
              {\rm d}u
             \frac{{\rm d} t}{t}
            {\rm d}x
            {\rm d}h_\infty  \\   \displaybreak[0]
    =&   C_\infty \cdot 
          \int_{H^0_1({\mathbf R})}  
          \int_{{\mathbf X}_\infty}
          \int^\infty_0
          \int_{{\rm SU}_2({\mathbf R})}   \\   \displaybreak[0]
      &\quad   \left\langle
                        \left\langle  \varphi_\infty(h^{-1}_\infty u^{-1} x),  W_{\sigma, \infty}\left( \begin{pmatrix} t &  0 \\ 0 & 1 \end{pmatrix} u g \right)   \right\rangle_{\mathcal W},  
                        \overline{ \left\langle  \varphi_\infty(x),    W_{\sigma, \infty}\left( \begin{pmatrix} t &  0 \\ 0 & 1 \end{pmatrix}   \right)      \right\rangle_{\mathcal W}  }
                     \right\rangle_{\mathcal L}
              {\rm d}u
             \frac{{\rm d} t}{t}
            {\rm d}x
            {\rm d}h_\infty   \\   \displaybreak[0]
     =&   C_\infty \cdot 
          \int_{H^0_1({\mathbf R})}  
          \int_{{\mathbf X}_\infty}
          \int^\infty_0 \\
      &\quad   \left\langle
                       \left\langle  \varphi_\infty(h^{-1}_\infty  x),  W_{\sigma, \infty}\left( \begin{pmatrix} t &  0 \\ 0 & 1 \end{pmatrix}  g \right)   \right\rangle_{\mathcal W},  
                                      \overline{ \left\langle  \varphi_\infty(x),    W_{\sigma, \infty}\left( \begin{pmatrix} t &  0 \\ 0 & 1 \end{pmatrix}   \right)      \right\rangle_{\mathcal W}  } 
                     \right\rangle_{\mathcal L}
             \frac{{\rm d} t}{t}
            {\rm d}x
            {\rm d}h_\infty,
\end{align*}
where we write $h_\infty=\varrho(g,\alpha)$.

Let   $v_{\varphi, W} = \langle \varphi( x),    W  \rangle_{\mathcal W}\in {\mathcal S}({\mathbf X}_\infty)$ 
for $\varphi \in {\mathcal S}({\mathbf X}_\infty)\otimes {\mathcal W}_{2n+2} ({\mathbf C})^\vee$ 
and $W\in {\mathcal W}(\sigma_\infty, \psi_{F,\infty})$.  
Define   
\begin{align*}
\sigma^\prime_\infty(h)v_{\varphi_\infty, W_{\sigma,\infty}}   
               =& \langle \varphi_\infty( h^{-1} x),    \sigma_\infty(h)W_{\sigma, \infty}    \rangle_{\mathcal W}  \in {\mathcal S}({\mathbf X}_\infty),     \\
{\mathcal S}^\prime({\mathbf X}_\infty)
=& \langle    \sigma^\prime_\infty(h)v_{\varphi_\infty, W_{\sigma,\infty}}     | h \in H^0_1({\mathbf R})  \rangle_{\mathbf C}.
\end{align*}
Then, $\sigma^\prime_\infty: H^0_1({\mathbf R}) \to {\rm Aut}({\mathcal S}^\prime({\mathbf X}_\infty))$ gives a representation of $H^0_1({\mathbf R})$
and $v_{\varphi_\infty, W_{\sigma,\infty}}$ is an ${\rm SU}_2({\mathbf R})$-fixed vector by Lemma \ref{HSTwt}. 
For $v_{\varphi_1, W_1}, v_{\varphi_2, W_2} \in {\mathcal S}^\prime({\mathbf X}_\infty)$, the pairing 
\begin{align*}
 \langle v_{\varphi_1, W_1}, v_{\varphi_2, W_2}  \rangle_{{\mathcal S}^\prime} 
    :=  \int_{{\mathbf X}_\infty}
      \int^\infty_0 
       \left\langle
                 \left\langle  \varphi_1(  x),  W_1\left( \begin{pmatrix} t &  0 \\ 0 & 1 \end{pmatrix}  \right)   \right\rangle_{\mathcal W},  
                            \overline{ \left\langle  \varphi_2(x),    W_2\left( \begin{pmatrix} t &  0 \\ 0 & 1 \end{pmatrix}   \right)      \right\rangle_{\mathcal W}  }
         \right\rangle_{\mathcal L}
             \frac{{\rm d} t}{t}
            {\rm d}x
\end{align*}
 makes $\sigma^\prime$ a unitary representation of $H^0_1({\mathbf R})$.
Hence \cite[X. Theorem 4.5]{hel62} shows that 
\begin{align*}
 \Omega(h) := \langle  \sigma^\prime_\infty(h)v_{\varphi_\infty, W_{\sigma, \infty}} ,  v_{\varphi_\infty, W_{\sigma, \infty}}  \rangle_{{\mathcal S}^\prime}
 \end{align*}
 is a spherical function on $H^0_1({\mathbf R})$. 
Furthermore, \cite[X. Proposition 3.2]{hel62} proves that
\begin{align}
  \int_{{\rm SU}_2({\mathbf R})} \Omega(uh)  {\rm d}u = \Omega(h).  \label{eq:sph}
\end{align}

Let $\rho(u_1, u_2, a, \varepsilon)$ be as in Lemma \ref{l:cptint}.  
Then Lemma \ref{l:cptint} shows that 
\begin{align*}
         \int_{H^0_1({\mathbf R})}  \Omega(h_\infty) {\rm d} h_\infty  
    =&  \pi \sum_{\varepsilon=\pm 1}   
       \int_{{\rm SU}_2({\mathbf R})/\{\pm 1\}}
       \int_{{\rm SU}_2({\mathbf R})/\{\pm 1\}}
       \int^1_0
       \Omega(\varrho(u_1, u_2, a, \varepsilon)) 
       \frac{(a-a^{-1})^2 {\rm d}a}{a} {\rm d}u_1 {\rm d}u_2  \\
    =&  \frac{\pi}{2}    
       \int_{{\rm SU}_2({\mathbf R}) }
       \int^1_0
       \Omega(\varrho(u_1, 1_2, a, 1)) 
       \frac{(a-a^{-1})^2 {\rm d}a}{a} {\rm d}u_1  \\
    =&  \frac{\pi}{2}    
       \int^1_0
       \Omega(\varrho( 1_2, 1_2, a, 1)) 
       \frac{(a-a^{-1})^2 {\rm d}a}{a}.
\end{align*}
In the last equality, we used (\ref{eq:sph}).
Hence it suffices to compute the following integral:
\begin{align*}
     &  \frac{\pi}{2}
       \int^1_0       
       \Omega(\varrho( 1_2, 1_2, a, 1)) 
       \frac{(a-a^{-1})^2 {\rm d}a}{a}   \\
   =&   \frac{\pi}{2}
       \int^1_0 \frac{(a-a^{-1})^2 {\rm d}a}{a}      
     \int_{{\mathbf X}_\infty}  {\rm d}x
          \int^\infty_0  \frac{{\rm d} t}{t}      \\
     &  \quad
             \left\langle
                       \left\langle  \varphi_\infty \left( \begin{pmatrix}  a^{-\frac{1}{2}}  & 0 \\  0 &  a^{\frac{1}{2}}   \end{pmatrix}  x \right),  
                                               W_{\sigma, \infty}\left( \begin{pmatrix} t &  0 \\ 0 & 1 \end{pmatrix}     
                                                                                   \begin{pmatrix}  a^{\frac{1}{2}}  & 0 \\  0 &  a^{-\frac{1}{2}}   \end{pmatrix} \right)   \right\rangle_{\mathcal W},  
                                   \overline{ \left\langle  \varphi_\infty(x),    W_{\sigma, \infty}\left( \begin{pmatrix} t &  0 \\ 0 & 1 \end{pmatrix}   \right)      \right\rangle_{\mathcal W} }
              \right\rangle_{\mathcal L} \\
    =& \frac{\pi}{2}I_n. 
\end{align*}
Then Conjecture \ref{conj:arint} shows that
\begin{align*}
   {\mathcal Z}_\infty(\varphi^0_\infty, f^0_\infty) = \frac{\pi}{2} C_\infty I_n 
   =&  (-1)^{\frac{n}{2}} 
      \cdot  2^{-2n-10} 
      \cdot  \pi^{-n-1} 
      \cdot   (2n+3)
      \cdot n!  \\
   =&   \frac{ (-1)^{\frac{n}{2}} 2^{-n-9} (2n+3) }{n+1} 
        \cdot \frac{  L(1, {\rm As}^+(\pi_\infty)) }{ \Gamma_{\mathbf R}(2) \Gamma_{\mathbf R}(4)  }.   
\end{align*}
This proves the lemma. 
\end{proof}

\section{Bessel period formula}\label{sec:Bessel}

In this section, we give an explicit Bessel period formula for HST lifts. 
See Theorem \ref{th:bessel} for the result. 
We recall a definition of Bessel periods in Section \ref{s:defbessel}.  
The Bessel periods are  torus integrals which are determined by an embedding 
${\rm Res}_{K/{\mathbf Q}} {\mathbb G}_{m,/K} \to {\rm GSp}_4$ 
for each imaginary quadratic field $K$. 
We give an embedding which is suitable for our computation in Section \ref{sec:chS},   
and an explicit Bessel period formula  is given  in Theorem \ref{th:bessel} assuming $K=F$  
after reducing the computation to computations of local integrals in Section \ref{sec:exBess}. 
In Section \ref{sec:claBess}, we give a relation between Bessel periods and Fourier coefficients of HST lifts in an explicit manner. 
See Corollary \ref{cor:Bess} for the result,   
   which proves that $\theta^\ast_{  \widetilde{\mathbf f}^\dag }$ has an integral Fourier coefficient 
   and which gives a criterion  for the non-vanishing of Fourier coefficients of HST lifts  modulo a prime.

\subsection{Definition of Bessel periods}\label{s:defbessel}
We introduce the definition of $S$-th Bessel periods according to \cite{fu93}.
Define an algebraic group $T_S$ over ${\mathbf Q}$ by
\begin{align}
 T_S = \{ g\in {\rm GL}_2  \ |  \ {}^{\rm t}gSg = (\det g) S \}.  \label{torusS}
\end{align}
We consider $T_S$ as a subgroup of ${\rm GSp}_4$ by the following embedding: 
\begin{align*}
  g\mapsto \begin{pmatrix}  g & 0_2  \\ 0_2 & (\det g)  {}^{\rm t}g^{-1} \end{pmatrix}. 
\end{align*}
Define a subgroup $R$ of ${\rm GSp}_4$ to be 
\begin{align*}
  R= T_SU = T_S \rtimes U. 
\end{align*}
For each character $\phi:  {\mathbf A}^\times  T_S({\mathbf Q}) \backslash T_S({\mathbf A}) \to {\mathbf C}^\times$ 
    and each Siegel cusp form ${\mathcal F}: {\rm GSp}_4({\mathbf A}) \to {\mathcal L}_\lambda({\mathbf C})$,  
we define the Bessel periods ${\mathbf B}_{{\mathcal F},S,\phi}$ to be a function ${\mathbf B}_{{\mathcal F},S,\phi}: {\rm GSp}_4({\mathbf A}) \to {\mathcal L}_\lambda({\mathbf C})$ by 
\begin{align*}
  {\mathbf B}_{{\mathcal F},S,\phi}(g) 
  = \int_{{\mathbf A}^\times R( {\mathbf Q} ) \backslash R( {\mathbf A} ) } {\mathcal F}(rg)\phi\otimes\psi_S(r) {\rm d}r.  
\end{align*}
By the definition, we have
\begin{align}\label{BessWhitt}
  {\mathbf B}_{{\mathcal F},S,\phi}(g) 
   = \int_{{\mathbf A}^\times T_S({\mathbf Q})\backslash T_S({\mathbf A}) } 
          {\mathbf W}_{{\mathcal F},S}(tg)\phi(t) {\rm d}t,  
\end{align}
where ${\mathbf W}_{{\mathcal F},S}$ is the $S$-th adelic Fourier coefficient of ${\mathcal F}$ which is introduced in (\ref{def:FWcoeff}).

\subsection{Choice of $S$}\label{sec:chS}

In this subsection, we write Bessel periods of HST lifts  
   in terms of automorphic forms ${\mathbf f}^\dag_{\mathcal R}$ on $H^0_1({\mathbf A})$  
   for ${\mathcal R} \subset \Sigma_1$    
   by making a choice of $S=S_x$ for some $x\in {\mathbf X}$.   
See Lemma \ref{lem:bes} for the formula.  
 
Let $K$ be an imaginary quadratic field with the absolute discriminant $\Delta_K>0$.  
Suppose that the following condition:
\begin{itemize}
\item[(FK)]  For some $t_K\in F$, we have ${\rm Nr}_{F/{\mathbf Q}} (t_K) = \Delta_K$  
\end{itemize}
We fix  $t_K \in F$ satisfying (FK).   
Define
\begin{align*}
   \vartheta^\circ_K  = \frac{1}{2}  \begin{pmatrix}  t_K - \bar{t_K}  &  t_K+\bar{t_K}   \\
                                                                     -(t_K+\bar{t_K})  & -(t_K-\bar{t_K})   \end{pmatrix}. 
\end{align*}
Since we  have  
\begin{align*}
   (\vartheta^\circ_K)^2 = - {\rm Nr}_{F/{\mathbf Q}} (t_K) = - \Delta_K,   
\end{align*}
we have an embedding  $\iota:K^\times  \to H_0({\mathbf Q}) $ such that 
\begin{align*}
  \iota( a + b \sqrt{-\Delta_K} ) = \varrho(a + b \vartheta^\circ_K, {\rm Nr}_{K/{\mathbf Q}}  (a + b \sqrt{-\Delta_K})  ),   
  \quad (a, b \in {\mathbf Q}).   
\end{align*}
Define 
\begin{align*}
   \vartheta_K 
   = \begin{cases}  \vartheta^\circ_K,   &  (\Delta_K \equiv 1, 2 \ {\rm mod} \ 4),  \\
                               \frac{  1 + \vartheta^\circ_K}{2},  &  (\Delta_K \equiv 3 \ {\rm mod} \ 4)
      \end{cases}   ; \quad 
   x = (1_2, \vartheta_K); \quad 
   S =S_x.
\end{align*}

Define a quaternion algebra $E$ to be $F\otimes_{\mathbf Q} K$, then we extend 
$\iota:K^\times \to H_0({\mathbf Q})$ 
to $\widetilde{\iota}: E^\times \to  H_0({\mathbf Q})$ 
by 
\begin{align*}
   \widetilde{\iota}(a+b\sqrt{-\Delta_K})  
    =  a\otimes 1_2 +  b\otimes \vartheta_K,  
     \quad (a, b \in F). 
\end{align*} 
Hereafter we write $\widetilde{\iota}$ by $\iota$ for the simplicity. 
Let $H_x$ be an algebraic subgroup of $H$ which fixes $x \in {\mathbf X}$.

The following lemma easily follows from the definitions. (See also \cite[Lemma 4.1]{hn17}.)

\begin{lem}\label{fixlem}
{\itshape 
Define an algebraic group $E^\times_0 $ over ${\mathbf Q}$  to be 
       $ \left\{  t \in E^\times \ |\  {\rm Nr}_{E/K}(t) := \iota(t)\iota(t)^\ast \in {\mathbf Q} \right\}$. 
         Then, the map $\iota: E^\times \to H_0$ induces an isomorphism   
         \begin{align*}
             \iota: \begin{CD}  E^\times_0/F^\times   @>\sim>>  H_x \end{CD}.  
         \end{align*}
}
\end{lem}

For $t\in K$, we find a matrix $\Psi_K(t) \in {\rm M}_2({\mathbf Q})$ such  that
\begin{align*}
   \iota(t)\cdot x = x \Psi_K(t).   
\end{align*}
Hence, we find that 
\begin{align}\label{iotaPsi}
   \iota(t) \cdot x =  x \Psi_K( {\rm Nr}_{E/K}(t) ), \quad (t \in E^\times).
\end{align}
Then, $\Psi_K$ induces an isomorphism $\Psi_K:K^\times \cong T_S \hookrightarrow {\rm GSp}_4$.  
We define $j: E^\times \to {\rm GSp}_4$ by
\begin{align*}
    j(t) =  \begin{pmatrix} \Psi_K({\rm Nr}_{E/K}(t))  & 0_2  \\ 0_2 &  \det \Psi_K({\rm Nr}_{E/K}(t)) \cdot  {}^{\rm t}\Psi_K({\rm Nr}_{E/K}(t))^{-1}  \end{pmatrix},   
     \quad (t\in E^\times).
\end{align*}

Let $\phi:  {\mathbf A}^\times K^\times \backslash K^\times_{\mathbf A} \to {\mathbf C}^\times$ be a finite-order Hecke character. 
The $S$-th Bessel period ${\mathbf B}_{S,\phi} := {\mathbf B}_{\theta(\widetilde{\varphi}, \widetilde{\mathbf f}^\dag), S, \phi}$ in (\ref{BessWhitt}) is given by 
\begin{align}
  {\mathbf B}_{S,\phi}(g) = \int_{[E^\times/ E^\times_0]} {\mathbf W}_{\theta(\varphi, \widetilde{\mathbf f}^\dag), S} (j(t)g) \phi({\rm Nr}_{E/K}(t)) {\rm d} t.  
  \label{E:Bessel.HST}
\end{align}

The following lemma describes ${\mathbf B}_{S,\phi}$  in terms of automorphic forms on $H^0_1({\mathbf A})$. 
This is necessary to reduce a computation of ${\mathbf B}_{S,\phi}$ to a computation of local integrals 
    ${\mathcal B}^j_v(\varphi, \phi)$ on $H^0_1({\mathbf Q}_v)$ for each $v\in \Sigma_{\mathbf Q}, j=0,1$ in the next subsection. 

\begin{lem}\label{lem:bes}
{\itshape 
We have 
\begin{align*}
        {\mathbf B}_{S,\phi}(g)    
=& 2^{1-\sharp \Sigma_1} \sum_{  {\mathcal R} \subset \Sigma_1}
                 \int_{[E^\times/ F^\times]} {\rm d} t  \phi({\rm Nr}_{E/K}(t))
           \int_{H_x({\mathbf A}) \backslash H^0_1({\mathbf A}) }        {\rm d}h
          \left\langle 
                     \omega_V( g)\varphi( h^{-1} x ),
           {\mathbf f}^\dag_{ {\mathcal R}} (  \iota(t)    w_{0,{\mathcal R}}   h^c_{{\mathcal R}}  h^{ {\mathcal R}}  )      
          \right\rangle_{\mathcal W}  \\
    & \quad \quad \quad \quad \quad \quad   
        +  \left\langle 
                     \tau_{2n+2}(w_0) \omega_V( g)\varphi( h^{-1}  x ),
                    {\mathbf f}^\dag_{ {\mathcal R} } (  \iota(t)    w_{0,{\mathcal R}\triangle\{\infty\}}   h^c_{{\mathcal R}\triangle\{\infty\}}  h^{ {\mathcal R}\triangle\{\infty\} }  )      
          \right\rangle_{\mathcal W}. 
\end{align*}
}
\end{lem}
\begin{proof}
We easily find that $\nu(\iota(t)) = \nu(j(t))$ for $t\in E^\times$. 
For $g\in {\rm Sp}_4({\mathbf A})$, 
(\ref{thetaFC}) shows that 
\begin{align*}
      {\mathbf B}_{S,\phi}(g)  
  =&   \int_{ H_x({\mathbf A})  \backslash  H_1({\mathbf A}) } {\rm d}h
        \int_{ [H_x] } {\rm d}s
        \int_{[E^\times/ E^\times_0]} {\rm d}t   \\
     & \quad 
            \left\langle \omega_V( g) \widetilde{\varphi} (\iota(t)^{-1} h^{-1}   x \Psi_K({\rm Nr}_{E/K}(t))), \widetilde{\mathbf f}^\dag(  sh \iota(t))  \right\rangle_{\widetilde{\mathcal W}} 
            \phi({\rm Nr}_{E/K}(t)).   
\end{align*}
By using (\ref{iotaPsi}) and changing variable $h$ by $\iota(t) h \iota(t)^{-1}$, we find that  
\begin{align*}
      {\mathbf B}_{S,\phi}(g)  
  =&    \int_{ H_x({\mathbf A})  \backslash  H_1({\mathbf A}) } {\rm d}h
        \int_{ [H_x] } {\rm d}s
        \int_{[E^\times/ E^\times_0]} {\rm d}t  
            \left\langle \omega_V( g) \widetilde{\varphi} (  h^{-1}  x  ), \widetilde{\mathbf f}^\dag(  s \iota(t) h )  \right\rangle_{\widetilde{\mathcal W}} 
            \phi({\rm Nr}_{E/K}(t)).  
\end{align*}
Since $\varphi\circ {\rm Nr}_{E/K}$ is trivial on $E^\times_0$, 
 Lemma \ref{fixlem} proves the following identity: 
\begin{align*}
      {\mathbf B}_{S,\phi}(g)  
  =&   \int_{ H_x({\mathbf A})  \backslash  H_1({\mathbf A}) } {\rm d}h
        \int_{[E^\times/ F^\times]} {\rm d}t  
         \left\langle \omega_V( g) \widetilde{\varphi} ( h^{-1}   x ), \widetilde{\mathbf f}^\dag(  \iota(t) h )  \right\rangle_{\widetilde{\mathcal W}}
            \phi( {\rm Nr}_{E/K}(t)) \\
  =&   \int_{ H_x({\mathbf A})  \backslash  H^0_1({\mathbf A}) } {\rm d}h
        \int_{\mu_2({\mathbf A})} {\rm d}\tau
        \int_{[E^\times/ F^\times]} {\rm d}t  
         \left\langle \omega_V( g) \widetilde{\varphi} ( h^{-1} \tau   x ), \widetilde{\mathbf f}^\dag(  \iota(t)   \tau h )  \right\rangle_{\widetilde{\mathcal W}}
            \phi( {\rm Nr}_{E/K}(t)).  
\end{align*}
The group $\mu_2({\mathbf Q}_v)$ is realized in $H({\mathbf Q}_v)$ as follows:
\begin{align*}
   \mu_2({\mathbf Q}_v)  = \left\{ \varrho( 1_2, 1), \varrho(w_{0,v}, 1) {\mathbf t}_v \right\},  
      \quad \left( w_{0,v}=\begin{pmatrix} 0 & 1 \\ -1 & 0 \end{pmatrix}  \right). 
\end{align*}
We write $w_{0,v}{\mathbf t}_v$ as $\varrho(w_{0,v}, 1) {\mathbf t}_v$ for the simplicity.
We note that
\begin{align*}
  & \int_{ H_x({\mathbf A})  \backslash  H^0_1({\mathbf A}) } {\rm d}h
   \int_{\mu_2({\mathbf A})} \left\langle \omega_V( g) \widetilde{\varphi} (  h^{-1} \tau   x ), \widetilde{\mathbf f}^\dag(  \iota(t)  \tau h  )  \right\rangle_{\widetilde{\mathcal W}} {\rm d}\tau   \\
 =& \frac{1}{2^{\sharp \Sigma_1}} 
       \int_{ H_x({\mathbf A})  \backslash  H^0_1({\mathbf A}) } {\rm d}h 
       \sum_{{\mathcal R}\subset \Sigma_1}
       \left\langle \omega_V( g) \widetilde{\varphi} (  h^{-1} w_{0, {\mathcal R}} {\mathbf t}_{\mathcal R}     x ),
           \widetilde{\mathbf f}^\dag(  \iota(t)     w_{0, {\mathcal R}} {\mathbf t}_{\mathcal R} h ) \right\rangle_{\widetilde{\mathcal W}}  
\end{align*}
By the definition (\ref{def:extf}) of the extension of automorphic forms on $H^0({\mathbf A})$,  
     we also find that
\begin{align*}
      \sum_{{\mathcal R}\subset \Sigma_1}
           \widetilde{\mathbf f}^\dag(  \iota(t)     w_{0, {\mathcal R}} {\mathbf t}_{\mathcal R} h )   
   =& \sum_{{\mathcal R}\subset \Sigma_1}
           \widetilde{\mathbf f}^\dag(  \iota(t)   w_{0, {\mathcal R}}  h^c_{\mathcal R} h^{\mathcal R} {\mathbf t}_{\mathcal R}  )   \\
   =&  \sum_{{\mathcal R}\subset \Sigma_1}
           \left( {\mathbf f}^\dag_{\mathcal R}(  \iota(t)   w_{0, {\mathcal R}}  h^c_{\mathcal R} h^{\mathcal R}   )
                 +{\mathbf f}^\dag_{\Sigma_1\backslash {\mathcal R}}(  \iota(t)^c  w_{0, {\mathcal R}} h_{\mathcal R}  h^{{\mathcal R},c}  ),     \right.  \\
     &     \quad \quad \quad \quad 
          \left.    
              {\mathbf f}^\dag_{\mathcal R \triangle \{\infty\}}(  \iota(t)   w_{0, {\mathcal R}}  h^c_{\mathcal R} h^{\mathcal R}   )
                 +{\mathbf f}^\dag_{(\Sigma_1\backslash {\mathcal R} ) \triangle \{\infty\}}(  \iota(t)^c  w_{0, {\mathcal R}} h_{\mathcal R}  h^{{\mathcal R},c}  )   \right) 
\end{align*}
Since 
\begin{align*}
      w_0 {\mathbf t} \cdot 1_2 = -1_2, \quad 
       w_0 {\mathbf t} \cdot \vartheta_K 
       +{}^{\rm t}( w_0 {\mathbf t} \cdot \vartheta_K )
       = -(\vartheta_K+ {}^{\rm t}\vartheta_K),  
\end{align*}
 the definition of $\widetilde{\varphi}$ shows that
\begin{align*}
    \widetilde{\varphi}(  h^{-1}  w_{0, {\mathcal R}} {\mathbf t}_{\mathcal R}   x ) 
    = \widetilde{\varphi}(   h^{-1}    x ).
 \end{align*}
By using $(-1)^{n+1}\delta(\infty)=1$, we compute 
\begin{align*}
       &  \sum_{{\mathcal R}\subset \Sigma_1} 
            \left\langle \omega_V( g) \widetilde{\varphi} (   h^{-1} w_{0, {\mathcal R}} {\mathbf t}_{\mathcal R}    x ),       
                          \widetilde{\mathbf f}^\dag(  \iota(t)     w_{0, {\mathcal R}} {\mathbf t}_{\mathcal R} h )  \right\rangle_{\widetilde{\mathcal W}} \\   
   = &    \sum_{\mathcal R \subset \Sigma_1}
              \left\langle  \omega_V( g)\varphi( h^{-1}  x ),  
                            {\mathbf f}^\dag_{\mathcal R}(  \iota(t)   w_{0, {\mathcal R}}  h^c_{\mathcal R} h^{\mathcal R}   )
                                +{\mathbf f}^\dag_{\Sigma_1\backslash {\mathcal R}}(  \iota(t)^c  w_{0, {\mathcal R}} h_{\mathcal R}  h^{{\mathcal R},c}  )
                 \right\rangle_{\mathcal W}   \\   
      & \quad \quad   + \left\langle  \tau_{2n+2}(w_0) \omega_V( g)\varphi( h^{-1}    x ), 
                               {\mathbf f}^\dag_{\mathcal R \triangle \{\infty\}}(  \iota(t)   w_{0, {\mathcal R}}  h^c_{\mathcal R} h^{\mathcal R}   )
                                   +{\mathbf f}^\dag_{(\Sigma_1\backslash {\mathcal R} ) \triangle \{\infty\}}(  \iota(t)^c  w_{0, {\mathcal R}} h_{\mathcal R}  h^{{\mathcal R},c}  ) 
                         \right\rangle_{\mathcal W}. 
\end{align*}
Since 
\begin{align*}
    w_0\iota(t)w^{-1}_0 = \iota(t^c),  
\end{align*} 
we find that  for $ {\mathcal R} \subset \Sigma_1$
\begin{align*}
       {\mathbf f}^\dag_{\Sigma_1\backslash {\mathcal R}} (\iota(t)^c  w_{0,{\mathcal R}} h_{\mathcal R} h^{{\mathcal R},c}   )         
=& {\mathbf f}^\dag_{\Sigma_1\backslash {\mathcal R}} ( w_0 \iota(t)^c  w_{0,{\mathcal R}} h_{\mathcal R} h^{{\mathcal R},c}   )  \\
=& {\mathbf f}^\dag_{\Sigma_1\backslash {\mathcal R}} (  \iota(t)   w_{0,{\mathcal R}^c}  h^c_{{\mathcal R}^c} h^{{\mathcal R}^c}   )
\end{align*}
where ${\mathcal R}^c=\Sigma_{\mathbf Q}\backslash {\mathcal R}$. 
By making the change of variable $h$ by $w_{0,v}{\mathbf t}_v h {\mathbf t}_v w_{0,v}$ at each $v\not\in \Sigma_1$, we have
\begin{align}\label{eq:RRc}
\begin{aligned}
  &   \sum_{ {\mathcal R} \subset \Sigma_1} 
         \int_{H_x({\mathbf A}) \backslash H^0_1({\mathbf A}) }  
                    {\mathbf f}^\dag_{\Sigma_1\backslash {\mathcal R}} (\iota(t)^c w_{0,{\mathcal R}} h_{\mathcal R} h^{ {\mathcal R}, c} )       
            {\rm d}h  \\
=  & \sum_{ {\mathcal R} \subset \Sigma_1 } 
        \int_{H_x({\mathbf A}) \backslash H^0_1({\mathbf A}) }  
                     {\mathbf f}^\dag_{\Sigma_1\backslash {\mathcal R}} (  \iota(t) w_{0,{\mathcal R}^c} h^c_{{\mathcal R}^c} h^{ {\mathcal R}^c}  )        
           {\rm d}h  \\
=  &  \sum_{ {\mathcal R} \subset \Sigma_1}
           \int_{H_x({\mathbf A}) \backslash H^0_1({\mathbf A}) }  
           {\mathbf f}^\dag_{ {\mathcal R}} (  \iota(t)    w_{0,{\mathcal R}}   h^c_{{\mathcal R}}  h^{ {\mathcal R}}  )      
       {\rm d}h. 
\end{aligned}
\end{align}
In the same way, we also have 
\begin{align}\label{eq:RRc2}
\begin{aligned}
  & \sum_{{\mathcal R} \subset \Sigma_1} 
         \int_{H_x({\mathbf A}) \backslash H^0_1({\mathbf A}) }  
                    {\mathbf f}^\dag_{(\Sigma_1\backslash {\mathcal R})\triangle \{\infty\} } (\iota(t)^c w_{0,{\mathcal R}} h_{\mathcal R} h^{ {\mathcal R}, c} )       
            {\rm d}h  \\
=&   \sum_{{\mathcal R} \subset \Sigma_1} 
         \int_{H_x({\mathbf A}) \backslash H^0_1({\mathbf A}) }  
                    {\mathbf f}^\dag_{\mathcal R\triangle \{\infty\}} (\iota(t) w_{0,{\mathcal R}} h^c_{\mathcal R} h^{{\mathcal R}}  )       
            {\rm d}h.  
\end{aligned}
\end{align}
Then (\ref{eq:RRc}) and (\ref{eq:RRc2}) show that 
\begin{align*}
     &       \sum_{{\mathcal R}\subset \Sigma_1} 
            \int_{H_x({\mathbf A}) \backslash H^0_1({\mathbf A}) }        {\rm d}h
            \left\langle \omega_V( g) \widetilde{\varphi} (   h^{-1} w_{0, {\mathcal R}} {\mathbf t}_{\mathcal R}    x ),       
            \widetilde{\mathbf f}^\dag(  \iota(t)    w_{0, {\mathcal R}} {\mathbf t}_{\mathcal R}  h )    \right\rangle_{\widetilde{\mathcal W}} \\   
=& 2 \sum_{  {\mathcal R} \subset \Sigma_1}
           \int_{H_x({\mathbf A}) \backslash H^0_1({\mathbf A}) }        {\rm d}h
          \left\langle 
                     \omega_V( g)\varphi( h^{-1} x ),
           {\mathbf f}^\dag_{ {\mathcal R}} (  \iota(t)    w_{0,{\mathcal R}}   h^c_{{\mathcal R}}  h^{ {\mathcal R}}  )      
          \right\rangle_{\mathcal W}  \\
    & \quad \quad \quad \quad \quad \quad   
        +   \left\langle 
                     \tau_{2n+2}(w_0) \omega_V( g)\varphi( h^{-1}  x ),
                    {\mathbf f}^\dag_{ {\mathcal R}\triangle\{\infty\} } (  \iota(t)    w_{0,{\mathcal R}}   h^c_{{\mathcal R}}  h^{ {\mathcal R} }  )      
          \right\rangle_{\mathcal W}. 
\end{align*}
This proves the lemma.
\end{proof}

\subsection{Explicit Bessel period formula}\label{sec:exBess}

In this subsection, 
     we deduce an explicit Bessel period formula assuming $K=F$. 
We firstly write ${\mathbf B}_{S,\phi}$ as the product of the local integrals on $H^0_1({\mathbf Q}_v)$ for $v\in \Sigma_{\mathbf Q}$ in Lemma \ref{lem:besprod}.     
The explicit Bessel period formula is given in Theorem \ref{th:bessel}.

Assume that $K=F$. Then, we choose $t_F$ to be $\sqrt{-\Delta_F}$.  
Define 
\begin{align*}
   \vartheta^\prime_F
   = \begin{cases}  \frac{1}{2}\sqrt{-\Delta_F},   &  (\Delta_F \equiv 1, 2 \ {\rm mod} \ 4),  \\
                                \frac{  1 +  \sqrt{-\Delta_F} }{2},  &  (\Delta_F \equiv 3 \ {\rm mod} \ 4).
       \end{cases}
\end{align*}
Note that
\begin{align*}
x= \left(1_2,  \begin{pmatrix}  \vartheta^\prime_F & 0 \\ 0 & \overline{\vartheta^\prime_F}   \end{pmatrix}  \right),  \quad
S_x= \begin{pmatrix} 1 &  \frac{1}{2}{\rm Tr}_{F/{\mathbf Q} }(\vartheta^\prime_F)  \\ 
                                   \frac{1}{2} {\rm Tr}_{F/{\mathbf Q}}(\vartheta^\prime_F) &  {\rm Nr}_{F/{\mathbf Q}}(\vartheta^\prime_F) \end{pmatrix}, 
    \quad  \det S_x = \frac{\Delta_F}{4}.  
\end{align*}
We also see that the image of $[E^\times/F^\times]$ via $\iota$ is represented by the following set: 
\begin{align}
   \left\{  \begin{pmatrix}  t & 0  \\ 0 & 1  \end{pmatrix} \ | \ t \in [F^\times] \right\}.  \label{eq:iota(t)}
\end{align}

The computation of  ${\mathbf B}_{S,\phi}$ is reduced to the computation of the following local integrals, which are so called 
local Bessel period integrals.  
For each finite $v\in \Sigma_{\mathbf Q}$ and $g_v\in {\rm Sp}_4({\mathbf Q}_v)$, let 
\begin{align*}
  {\mathcal B}^0_v(\varphi, \phi)(g_v) =& \int_{F^\times_v}  {\rm d}^\times t  
                                    \int_{H_x({\mathbf Q}_v)\backslash H^0_1({\mathbf Q}_v)}  {\rm d}h  
                                      \omega_{V_v} (g_v) \varphi_v(h^{-1}x) W_{{\mathbf f}^\dag, v} \left( \begin{pmatrix} t & 0 \\ 0 & 1 \end{pmatrix}h \right)  \phi_v( t ),     \\
  {\mathcal B}^1_v(\varphi, \phi)(g_v) =&  \int_{F^\times_v}  {\rm d}^\times t  
                                    \int_{H_x({\mathbf Q}_v)\backslash H^0_1({\mathbf Q}_v)}  {\rm d}h  
                                      \omega_{V_v} (g_v) \varphi_v(h^{-1}x) W_{{\mathbf f}^\dag, v} \left(\begin{pmatrix} t & 0 \\ 0 & 1 \end{pmatrix} w_0 h \right)  \phi_v( t^c ).  
\end{align*}
Similarly, define, for  $g_\infty \in {\rm Sp}_4({\mathbf R})$,
\begin{align*}
  {\mathcal B}^0_\infty(\varphi, \phi)(g_\infty) =& \int_{F^\times_\infty}  {\rm d}^\times t  
                                    \int_{H_x({\mathbf Q}_\infty)\backslash H^0_1({\mathbf Q}_\infty)}  {\rm d}h  
                                     \langle \omega_{V_\infty} (g_\infty) \varphi_v(h^{-1}x),  
                                                  W_{{\mathbf f}^\dag, \infty} \left(\begin{pmatrix} t & 0 \\ 0 & 1 \end{pmatrix}h\right) \rangle_{\mathcal W} \phi_\infty( t ),     \\
  {\mathcal B}^1_\infty(\varphi, \phi)(g_\infty) =&  \int_{F^\times_\infty}  {\rm d}^\times t  
                                    \int_{H_x({\mathbf Q}_\infty)\backslash H^0_1({\mathbf Q}_\infty)}  {\rm d}h  
                                     \langle \omega_{V_\infty} (g_\infty) \varphi_v(h^{-1}x),  
                                                      W_{{\mathbf f}^\dag, \infty} \left(\begin{pmatrix} t & 0 \\ 0 & 1 \end{pmatrix} w_0 h \right) \rangle_{\mathcal W}  \phi_\infty( t^c ).  
\end{align*}

\begin{lem}\label{lem:besprod}
{\itshape 
We have
\begin{align*}
    {\mathbf B}_{S,\phi}(g)
   = 2^{2-\sharp \Sigma_1} \prod_{v\in \Sigma_1 }
                       \left( {\mathcal B}^0_v(\varphi,\phi)(g_v)  + \delta(v) {\mathcal B}^1_v(\varphi, \phi)(g_v) \right)
           \prod_{v \in \Sigma_{\mathbf Q} \backslash \Sigma_1 }
                       {\mathcal B}^0_v(\varphi, \phi)(g_v).
 \end{align*}
 }
\end{lem}
\begin{proof}
By Lemma \ref{lem:bes}, (\ref{eq:iota(t)}) and the definition (\ref{eq:whittdef}) of the Fourier-Whittaker expansion, ${\mathbf B}_{S,\phi}$ is equal to 
\begin{align*}
   & 2^{1-\sharp \Sigma_1} \sum_{\mathcal R\subset \Sigma_1}
          \int_{F^\times_{\mathbf A}} \phi( t )  {\rm d}^\times t  
         \int_{H_x({\mathbf A})\backslash H^0_1({\mathbf A})}  {\rm d}h  \\
    & \quad \quad     \langle  \omega_V( g)\varphi( h^{-1}  x ),   
                           W_{{\mathbf f}^\dag, \mathcal R}(  \iota(t) w_{0, {\mathcal R}}  h^c_{\mathcal R} h^{\mathcal R} )   \rangle_{\mathcal W}    \\
     &  \quad \quad \quad \quad          +\langle \tau_{2n+2}(w_0) \omega_V( g)\varphi( h^{-1}  x ),   
                           W_{{\mathbf f}^\dag, \mathcal R}(  \iota(t) w_{0, {\mathcal R}\triangle\{\infty\} }  h^c_{\mathcal R\triangle\{\infty\} } h^{\mathcal R\triangle\{\infty\} } )   \rangle_{\mathcal W}.
\end{align*}
The definition of Whittaker functions (\ref{def:whR})    
     and Lemma \ref{l:minwhitt} show that
\begin{align*}
    & \tau_{2n+2}(w_0)  
    W_{{\mathbf f}^\dag,\mathcal R}(  \iota(t) w_{0, {\mathcal R}\triangle\{\infty\} }  h^c_{\mathcal R\triangle\{\infty\} } h^{\mathcal R\triangle\{\infty\} }  )   \\
  =& (-1)^{n+1}
     \times
     \begin{cases} \xi_\infty       W_{{\mathbf f}^\dag,\mathcal R}(  \iota(t) 
                                                w_{0, {\mathcal R}\cup\{\infty\} }
                                                  h^c_{\mathcal R } h^c_\infty h^{\mathcal R\cup\{\infty\} }  ),    &   (\infty\not\in {\mathcal R}),  \\
                            \xi_\infty       W_{{\mathbf f}^\dag,\mathcal R}(  \iota(t) 
                                                w_{0, {\mathcal R}\backslash \{\infty\} }
                                                  h^c_{\mathcal R\backslash \{\infty\} } h_\infty  h^{\mathcal R\backslash \{\infty\} }  ),    &   (\infty\in {\mathcal R}),                 
    \end{cases} \\
   =& \delta(\infty)  
     \times
     \begin{cases}     W_{{\mathbf f}^\dag, \mathcal R}(  \iota(t)^c_\infty \iota(t)_{\rm fin} 
                                                w_{0, {\mathcal R}\cup\{\infty\} }
                                                  h^c_{\mathcal R } h_\infty h^{\mathcal R\cup\{\infty\} }  ),    &   (\infty\not\in {\mathcal R}),  \\
                                W_{{\mathbf f}^\dag, \mathcal R}(  \iota(t)^c_\infty \iota(t)_{\rm fin} 
                                                w_{0, {\mathcal R}\backslash \{\infty\} }
                                                  h^c_{\mathcal R\backslash \{\infty\} } h^c_\infty  h^{\mathcal R\backslash \{\infty\} }  ),    &   (\infty\in {\mathcal R}),                 
    \end{cases} \\
   =& \delta_{{\mathcal R}, \text{fin}} W_{{\mathbf f}^\dag, \text{fin}}(\iota(t)_{\text{fin}} w_{0, {\mathcal R}, \text{fin}} h^c_{{\mathcal R}, \text{fin}}   h^{{\mathcal R}}_{\rm fin}   )  \\
     & \quad \times  
          \begin{cases}  \delta(\infty)   W_{{\mathbf f}^\dag,\infty }(  \iota(t)^c_\infty  w_{0,\infty}   h_\infty   ),    &   (\infty\not\in {\mathcal R}),  \\
                                W_{{\mathbf f}^\dag, \infty }(  \iota(t)_\infty  h_\infty  ),    &   (\infty\in {\mathcal R}),                 
        \end{cases}  \\
   =& \delta_{{\mathcal R}, \text{fin}} W_{{\mathbf f}^\dag, \text{fin}}(\iota(t)_{\text{fin}} w_{0, {\mathcal R}, \text{fin}} h^c_{{\mathcal R}, \text{fin}}   h^{{\mathcal R}}_{\rm fin}   )  \\
     &  \times  W_{{\mathbf f}^\dag, {\mathcal R}\triangle\{\infty\}, \infty}
                 (    \iota(t)_\infty   
                       w_{0,{\mathcal R}\triangle\{\infty\}, \infty    }   
                 (h^c_{{\mathcal R}\triangle\{\infty\}}  h^{{\mathcal R}\triangle\{\infty\}} )_\infty  )       \\
    =& W_{{\mathbf f}^\dag, \mathcal R\triangle\{\infty\}}(  \iota(t) w_{0, {\mathcal R}\triangle\{\infty\}}  h^c_{\mathcal R\triangle\{\infty\}} h^{\mathcal R\triangle\{\infty\}} ).  
 \end{align*}
 This proves that
 \begin{align*}
    {\mathbf B}_{S,\phi}(g)
    =& 2^{2-\sharp \Sigma_1} \sum_{\mathcal R\subset \Sigma_1}
          \int_{F^\times_{\mathbf A}} \phi( t )  {\rm d}t  
         \int_{H_x({\mathbf A})\backslash H^0_1({\mathbf A})}  {\rm d}h  
        \langle  \omega_V( g)\varphi( h^{-1}  x ),   
                           W_{{\mathbf f}^\dag, \mathcal R}(  \iota(t) w_{0, {\mathcal R}}  h^c_{\mathcal R} h^{\mathcal R} )   \rangle_{\mathcal W} \\
   =&2^{2-\sharp \Sigma_1} \sum_{\mathcal R\subset \Sigma_1}
           \prod_{v\in {\mathcal R}}
                      \delta(v) {\mathcal B}^1_v(\varphi, \phi)(g_v)
           \prod_{v \in \Sigma_{\mathbf Q} \backslash {\mathcal R}}
                       {\mathcal B}^0_v(\varphi, \phi)(g_v)  \\
   =&2^{2-\sharp \Sigma_1} \prod_{v\in \Sigma_1}
                       \left( {\mathcal B}^0_v(\varphi,\phi)(g_v)  + \delta(v) {\mathcal B}^1_v(\varphi, \phi)(g_v)   \right)
           \prod_{v \in \Sigma_{\mathbf Q} \backslash \Sigma_1}
                       {\mathcal B}^0_v(\varphi, \phi)(g_v). \qedhere
 \end{align*}
\end{proof}

The computations of each local integrals ${\mathcal B}^j_v (\varphi, \phi)$ $(j=0,1)$  are given in Section \ref{s:besloc}. 
In this subsection,  we describe the Bessel periods formula by using an explicit formula for ${\mathcal B}^j_v (\varphi, \phi)$ $(j=0,1)$. 

We prepare some notation. 
Let $C$ be the positive integer such that $C{\mathcal O}_F$ gives the conductor of $\phi$. 
We assume that the following conditions on $C$: 
\begin{itemize}
 \item[(CF)]  $C$ is prime to $N_F:={\rm l.c.m.}(N, \Delta_F)$ and each prime $v$ of ${\mathbf Q}$ dividing $C$ is split in $F$. 
\end{itemize}
Define   
\begin{align}\label{def:varsig}
\varsigma = \varsigma_\infty \varsigma_{\text{fin}} \in {\rm GL}_2({\mathbf A}); \quad  
   \varsigma_\infty =  ({\rm Im}(\vartheta^\prime_F))^{-1} \cdot  \begin{pmatrix} {\rm Im}(\vartheta^\prime_F) &  -{\rm Re}(\vartheta^\prime_F)  & \\ 0 & 1 \end{pmatrix} ; \quad
   \varsigma_{\text{fin}}  = \begin{pmatrix}  1 & 0 \\ 0 & C  \end{pmatrix}_\text{fin}. 
\end{align}
The definition of $\varsigma_\infty$ immediately shows that
\begin{align}\label{eq:diag}
    x \varsigma_\infty = \left(1_2, \begin{pmatrix} \vartheta^\prime_F  &  0 \\ 0 & \overline{\vartheta^\prime_F}   \end{pmatrix}   \right) \varsigma_\infty
         = \left(1_2, \begin{pmatrix} \sqrt{-1}  &  0 \\ 0 & -\sqrt{-1}   \end{pmatrix} \right).
\end{align}
Then the explicit Bessel period formula for ${\mathbf B}_{S,\phi}$ is given as follows:

\begin{thm}\label{th:bessel}
{\itshape 
Let $C{\mathcal O}_F$ be the conductor of $\phi$. 
Assume that 
\begin{itemize}
\item $n$ is even; 
\item $\pi$ is not conjugate self-dual; 
\item $\delta(\infty)=-1$ and $\delta(v)=1$ for each finite $v\in \Sigma_{\mathbf Q}$; 
\item $\Delta_F$ and ${\rm Nr}_{F/{\mathbf Q}}({\mathfrak N})$ are coprime;
\item the condition {\rm (CF)}.
\end{itemize}

Then we have 
\begin{align}\label{eq:mainbes}
\begin{aligned}
   {\mathbf B}_{S,\phi} \left(\begin{pmatrix}  \varsigma & 0_2  \\ 0_2 & {}^{\rm t}\varsigma^{-1}  \end{pmatrix} \right) 
   = & {\rm vol}({\mathcal U}_{N,1}, {\rm d}h_{\rm fin})  \cdot e^{-4\pi}  \cdot (-\sqrt{-1})\cdot 2^{-1} \cdot C^{-2}   \\
      & \times  
           L(\frac{1}{2}, \pi\otimes \phi) \cdot e(\pi, \phi, \delta)   \prod_{v\mid C}\epsilon(0,\phi^{-1}_v)   
       \times (-1)^{\frac{n}{2}}  \binom{n}{\frac{n}{2}}
          (X^2+Y^2)^{\frac{n}{2}},
\end{aligned}
\end{align}
where
\begin{align*}
    e(\pi,\phi, \delta) 
       = 2^{-\sharp\Sigma_1}(1-\delta(\infty)) \cdot \prod_{v\in \Sigma_1, v<\infty} (1+\delta(v)) 
       \prod_{v\mid N} (1+\epsilon(\frac{1}{2}, \pi_v\otimes\phi_v))
       \prod_{v\in {\mathcal P}} (1+\epsilon(\frac{1}{2}, \pi_v\otimes\phi^{-1}_v)) . 
\end{align*}
}
\end{thm}
\begin{proof}
Lemma \ref{lem:besprod}, Lemma \ref{lem:narbessel} and Lemma \ref{lem:arbessel} show that 
\begin{align*}
   {\mathbf B}_{S,\phi} \left( \begin{pmatrix}  \varsigma & 0_2 \\ 0_2 & {}^{\rm t}\varsigma^{-1}  \end{pmatrix} \right) 
   = & {\rm vol}({\mathcal U}_{N,1}, {\rm d}h_{\rm fin})  \cdot e^{-4\pi}  \cdot (-\sqrt{-1})\cdot 2^{-1} \cdot C^{-2}   \\
      & \times   L(\frac{1}{2}, \pi\otimes \phi) \cdot e(\pi, \phi, \delta)   \prod_{v\mid C}\epsilon(0,\phi^{-1}_v)   \\
      & \times \sum^{n}_{\alpha=0} \sqrt{-1}^{n-\alpha }  \sum_{c\in {\mathbf Z}} (-1)^c \binom{\alpha}{c} \binom{n-\alpha}{\frac{n}{2}-c} \cdot \binom{n}{\alpha} X^\alpha Y^{n-\alpha}. 
\end{align*}
We compute the summation in the above identity.  
If $\alpha$ is odd, the summation in the right-hand side of (\ref{eq:mainbes}) is zero:
\begin{align}\label{oddsum} 
\sum_{c\in {\mathbf Z}} (-1)^c \binom{\alpha}{c} \binom{n-\alpha}{\frac{n}{2}-c} =0.
\end{align}
If $\alpha$ is even, we find that
\begin{align*}
\sum_{c\in {\mathbf Z}} (-1)^c \binom{\alpha}{c} \binom{n-\alpha}{\frac{n}{2}-c} 
=& \frac{\alpha!(n-\alpha)!}{(\frac{n}{2}!)^2}\sum_{c}(-1)^c\binom{\frac{n}{2}}{c}\binom{\frac{n}{2}}{\alpha-c} \\    \displaybreak[0]
=& \frac{\alpha!(n-\alpha)!}{(\frac{n}{2}!)^2}\times \left(\text{the coefficient of $X^\alpha$ in $(1-X)^{\frac{n}{2}} (1+X)^{\frac{n}{2}}$} \right)   \\   \displaybreak[0]
=& \frac{\alpha!(n-\alpha)!}{(\frac{n}{2}!)^2}\times \binom{\frac{n}{2}}{\frac{\alpha}{2}} (-1)^{\frac{\alpha}{2}} \\    \displaybreak[0]
=& (-1)^{\frac{\alpha}{2}} \binom{\frac{n}{2}}{\frac{\alpha}{2}} \binom{ n }{ \frac{n}{2} } \binom{n}{\alpha}^{-1}.
\end{align*}
Hence we see that 
\begin{align*}
       \sum^{n}_{\alpha=0} \sqrt{-1}^{n-\alpha }  \sum_{c\in {\mathbf Z}} (-1)^c \binom{\alpha}{c} \binom{n-\alpha}{\frac{n}{2}-c} \cdot \binom{n}{\alpha} X^\alpha Y^{n-\alpha}  
  =&   (-1)^{\frac{n}{2}}  \binom{n}{\frac{n}{2}}
       \sum^{n}_{ \substack{ \alpha=0  \\ \alpha: \text{even}  }  }   
         \binom{\frac{n}{2}}{\frac{\alpha}{2}} (X^2)^{\frac{\alpha}{2}} (Y^2)^{\frac{n-\alpha}{2}}  \\
  =&   (-1)^{\frac{n}{2}}  \binom{n}{\frac{n}{2}}
          (X^2+Y^2)^{\frac{n}{2}}.
\end{align*}
This proves the theorem. 
\end{proof}

\subsection{A criterion for non-vanishing modulo a prime}\label{sec:claBess}

In this subsection, we give a relation between Fourier coefficients of HST lifts 
and special values of the standard $L$-function of $\pi$ by using Theorem \ref{th:bessel}. 
The relation will give a criterion for the non-vanishing of HST lifts modulo a prime.  
See Corollary \ref{cor:Bess} for the main result in this subsection.

We firstly give a relation of  adelic Fourier coefficients (\ref{def:FWc}) 
 with Fourier coefficients of the classical theta lift $\theta^\ast_{ \widetilde{\mathbf f}^\dag }$.
Let 
\begin{align*}
 \Lambda_2 = \left\{ S= \begin{pmatrix} a & \frac{b}{2}  \\ \frac{b}{2} & c \end{pmatrix} 
                                 | \ a, b, c \in {\mathbf Z}, S \text{ is semi-positive definite}  \right\},  
\end{align*}
and consider the Fourier expansion of $\theta^\ast_{ \widetilde{\mathbf f}^\dag }$:
\begin{align*}
 \theta^\ast_{ \widetilde{\mathbf f}^\dag }(Z)
    = \sum_{S\in \Lambda_2} {\mathbf a}(S) q^S, \quad (q^S= \exp(2\pi\sqrt{-1} {\rm Tr}(SZ))).   
\end{align*}

Let $Z=X+ \sqrt{-1} Y \in {\mathfrak H}_2$ 
          and take $g_\infty \in {\rm Sp}_4({\mathbf R})$ such that $g_\infty\cdot {\mathbf i} =Z$.

\begin{lem}\label{lem:Fou}
{\itshape   
We have the following statements:
\begin{enumerate}
\item If ${\mathbf a}(S)\neq 0$, then $S=S_x$ for some $x\in {\mathbf X}$, where $S_x$ is defined in (\ref{def:Sx}).  
\item We have  
        \begin{align*}
             {\mathbf W}_{\theta(\widetilde{\varphi}, \widetilde{\mathbf f}^\dag),S}(g_\infty)
        =   {\rm vol}( {\mathcal U}_{N,1} , {\rm d}h_{\rm fin})
              \varrho_\lambda(g_\infty, {\mathbf i})^{-1} \cdot  {\mathbf a}(S)q^S.
        \end{align*}
\end{enumerate}
}
\end{lem}
\begin{proof}
The first statement follows from (\ref{wittFC}). See also \cite[Proposition 3.6]{hn17}.

We prove the second statement. 
Then the definition (\ref{def:FWc}) of $W_{\theta(\widetilde{\varphi}, \widetilde{\mathbf f}^\dag),S}$ shows that 
\begin{align*}
   & {\mathbf W}_{\theta(\widetilde{\varphi}, \widetilde{\mathbf f}^\dag),S}(g_\infty)  \\
=& \int_{[U]}  \theta(\widetilde{\varphi}, \widetilde{\mathbf f}^\dag) (ug_\infty) \psi_S( u )    {\rm d}u \\   
=& {\rm vol}( {\mathcal U}_{N,1} , {\rm d}h_{\rm fin})     
      \cdot  \int_{ {\rm M}_2({\mathbf Z}) \backslash {\rm M}_2({\mathbf R}) } 
       \varrho_\lambda\left(  J\left( \begin{pmatrix} 1_2 & T \\ 0_2 & 1_2 \end{pmatrix} g_\infty, {\mathbf i} \right)^{-1}  \right)
        \theta^\ast_{ \widetilde{\mathbf f}^\dag} \left( \begin{pmatrix} 1_2 & T \\ 0_2 & 1_2 \end{pmatrix}  Z  \right) \psi_{{\mathbf Q}, \infty} (-{\rm Tr}(ST) )    {\rm d}T   \\
=& {\rm vol}( {\mathcal U}_{N,1} , {\rm d}h_{\rm fin})
      \cdot  \sum_{S^\prime\in \Lambda_2}
      \int_{ {\rm M}_2({\mathbf Z}) \backslash {\rm M}_2({\mathbf R}) }
          \varrho_\lambda(g_\infty, {\mathbf i})^{-1}
      \cdot   {\mathbf a}(S^\prime) q^{S^\prime} 
        \psi_{{\mathbf Q}, \infty} ({\rm Tr}(S^\prime T ) )  
        \psi_{{\mathbf Q}, \infty} (-{\rm Tr}(ST) ) 
      {\rm d}T   \\
=& {\rm vol}( {\mathcal U}_{N,1} , {\rm d}h_{\rm fin})
      \varrho_\lambda(g_\infty, {\mathbf i})^{-1} \cdot  {\mathbf a}(S)q^S. \qedhere       
\end{align*}
\end{proof}

We prepare some notation 
to give a relation between classical Fourier coefficients and  Bessel periods of HST lifts. 
Let  $x \in {\mathbf X}$ as in Section \ref{sec:exBess} and $S=S_x$. 
Put 
\begin{align*}
   Q_S(X,Y)  =&   \left( (X,Y) S \begin{pmatrix} X \\ Y  \end{pmatrix} \right)^{\frac{n}{2}},  \\ 
   Q^\varsigma_S(X,Y) =&  \varrho_\lambda({}^{\rm t} \varsigma_\infty  )   Q_S(X,Y) (\det \varsigma_\infty)^{-n-4}. 
\end{align*}
Note that
\begin{align*}
          Q^\varsigma_S(X,Y)  
   =&  Q_{S_{x\varsigma_\infty}}(X, Y)  (\det \varsigma_\infty)^{-n-2}
   =  \left(-\frac{\Delta_F}{4} \right)^{\frac{n+2}{2}}  (X^2+Y^2)^{\frac{n}{2}}.
\end{align*}
Define a finite group ${\rm Cl}^{\rm a}_F(C)$ to be  
\begin{align*}
F^\times {\mathbf A}^\times \backslash F^\times_{\mathbf A}/ {\mathbf C}^1  \widehat{\mathcal O}^\times_F(C), 
 \quad  ({\mathbf C}^1 :=\{ z \in {\mathbf C} \ | \  |z|=1  \} \subset F^\times_\infty  ),   
\end{align*}
and denote by ${\rm Cl}^{\rm a}_F(C)^\vee$ the set of characters on ${\rm Cl}^{\rm a}_F(C)$.

\begin{prop}\label{BFC}
{\itshape 
Let $\varsigma\in {\rm GL}_2({\mathbf A})$ as in (\ref{def:varsig})
and write 
\begin{align*}
    \widetilde{\varsigma} = \begin{pmatrix}  \varsigma & 0_2 \\ 0_2 & {}^{\rm t}\varsigma^{-1}  \end{pmatrix} \in {\rm Sp}_4({\mathbf A}).  
\end{align*}
Then, we have 
\begin{align*}
     \langle {\mathbf B}_{S,\phi}( \widetilde{\varsigma} ), Q^\varsigma_S(X,Y)\rangle_{\mathcal L}  
=& [\widehat{\mathcal O}^\times_F : \widehat{\mathcal O}^\times_F(C)]^{-1} 
      {\rm vol}( {\mathcal U}_{N,1} , {\rm d}h_{\rm fin})  e^{-4\pi}  C^{-n-4}             \\
  & \times     \sum_{ [t] \in {\rm Cl}^{\rm a}_F(C)}  
      \langle 
      {\mathbf a}(S^{\gamma_t} ), 
                     Q_{S^{\gamma_t}}(X,Y) \rangle_{\mathcal L} 
       \phi(t).
\end{align*}
}
\end{prop}
\begin{proof}
Let  $\Psi_F$ is defined in Section \ref{sec:chS}, 
which is characterized by the following formula:
\begin{align*}
   \iota(t) x = x \Psi_F(t),
 \quad  \left( \iota(t) =  \begin{pmatrix} t & 0 \\ 0 & 1 \end{pmatrix} \right).  
\end{align*}   
By the definition of $x$, we find  that $\Psi_F: {\rm Res}_{F/{\mathbf Q}}  {\mathbb G}_{m /F} \stackrel{\sim}{\to} T_S$. 
Recall also 
\begin{align*}
j(t) = \begin{pmatrix}  \Psi_F(t) &  0_2 \\ 0_2 & {\rm det }\Psi_F(t) \cdot {}^{\rm t}\Psi_F(t)^{-1}  \end{pmatrix} \in {\rm GSp}_4,   
\quad   (t\in {\rm Res}_{F/{\mathbf Q}}  {\mathbb G}_{m /F}, \det \Psi_F(t) = {\rm Nr}_{F/{\mathbf Q}}(t)).
\end{align*}
By (\ref{E:Bessel.HST}), we find that 
\begin{align*}
    {\mathbf B}_{S,\phi}( \widetilde{\varsigma} )
  =&  \int_{{\mathbf A}^\times T_S({\mathbf Q}) \backslash  T({\mathbf A})}     
       {\mathbf W}_{\theta(  \widetilde{\varphi}, \widetilde{\mathbf f}^\dag ), S}  ( j(t) \widetilde{\varsigma} )  \phi(t)   {\rm d}t   \\
  =&  {\rm vol}( \widehat{\mathcal O}^\times_F(C), {\rm d}t)  
          \sum_{ [t] \in {\rm Cl}^{\rm a}_F(C)  }   
                    {\mathbf W}_{\theta(  \widetilde{\varphi}, \widetilde{\mathbf f}^\dag ), S}  ( j(t) \widetilde{\varsigma}   )  \phi(t).       
\end{align*}
For each $[t] \in {\rm Cl}^{\rm a}_F(C)$, take a representative $t_0\in F^\times_{{\mathbf A}, {\rm fin}}$ of $[t]$.
We write 
\begin{align*}
     \Psi_F(t_0) \varsigma_{\rm fin} =& \alpha_{t, {\rm fin}}   u_t  \in {\rm GL}_2({\mathbf A}_{\rm fin}), 
                                 \quad (\alpha_t \in {\rm GL}_2({\mathbf Q}),  u_t\in {\rm GL}_2(\widehat{\mathbf Z}) ),  \\
         \det\Psi_F(t_0) =& \alpha^\prime_{t, {\rm fin}} u^\prime_t \in   {\mathbf A}^\times_{\rm fin},  
           \quad (\alpha^\prime_t  \in {\mathbf Q}^\times,   u^\prime_t \in \widehat{\mathbf Z}^\times), 
\end{align*}
by the strong approximation theorem.
Put 
\begin{align}\label{def:gt}
  \gamma_t  =   (\alpha^{\prime}_t)^{-\frac{1}{2}} \alpha_t , 
  \quad S^{\gamma_t} = {}^{\rm t}\gamma_t S \gamma_t, 
  \quad g_\infty =   \begin{pmatrix}    \gamma^{-1}_{t,\infty}  \varsigma_\infty  &   0_2  \\   0_2 & {}^{\rm t}\gamma_{t,\infty}  {}^{\rm t}\varsigma^{-1}_\infty   \end{pmatrix}. 
\end{align}
Then (\ref{FCtw}) and  Lemma \ref{lem:Fou} yield that  
\begin{align*}
    {\mathbf W}_{\theta(  \widetilde{\varphi}, \widetilde{\mathbf f}^\dag ), S}  ( j(t_0) \widetilde{\varsigma} )   
 =& {\mathbf W}_{\theta(  \widetilde{\varphi}, \widetilde{\mathbf f}^\dag ), S}  
         \left(  \begin{pmatrix}   \alpha_{t, {\rm fin}} &   0_2  \\   0_2 &   \alpha^\prime_{t, {\rm fin}}   {}^{\rm t} \alpha^{-1}_{t, {\rm fin}}   \end{pmatrix} 
                  \begin{pmatrix}   \varsigma_\infty  &   0_2  \\   0_2 & {}^{\rm t}\varsigma^{-1}_\infty   \end{pmatrix}   \right)   \\
 =& {\mathbf W}_{\theta(  \widetilde{\varphi}, \widetilde{\mathbf f}^\dag ), S^{\gamma_t}}  
         \left(  \begin{pmatrix}    \gamma^{-1}_{t,\infty}  \varsigma_\infty  &   0_2  \\   0_2 & {}^{\rm t}\gamma_{t,\infty}  {}^{\rm t}\varsigma^{-1}_\infty   \end{pmatrix}   \right)   \\
 =& {\rm vol}( {\mathcal U}_{N,1}, {\rm d}h_{\rm fin})  
      \varrho_\lambda( {}^{\rm t}\gamma_{t}  {}^{\rm t} \varsigma^{-1}_\infty  )^{-1}
      {\mathbf a}(S^{\gamma_t} )   
        \exp(  2\pi \sqrt{-1}  {\rm Tr}(  S^{\gamma_t}  ( g_\infty \cdot {\mathbf i})   ) ).         
\end{align*} 
We find that 
\begin{align*}
  {\rm Tr}(  S^{\gamma_t}  ( g_\infty \cdot {\mathbf i})   ) 
=& {\rm Tr}(  {}^{\rm t}\gamma_t  S \gamma_t 
                   \cdot  \sqrt{-1} \gamma^{-1}_t \varsigma_\infty    
                   \cdot   ( {}^{\rm t}\gamma_t {}^{\rm t}\varsigma^{-1}_\infty  )^{-1}   )   \\
=& \sqrt{-1}  {\rm Tr}(  {}^{\rm t}\varsigma_\infty  S \varsigma_\infty )    
=  \sqrt{-1}  {\rm Tr}( S_{x\varsigma_\infty}    )  
= 2\sqrt{-1}. 
\end{align*}
Summarizing the above computations and (\ref{eq:pairequiv}), we obtain 
\begin{align*}
     \langle {\mathbf B}_{S,\phi}( \widetilde{\varsigma} ), Q^\varsigma_S(X,Y)\rangle_{\mathcal L}  
=& [\widehat{\mathcal O}^\times_F : \widehat{\mathcal O}^\times_F(C)]^{-1} 
      {\rm vol}( {\mathcal U}_{N,1} , {\rm d}h_{\rm fin})  e^{-4\pi}               \\
  &\quad  \times     \sum_{ [t] \in {\rm Cl}^{\rm a}_F(C)}  
     \det( \gamma^{-1}_t \varsigma_\infty  )^{n+4}  
     \langle 
      {\mathbf a}(S^{\gamma_t} ), 
       \varrho_\lambda ( {}^{\rm t}\gamma_t  {}^{\rm t}\varsigma^{-1}_\infty   )  
                     Q^\varsigma_S(X,Y) \rangle_{\mathcal L} 
       \phi(t)  \\    
=& [\widehat{\mathcal O}^\times_F : \widehat{\mathcal O}^\times_F(C)]^{-1} 
      {\rm vol}( {\mathcal U}_{N,1} , {\rm d}h_{\rm fin})  e^{-4\pi}  C^{-n-4}             \\
  &\quad  \times     \sum_{ [t] \in {\rm Cl}^{\rm a}_F(C)}  
      \langle 
      {\mathbf a}(S^{\gamma_t} ), 
                     Q_{S^{\gamma_t}}(X,Y) \rangle_{\mathcal L} 
       \phi(t),     
\end{align*}
since $\det\gamma_t = C$.
This shows the proposition. 
\end{proof}

In the next corollary,  we clarify a relation between Fourier coefficients of HST lifts and 
   the central value of the standard $L$-function of $\pi$.    
For this purpose, we briefly recall the definition of Hida's canonical period $\Omega_{\pi,p}$ of $\pi$ 
    to discuss an integrality of ${\mathbf a}^0(S)$ for $S\in \Lambda_2$ with $\det S = \frac{\Delta_F C^2}{4}$. 
The definition of $\Omega_{\pi,p}$ is given in \cite[Section 8]{hi94cr} by using cuspidal cohomology groups, 
   however we introduce a definition of $\Omega_{\pi,p}$ by using cohomology group with compact support 
   to ensure an integrality of $L(\frac{1}{2}, \pi\otimes\phi)$ according to the same idea in \cite{hi94cr}.       
See \cite[Section 3]{na17} for more detail of the definition.
Let $p$ be an odd prime 
   and fix an isomorphism ${\mathbf C} \cong {\mathbf C}_p$ as fields. 
Let 
\begin{align*}
Y_{\mathfrak N} 
  =  F^\times \backslash {\rm GL}_2(F_{\mathbf A}) / F^\times_{{\mathbf A}, \infty} {\rm SU}_2({\mathbf R})U_0({\mathfrak N})
\end{align*}
and  ${\mathscr L}(n; {\mathbf C}_p)$  
         a local system on $Y_{\mathfrak N}$ which is defined in \cite[Section 2]{hi94cr},  
then we have the Eichler-Shimura-Harder isomorphism (\cite[Proposition 3.1]{hi94cr}): 
 \begin{align*}
             \delta: \begin{CD}  S_{n+2}(U_0({\mathfrak N}))   @>\sim>>  H^1_{\rm cusp}(Y_{\mathfrak N}, {\mathscr L}(n; {\mathbf C})) \end{CD},   
\end{align*}
where $H^1_{\rm cusp}$ is the cuspidal cohomology group. 
By \cite[Section 2.1]{hi99}, we have a Hecke equivariant section:
\begin{align*}
  i:  
  \begin{CD} H^1_{\rm cusp}(Y_{\mathfrak N}, {\mathscr L}(n; {\mathbf C})) 
      @>>> H^1_{\rm c}(Y_{\mathfrak N}, {\mathscr L}(n; {\mathbf C}))  \end{CD},        
\end{align*}
where $H^1_{\rm c}$ is the cohomology group with compact support. 
The fixed isomorphism ${\mathbf C} \cong {\mathbf C}_p$ 
    induces an isomorphism between ${\mathscr L}(n; {\mathbf C})$ and 
    the $p$-adic avatar ${\mathscr L}(n; {\mathbf C}_p)$ (\cite[Proposition 3.14]{na17}). 
Hence we have a Hecke equivariant map  
$H^1_{\rm cusp}(Y_{\mathfrak N}, {\mathscr L}(n; {\mathbf C})) 
       \to H^1_{\rm c}(Y_{\mathfrak N}, {\mathscr L}(n; {\mathbf C}_p) )$, which we also denote by $i$.     
The natural embedding ${\mathcal O}_{{\mathbf C}_p} \to {\mathbf C}_p$ induces a map 
\begin{align*}
    j: 
     \begin{CD} H^1_{\rm c}(Y_{\mathfrak N}, {\mathscr L}(n; {\mathcal O}_{ {\mathbf C}_p}  ) )
         @>>>   H^1_{\rm c}(Y_{\mathfrak N}, {\mathscr L}(n; {\mathbf C}_p) )  \end{CD}.   
\end{align*}  
Let $\eta_f$ be a generator of 
    $\langle i\circ \delta(f)  \rangle_{ {\mathbf C}_p }
       \cap j(  H^1_{\rm c}(Y_{\mathfrak N}, {\mathscr L}(n; {\mathcal O}_{ {\mathbf C}_p}  ) )  )$
as an ${\mathcal O}_{{\mathbf C}_p}$-module.
Then define Hida's canonical period $\Omega_{\pi,p} \in {\mathbf C}^\times_p$ so that 
\begin{align*}
  \Omega_{\pi,p} \eta_f = i \circ \delta(f).     
\end{align*}
Note that $\Omega_{\pi,p}$ is  unique up to multiplications of elements in ${\mathcal O}^\times_{{\mathbf C}_p}$.

\begin{cor}\label{cor:Bess}
{\itshape  
Assume the conditions which are given in Theorem \ref{th:bessel}.
\begin{enumerate}
\item \label{cor:Bess(i)} Let $S=S_x\in \Lambda_2$ be a matrix introduced in Section \ref{sec:exBess}.  
 For each $t\in {\rm Cl}^{\rm a}_F(C)$, let $\gamma_t$ and $S^{\gamma_t}$ be as in (\ref{def:gt}).
Then, we have      
\begin{align}\label{FCBess}
 \begin{aligned}
  {\mathbf a}^0(S^{\gamma_t})
  :=&   \frac{\langle {\mathbf a}(S^{\gamma_t}), Q_{S^{\gamma_t}}(X,Y) \rangle_{\mathcal L} }{\Omega_{\pi,p} }   \\
 =& \frac{   2^{-3} \sqrt{-1}  \Delta^{\frac{n+2}{2}}_F  [  \widehat{\mathcal O}^\times_F :  \widehat{\mathcal O}^\times_F(C) ]   C^{n+2}   
                 }{       \sharp{\rm Cl}^{\rm a}_F(C)    }  
     \sum_{ \phi \in {\rm Cl}^{\rm a}_F(C)^\vee  }
     \frac{L( \frac{1}{2} ,\pi\otimes \phi)}{\Omega_{\pi,p} }
     e(\pi, \phi, \delta)
     \phi^{-1}(t)
     \prod_{v\mid C}\epsilon( 0,\phi^{-1}_v).
\end{aligned}
\end{align}
\item \label{cor:Bess(ii)} Let $p$ be a rational prime numer with  
          $p\nmid 2\cdot n!   \cdot \sharp{\rm Cl}^{\rm a}_F(1)$.  
          Then for each $S\in \Lambda_2$ with $\det S=\frac{\Delta_F C^2}{4}$,  
          ${\mathbf a}^0(S)/\Omega_{\pi, p}$ is $p$-adically integral. 
\item  \label{cor:Bess(iii)} Let $p$ be a rational prime. 
          Assume 
         \begin{itemize}
         \item $\delta(\infty)=-1$ and $\delta(v)=1$ for each finite $v\in \Sigma_{\mathbf Q}$; 
         \item  the condition {\rm (LR)} which is given in Corollary \ref{cor:nonvan};     
         \item  $p\nmid 2\cdot n!  \cdot \Delta_F  \cdot {\rm Cl}^{\rm a}_F(1)$.     
         \end{itemize}
         Then there exists a positive integer $C\in {\mathbf Z}$ satisfying the condition $({\rm CF})$ in Section \ref{sec:exBess} and  
         \begin{align*}
              p\nmid C \cdot \prod_{ \substack{ \ell:\text{prime} \\  \ell \mid C} }(\ell - 1) 
         \end{align*}
         such that 
         the following statements are equivalent:          
         \begin{itemize}
         \item For some $\phi \in {\rm Cl}^{\rm a}_F(C)$,  $\displaystyle \frac{L(\frac{1}{2}, \pi\otimes\phi)}{\Omega_{\pi,p} } \in {\mathcal O}^\times_{\mathbf C_p}$; 
         \item For some $S\in \Lambda_2$ with $\det S= \frac{\Delta_F C^2}{4}$, $\displaystyle \frac{ {\mathbf a}^0(S) }{ \Omega_{\pi,p} } \in {\mathcal O}^\times_{\mathbf C_p}$.
         \end{itemize}
\end{enumerate}
}
\end{cor}
\begin{proof}
We prove the first statement. 
The Fourier inversion theorem and Proposition \ref{BFC} yield that  
\begin{align*}
    & \langle {\mathbf a}(S^{\gamma_t} ),   Q_{S^{\gamma_t}}(X, Y)  \rangle  \\ 
 =&   \frac{   [  \widehat{\mathcal O}^\times_F :  \widehat{\mathcal O}^\times_F(C) ]  e^{4\pi} C^{n+4}   
               }{ {\rm vol}(  {\mathcal U}_{N,1} , {\rm d}h_{\rm fin})  \sharp{\rm Cl}^{\rm a}_F(C)    }  
     \sum_{\phi \in {\rm Cl}^{\rm a}_F(C)^\vee } 
      \langle  {\mathbf B}_{S,\phi}(\widetilde{\varsigma}),  Q^\varsigma_S(X,Y)  \rangle_{\mathcal L} \\           
  =&   \frac{    [  \widehat{\mathcal O}^\times_F :  \widehat{\mathcal O}^\times_F(C) ]   C^{n+2}   
                    \cdot (-1)^{\frac{n}{2}}  \cdot (-\frac{\Delta_F}{4})^{\frac{n+2}{2}}      }{ \sqrt{-1} \cdot 2 \cdot   \sharp{\rm Cl}^{\rm a}_F(C)    }    
       \times  \binom{n}{\frac{n}{2}} \langle (X^2+Y^2)^{\frac{n}{2}}, (X^2+Y^2)^{\frac{n}{2}}  \rangle_{\mathcal L}          \\     
    &\quad \times \sum_{\phi \in {\rm Cl}^{\rm a}_F(C)^\vee } 
        L(\frac{1}{2},\pi\otimes \phi)
     e(\pi, \phi, \delta)
     \phi^{-1}(t)
     \prod_{v\mid C}\epsilon(0,\phi^{-1}_v).   
\end{align*}
Since we have 
\begin{align*}
 & \binom{n}{\frac{n}{2}}  \langle (X^2+Y^2)^{\frac{n}{2}}, (X^2+Y^2)^{\frac{n}{2}}  \rangle_{\mathcal L}   \\
=& \binom{n}{\frac{n}{2}}   \det \begin{pmatrix} 1 & \sqrt{-1}  \\ \sqrt{-1} & 1  \end{pmatrix}^{-n} 
     \left \langle \begin{pmatrix} 1 & \sqrt{-1}  \\ \sqrt{-1} & 1  \end{pmatrix}\cdot (X^2+Y^2)^{\frac{n}{2}}, 
                \begin{pmatrix} 1 & \sqrt{-1}  \\ \sqrt{-1} & 1  \end{pmatrix}\cdot (X^2+Y^2)^{\frac{n}{2}}  \right\rangle_{\mathcal L}   \\
=& \binom{n}{\frac{n}{2}}   2^{-n} 
   \langle (4\sqrt{-1}XY)^{\frac{n}{2}}, 
               (4\sqrt{-1}XY)^{\frac{n}{2}}  \rangle_{\mathcal L}   \\
=&\binom{n}{\frac{n}{2}}   2^{n} (-1)^{\frac{n}{2}} \times (-1)^{\frac{n}{2}} \binom{n}{\frac{n}{2}}^{-1} \\
=& 2^n, 
\end{align*}
this proves the first statement.

Since $\sharp {\rm Cl}^{\rm a}_F(C)$ is a divisor of $\sharp{\rm Cl}^{\rm a}_F(1)\cdot \sharp (\widehat{\mathcal O}^\times_F/\widehat{\mathcal O}^\times_F(C))$, 
(\ref{FCBess}) shows the second statement.   
 
We prove the third statement. 
The first two assumptions imply that  
$e(\pi, \phi, \delta)$ is a power of $2$ 
 for sufficiently large $C\in {\mathbf Z}$.  
 The condition  $p\nmid C$ implies $p\nmid \epsilon(0, \phi^{-1}_v)$ for each $\phi \in {\rm Cl}^{\rm a}(C)$.   
Hence under the assumptions of the statement, the first statement proves 
\begin{align}\label{aL}
   \frac{ {\mathbf a}^0(S^{\gamma_t}) }{ \Omega_{\pi, p} }
  =  u_{F,n,C}     \sum_{ \phi \in {\rm Cl}^{\rm a}_F(C)^\vee  }
     \frac{L(\frac{1}{2},\pi\otimes \phi)}{\Omega_{\pi,p} }
     \phi^{-1}(t) u_{\pi, \phi.\delta}, 
\end{align}
for some $u_{F,n,C}, u_{\pi, \phi.\delta}  \in {\mathcal O}^\times_{{\mathbf C}_p}$.
By using the Fourier inversion formula again, we also find that
\begin{align}\label{La}
     u_{\pi, \phi.\delta}
     \frac{L(\frac{1}{2},\pi\otimes \phi)}{\Omega_{\pi,p} }
  = u^\prime_{F,n,C}  
      \sum_{[t]\in {\rm Cl}^{\rm a}_F(C)} 
       \frac{ {\mathbf a}^0(S^{\gamma_t}) }{ \Omega_{\pi, p} }  \phi(t),  
\end{align}
for some $u^\prime_{F,n,C}\in {\mathcal O}^\times_{{\mathbf C}_p}$.
Then the third statement follows from (\ref{aL}) and (\ref{La}) immediately.  
\end{proof}

\section{Proof of Theorem \ref{th:bessel}}\label{s:besloc}

In this section, we prove Theorem \ref{th:bessel}    
   by computing  local integrals ${\mathcal B}^j_v(\varphi, \phi)$ in an explicit manner. 
Computations of non-archimedean local integrals are given in Section \ref{sec:nonarBess}
and a computation of archimedean local integral is given in Section \ref{sec:arBess}.

\subsection{The non-archimedean local integrals}\label{sec:nonarBess}

In this subsection, we compute 
\begin{align*}
  {\mathcal B}^j_v(\varphi,\phi)   
   := {\mathcal B}^j_v(\varphi,\phi)  \left( \begin{pmatrix} \varsigma_v & 0_2 \\ 0_2 & {}^{\rm t}\varsigma^{-1}_v \end{pmatrix} \right), \quad ( j\in \{0,1\}),  
\end{align*}
for each finite place $v\in \Sigma_{\mathbf Q}$.
See Lemma \ref{lem:narbessel} for the result. 
We recall some notation. The ideal ${\mathfrak N}$ of ${\mathcal O}_F$ is the minimal ideal 
such that $\pi$ has a $U_0({\mathfrak N})$-fixed vector. 
Define the positive integer $N$ so that $N{\mathbf Z} = {\mathfrak N} \cap {\mathbf Z}$. 
We denote the discriminant of $F$ by $\Delta_F$ and put 
$N_F=\text{l.c.m.}(N, \Delta_F)$.
Recall ${\mathcal U}_{N,1}$ be the open compact subgroup of $H^0_1({\mathbf A}_{\rm fin})$ which is introduced in (\ref{def:un1}).
By the definition (\ref{def:lattice}) of $V^\prime({\mathcal O}_{F,v})$, 
it is easy to verify that
\begin{align*}
   \varrho({\mathcal U}_{N,1,v})V^\prime({\mathcal O}_{F,v}) \subset V^\prime({\mathcal O}_{F,v}).   
\end{align*}
Let $C$ be a positive integer satisfying the condition (CF) in Section \ref{sec:exBess}.
Put
\begin{align*}
  c_v={\rm ord}_v(C); \quad 
  E_{C,x,v} = \left\{  h\in H^0_1({\mathbf Q}_v) \ | \ h^{-1} x \varsigma_v \in V^\prime({\mathcal O}_{F,v}) \oplus V^\prime({\mathcal O}_{F,v})   \right\}.   
\end{align*}
For each finite place $v$ of ${\mathbf Q}$, we let $\varpi_v$ be a uniformizer of ${\mathbf Q}_v$. 
If $v=w\bar{w}$ is split, we write $F_v={\mathbf Q}_v\oplus {\mathbf Q}_v$ and let $\varpi_w = (\varpi_v, 1), \varpi_{\bar{w}} = (1, \varpi_v)$.  
The following lemma can be verified by a direct computation. 

\begin{lem}\label{dcoset}
{\itshape
 Let 
$\widetilde{E}_{C,x,v}$ be a set of complete representatives of 
the double coset  $H_{x}({\mathbf Q}_v) \backslash E_{C,x,v} /{\mathcal U}_{N,1,v}$.  
Then, $\widetilde{E}_{C,x,v}$ is given as follows:
\begin{enumerate}
\item\label{dcoset(i)}
         If $v\nmid N$, then  
\begin{align*}
\widetilde{E}_{C,x,v} 
=\begin{cases}  
      \left\{  \varrho(1_2,1) \right\}, & \text{ ($v$ is non-split in $F$),  }  \\
      \left\{  \varrho\left( \begin{pmatrix} 1 & \varpi^{-r}_w \varpi^{-s}_{\bar{w}}  \\ 0 & 1 \end{pmatrix}, 1 \right)
                                          | \ 0\leq r, s \leq c_v \right\},   &  \text{  ($v=w\bar{w}$ is split in $F$).  }  
   \end{cases}                                       
\end{align*}
\item\label{dcoset(ii)}  
        If $v\mid N$, then
\begin{align*}
  \widetilde{E}_{C,x,v} = \left\{ \varrho\left( 1_2, 1\right),   \varrho\left( \begin{pmatrix} 0 & \delta^{-1}_{F,v}  \\ -\delta_{F,v} & 0 \end{pmatrix} \right)  \right\},  
\end{align*}
where $\delta_{F,v} \in F_v$ is a generator of the different ${\mathfrak d}_v$ of $F_v/{\mathbf Q}_v$.
\end{enumerate}
}
\end{lem}
\begin{proof}
Suppose that $v\nmid N$. 
By the Iwasawa decomposition, for each $h\in H^0_1({\mathbf Q}_v)$, we write 
\begin{align*} 
   h = \varrho\left( \begin{pmatrix} a & b \\ 0 & 1  \end{pmatrix}k, \alpha \right),  
        \quad ( a\in F^\times_v,\,b\in F_v, k\in K_v, \alpha\in {\mathbf Q}^\times_v)
\end{align*}
with ${\rm Nr}_{F_v/ {\mathbf Q}_v} (a\det k)\alpha^2=1$. 
Since $\varrho\left( \begin{pmatrix} \alpha^{-1} & 0 \\ 0 & 1  \end{pmatrix}, \alpha\right)$ is an element in $H_{x}({\mathbf Q}_v)$,
each class $[h]$ of $\widetilde{E}_{C,x,v}$ is represented by $h=\varrho\left( \begin{pmatrix} a & b \\ 0 & 1\end{pmatrix}, 1 \right)$ with ${\rm Nr}_{F_v/ {\mathbf Q}_v}(a)=1$. 
Since $h\in\widetilde{E}_{C,x,v}$,  $h$ satisfies 
\begin{align*}
   \begin{pmatrix}  \overline{a} & \overline{b} - b \\ 0  & a  \end{pmatrix},  \
   C\begin{pmatrix}   \overline{a} & \overline{b} \vartheta^\prime_F - b \overline{\vartheta^\prime_F}   \\   0 & a \overline{\vartheta^\prime_F}   \end{pmatrix}    \in V^\prime({\mathcal O}_{F,v}).
\end{align*}
This yields that, for each  $t \in {\mathcal O}_{F,v}$, 
\begin{align*}
     C \begin{pmatrix} \overline{a}t &  -t\overline{b} +  \overline{t} b  \\ 0 & a \overline{t} \end{pmatrix} 
                          \in V^\prime({\mathcal O}_{F,v}).
\end{align*}
Hence we find  that $b\in C^{-1} {\mathfrak d}^{-1}_v$ and $a\in C^{-1} {\mathfrak d}^{-1}_v$ with ${\rm Nr}_{F_v/ {\mathbf Q}_v}(a)=1$. 
Then it is straightforward to verify the assertion in the case $v\nmid N$.

Now we consider the case $v\mid N$. 
Note that each class in $ \widetilde{E}_{C,x,v}$ is represented by an element in 
\begin{align*}
\left\{ \varrho\left( \begin{pmatrix} 1 & 0 \\ s\varpi_v \delta_{F,v}  & 1 \end{pmatrix} , 1 \right) |  s \in {\mathcal O}_{F,v}   \right\}
                     \cup \left\{ \varrho\left( \begin{pmatrix} s  & \delta^{-1}_{F,v} \\   - \delta_{F,v}  & 0 \end{pmatrix} , 1 \right)   |  s \in {\mathcal O}_{F,v}   \right\}.
\end{align*}
Let $[h]$ be a class of $\widetilde{E}_{C,z,v}$. 
Suppose that $[h]$ is represented by $\varrho\left( \begin{pmatrix} 1 & 0 \\ s\varpi_v \delta_{F,v}  & 1 \end{pmatrix} , 1 \right)$ for some $s \in {\mathcal O}_{F,v}$. 
As we discussed above, we find that 
\begin{align*}
 h^{-1}  \begin{pmatrix} t &  0 \\ 0 & \overline{t} \end{pmatrix} C  
       = C   \begin{pmatrix}  t & 0  \\   {\rm Tr}_{F/{\mathbf Q}}(s\varpi_v t) \delta_{F,v} &  \bar{t}   \end{pmatrix}  
             \in V^\prime({\mathcal O}_{F,v}) 
\end{align*}
for all $t\in{\mathcal O}_{F,v}$.  
This yields that ${\rm Tr}_{F_v/{\mathbf Q}_v}(s\varpi_v )$ and ${\rm Tr}_{F_v/{\mathbf Q}_v}(s\varpi_v\vartheta_F)$ are elements in $N_F {\mathbf Z}_v$.
Since ${\rm Tr}_{F_v/{\mathbf Q}_v}(s\varpi_v t)\in N_F {\mathbf Z}_v$  for each $t\in {\mathcal O}_{F,v}$, 
this shows that $s\varpi_v\in N {\mathcal O}_{F,v}$, and hence  $[h]=[ \varrho(1_2,1)]\in \widetilde{E}_{C,z,v}$.

Suppose that $[h]$ is represented by  $\varrho\left( \begin{pmatrix} s & \delta^{-1}_{F,v} \\ -\delta_{F,v}  & 0 \end{pmatrix} , 1 \right)$ for some  $s \in {\mathcal O}_{F,v}$.
Then we have
\begin{align*}
       h^{-1}  \begin{pmatrix} t &  0 \\ 0 & \overline{t} \end{pmatrix}  C
    = C  \begin{pmatrix}  \overline{t} & 0 \\  {\rm Tr}_{F_v/{\mathbf Q}_v}( \overline{s} t  \delta_{F,v}) & t \end{pmatrix} \in V^\prime({\mathcal O}_{F,v})
\end{align*}
for all $t\in{\mathcal O}_{F,v}$.
Hence, in the same way as above, we find that  $s\in N {\mathcal O}_{F,v}$. 
Note that 
\begin{align*}
      \begin{pmatrix} s & \delta^{-1}_{F,v} \\ -\delta_{F,v}  & 0 \end{pmatrix}
   = \begin{pmatrix} 0 & \delta^{-1}_{F,v} \\ -\delta_{F,v}  & 0 \end{pmatrix}
      \begin{pmatrix} 1 & 0 \\ s\delta_{F,v} & 1 \end{pmatrix}; \quad 
\begin{pmatrix} 1 & 0 \\ s\delta_{F,v} & 1 \end{pmatrix} \in {\mathcal U}_{N,1,v}. 
\end{align*}
Hence the class $[h]$ is equal to the class $[ \varrho\left( \begin{pmatrix} 0 & \delta^{-1}_{F,v} \\ - \delta_{F,v} & 0 \end{pmatrix} , 1 \right)     ] \in \widetilde{E}_{C,z,v}$.
This proves the lemma.
\end{proof}

\begin{lem}\label{lem:narbessel}
{\itshape 
Let $C{\mathcal O}_F$ be the conductor of $\phi$.   
Suppose that
\begin{itemize}
   \item $\Delta_F$ and ${\rm Nr}_{F/{\mathbf Q}}({\mathfrak N})$ are coprime; 
   \item  The condition $({\rm CF})$.
\end{itemize}
Let $v$ be a finite place of ${\mathbf Q}$ and $w$ a place of $F$ above $v$.   
\begin{enumerate}
\item If $v\not\in \Sigma_1\cup \{ v\mid C\}$, then
\begin{align*}
  {\mathcal B}^0_v(\varphi,\phi) 
      =& {\rm vol}({\mathcal U}_{N,1,v}, {\rm d}h_v) L(\frac{1}{2}, \pi_v\otimes\phi_v)  \\
        &\quad  \times \begin{cases}    
                                      1, &  ( v\nmid N ),  \\
                                     (1+\epsilon(\frac{1}{2}, \pi_v\otimes\phi_v)),   &  (v\mid N, v \not\in {\mathcal P}),  \\           
                                      (1+\epsilon(\frac{1}{2}, \pi_v\otimes\phi_v))
                                      (1+\epsilon(\frac{1}{2}, \pi_v\otimes\phi^{-1}_v)),  &  (v\mid N, v \in {\mathcal P})               
                    \end{cases}
\end{align*}
\item If $v\in \Sigma_1$, then 
\begin{align*}
     {\mathcal B}^0_v(\varphi,\phi) + \delta(v) {\mathcal B}^1_v(\varphi,\phi)  
  = {\rm vol}({\mathcal U}_{N,1,v}, {\rm d}h_v) (1+\delta(v))  
        (1+\epsilon(\frac{1}{2}, \pi_v\otimes\phi_v)) 
        L(\frac{1}{2}, \pi_v\otimes \phi_v).   
\end{align*}
\item  If $v\mid C$, then
\begin{align*}
   {\mathcal B}^0_v(\varphi,\phi) = {\rm vol}({\mathcal U}_{N,1,v}, {\rm d}h_v) p^{-2c_v}\epsilon(0,\phi^{-1}_v).   
\end{align*}
\end{enumerate}
}
\end{lem}
\begin{proof}
Write $n_w= {\rm ord}_w({\mathfrak N})$.    
Let $n_v = (n_w, n_{\overline{w}})$ if $v=w\overline{w}$ is split, 
      and  $n_v = n_w$ if $v=w^2$ is ramified.     

Suppose that  $v\not\in \Sigma_1$ and that ${\rm ord}_v(C)=0$. 
Firstly we assume that $\pi_v$ is unramified.    
Then Lemma \ref{dcoset} \ref{dcoset(i)} shows that 
\begin{align*}
    {\mathcal B}^0_v(\varphi,\psi)
    = {\rm vol}({\mathcal U}_{N,1,v}, {\rm d}h_v) \int_{F^\times_v}  
       W_{ {\mathbf f}^\dag  ,v}  \left( \begin{pmatrix} t & 0 \\ 0 & 1  \end{pmatrix}  \right)  \phi_v(t)  {\rm d}^\times t
    =  {\rm vol}({\mathcal U}_{N,1,v}, {\rm d}h_v)  L(\frac{1}{2}, \pi_v\otimes \phi_v). 
\end{align*}

Assume that $v \mid N$ and $v\not\in {\mathcal P}$.   
Then Lemma \ref{dcoset} \ref{dcoset(ii)} shows that 
\begin{align}
\label{B^0_v(i)}
 \begin{aligned} 
    {\mathcal B}^0_v(\varphi,\phi) 
             = & {\rm vol}({\mathcal U}_{N,1,v}, {\rm d}h_v) \cdot  \int_{F^\times_v} \phi_v(t)W_{{\mathbf f}^\dag ,v}\left( \begin{pmatrix} t & 0 \\ 0 & 1 \end{pmatrix}\right) {\rm d}^\times t   \\
                & \quad    +  {\rm vol}({\mathcal U}_{N,1,v}, {\rm d}h_v) \cdot  \int_{F^\times_v} \phi_v(t) 
                                    W_{{\mathbf f}^\dag ,v}\left( \begin{pmatrix} t & 0 \\ 0 & 1 \end{pmatrix} \begin{pmatrix} 0  & \delta^{-1}_{F,v} \\ -\delta_{F,v} & 0  \end{pmatrix}\right) {\rm d}^\times t \\ 
             = & {\rm vol}({\mathcal U}_{N,1,v}, {\rm d}h_v)  L(\frac{1}{2}, \pi_v\otimes \phi_v)  \\
                &   \quad    + {\rm vol}({\mathcal U}_{N,1,v}, {\rm d}h_v) \cdot  \int_{F^\times_v} \phi_v(t) 
                                    W_{{\mathbf f}^\dag ,v}\left( \begin{pmatrix} t  & 0 \\ 0 & \varpi^{-n_v}_v \end{pmatrix} 
                                    \begin{pmatrix} 0  & \delta^{-1}_{F,v} \\ - \varpi^{n_v}_v \delta_{F,v} & 0  \end{pmatrix}\right) {\rm d}^\times t.  
 \end{aligned} 
\end{align}
Since the central character of $\pi_v$ is trivial, (\ref{B^0_v(i)}) and a change of variable show that  
\begin{align*}
    {\rm vol}({\mathcal U}_{N,1,v}, {\rm d}h_v)   
    (1+\phi_v(\varpi^{n_v}_v) \varepsilon(\pi_v) ) 
       L(\frac{1}{2}, \pi_v\otimes \phi_v), 
\end{align*}
where $\varepsilon(\pi_v)$ is the Atkin-Lehner eigenvalue introduced in (\ref{fricke}).
By \cite[3.2.2 Theorem]{sc02} and \cite[(3.4.6)]{tat79}, we find that
 \begin{align*}
   \phi_v(\varpi^{n_v}_v)  \varepsilon(\pi_v) 
= \phi_v(\varpi^{n_v}_v)\epsilon(\frac{1}{2},\pi_v)  
= \epsilon(\frac{1}{2},\pi_v\otimes\phi_v).
\end{align*}
 This proves the first statement in this case. 

Assume that  $v\mid N$ and that $v = w \overline{w} \in {\mathcal P}$.
Suppose that $w\nmid {\mathfrak N}$ and $w\mid {\mathfrak N}$. 
Let $n_{\overline{w}} = {\rm ord}_{\overline{w} }({\mathfrak N})$.  
Similarly as above, 
the definition of ${\mathscr V}_v$ in (\ref{raiseop}) and   Lemma \ref{dcoset} \ref{dcoset(ii)}  show that 
\begin{align*}
      &  {\mathcal B}^0_v(\varphi,\psi)  \\
    =& {\rm vol}({\mathcal U}_{N,1,v}, {\rm d}h_v)
         \times \left\{  
          \int_{F^\times_v}  
              \left\{ W_{ {\mathbf f}^\dag ,v}  \left( \begin{pmatrix} t & 0 \\ 0 & 1  \end{pmatrix}  \right)  
                + \varepsilon(\pi_{\overline{w}}) 
                    W_{ {\mathbf f}^\dag ,v}  \left( \begin{pmatrix} t & 0 \\ 0 & 1  \end{pmatrix} \begin{pmatrix} 0 & 1 \\ -\varpi^{n}_{w} & 0  \end{pmatrix}  \right)     \right\}  \phi_v(t)   
        {\rm d}^\times t   \right. \\ 
      & \quad 
         +  \left. 
         \int_{F^\times_v}  
         \left\{ W_{ {\mathbf f}^\dag ,v}  \left( \begin{pmatrix} t & 0 \\ 0 & 1  \end{pmatrix}  
                                       \begin{pmatrix}  0 & 1 \\  -1  & 0   \end{pmatrix}   \right)  
                + \varepsilon(\pi_{\overline{w}})
                     W_{ {\mathbf f}^\dag ,v}  \left( \begin{pmatrix} t & 0 \\ 0 & 1  \end{pmatrix} 
                                                       \begin{pmatrix}  0 & 1 \\  -1  & 0   \end{pmatrix} 
                                                         \begin{pmatrix} 0 & 1 \\ -\varpi^{n}_{w} & 0  \end{pmatrix}  \right)     \right\}  \phi_v(t)   
        {\rm d}^\times t 
        \right\} \\   
    =&  {\rm vol}({\mathcal U}_{N,1,v}, {\rm d}h_v)  
           \left\{ 1+\varepsilon(\pi_{\overline{w}}) \phi_v(\varpi^{n}_w)      
                      +  \phi_v(\varpi^{n}_{\overline{w}}) \epsilon(\pi_{\overline{w}})  
                             + \varepsilon(\pi_{\overline{w}})   \phi_v(\varpi^{n}_w)  \phi_v(\varpi^{n}_{\overline{w}}) \epsilon(\pi_{\overline{w}})          \right\} 
            L(\frac{1}{2}, \pi_v\otimes \phi_v). 
\end{align*}
Since $\phi^{-1}_v = \phi^c_v$, where $\phi^c_v(t) = \phi_v(t^c)$, we obtain 
\begin{align*}
    {\mathcal B}^0_v(\varphi,\psi)
    = {\rm vol}({\mathcal U}_{N,1,v}, {\rm d}h_v)
       (1+ \varepsilon(\pi_{\overline{w}}) \phi_v(\varpi^n_{\overline{w}})  )
       (1+ \varepsilon(\pi_{\overline{w}}) \phi^{-1}_v(\varpi^n_{\overline{w}})  ) 
        L(\frac{1}{2}, \pi_v\otimes \phi_v).
\end{align*}
Since $\varepsilon(\pi_w)=1$ and $n_w=0$, 
we obtain 
\begin{align*}
(1+ \varepsilon(\pi_{\overline{w}}) \phi_v(\varpi^n_{\overline{w}})  )
       (1+ \varepsilon(\pi_{\overline{w}}) \phi^{-1}_v(\varpi^n_{\overline{w}})  )
 = (1+ \varepsilon(\pi_v) \phi_v(\varpi^{n_v}_v)  )
       (1+ \varepsilon(\pi_v) \phi^{-1}_v(\varpi^{n_v}_v)  ).      
\end{align*}
This proves the first  statement. 

We prove the second statement.  
Suppose that $v\in \Sigma_1 $. 
Note that $v\mid N$ by the definition of $\Sigma_1$  
    and that ${\mathcal P}=\emptyset$ by $\pi_v\cong \pi^c_v$ 
    and that ${\rm ord}_v(C)=0$ by the condition (CF).  
By Lemma \ref{dcoset} \ref{dcoset(ii)}, we have
\begin{align}
\label{B^0_v}
 &  \begin{aligned}
      {\mathcal B}^0_v(\varphi,\phi) 
             = & {\rm vol}({\mathcal U}_{N,1,v}, {\rm d}h_v) \cdot  \int_{F^\times_v} \phi_v(t) 
                           W_{{\mathbf f}^\dag ,v}\left( \begin{pmatrix} t & 0 \\ 0 & 1 \end{pmatrix}\right) {\rm d}^\times t   \\
                & \quad    +  {\rm vol}({\mathcal U}_{N,1,v}, {\rm d}h_v) \cdot  \int_{F^\times_v} \phi_v(t)  
                                        W_{{\mathbf f}^\dag ,v}\left( \begin{pmatrix} t & 0 \\ 0 & 1 \end{pmatrix} \begin{pmatrix} 0  & 1 \\ -1 & 0  \end{pmatrix}\right) 
                {\rm d}^\times t,  
    \end{aligned}   \\
\label{B^1_v}
 &  \begin{aligned}
    {\mathcal B}^1_v(\varphi,\phi) 
           =&  {\rm vol}({\mathcal U}_{N,1,v}, {\rm d}h_v) \cdot \int_{F^\times_v} \phi_v(t^c)  
                         W_{{\mathbf f}^\dag ,v}\left( \begin{pmatrix} t & 0 \\ 0 & 1 \end{pmatrix}w_0\right) {\rm d}^\times t   \\
              & \quad    + {\rm vol}({\mathcal U}_{N,1,v}, {\rm d}h_v) \cdot  \int_{F^\times_v} \phi_v(t^c) 
                                         W_{{\mathbf f}^\dag ,v}\left( \begin{pmatrix} t & 0 \\ 0 & 1 \end{pmatrix} w_0  \begin{pmatrix} 0 & 1 \\ -1 & 0  \end{pmatrix}\right) 
              {\rm d}^\times t, 
    \end{aligned}
\end{align}
where $w_0 = \begin{pmatrix} 0 & 1 \\ -1 & 0 \end{pmatrix}$.
The both integrals (\ref{B^0_v}) and (\ref{B^1_v}) are  computed in the same way with (\ref{B^0_v(i)}), 
    and hence we obtain
\begin{align*}
   {\mathcal B}^0_v(\varphi,\phi)  
      =& {\rm vol}({\mathcal U}_{N,1,v}, {\rm d}h_v)   
    (1+\phi_v(\varpi^{n_v}_v) \varepsilon(\pi_v) ) 
       L(\frac{1}{2}, \pi_v\otimes \phi_v),   \\
   {\mathcal B}^1_v(\varphi,\phi)  
      =& {\rm vol}({\mathcal U}_{N,1,v}, {\rm d}h_v)   
    (1+\phi^c_v(\varpi^{n_v}_v) \varepsilon(\pi_v) ) 
       L(\frac{1}{2}, \pi_v\otimes \phi^c_v).
\end{align*}
Since $\pi_v\cong \pi^c_v$ for $v\in \Sigma_1$, we find that 
\begin{align*}
L(\frac{1}{2},\pi_v\otimes\phi^c_v)
=L(\frac{1}{2},\pi^c_v\otimes\phi^c_v)
=L(\frac{1}{2},\pi_v\otimes\phi_v).
\end{align*} 
Also \cite[3.2.2 Theorem]{sc02} and \cite[(3.4.6)]{tat79} show that 
\begin{align*}
    \epsilon(\frac{1}{2}, \pi_v\otimes  \phi^{-1}_v) 
    = \epsilon(\frac{1}{2}, \pi_v )  \phi^{-1}_v(\varpi^{n_v}_v) 
    = \varepsilon(\pi_v)  \phi^{-1}_v(\varpi^{n_v}_v) ,   
\end{align*}
where the third identity follows from the definition of the Atkin-Lehner eigenvalue. 
Since $\pi^c_v \cong \pi_v$, we have $n_w=n_{\overline{w}}$ and hence 
$\phi_v(\varpi^{n_v}_v)=1$. 
In particular, we see that 
\begin{align*}
  \epsilon(\frac{1}{2}, \pi_v\otimes \phi_v)
  = \varepsilon(\pi_v) \phi_v(\varpi^{n_v}_v) 
 =  \varepsilon(\pi_v).   
\end{align*}
This shows that 
\begin{align*}
{\mathcal B}^1_v(\varphi,\phi)  
  = {\mathcal B}^0_v(\varphi,\phi) 
  = {\rm vol}({\mathcal U}_{N,1,v}, {\rm d}h_v)   
    (1+ \varepsilon(\frac{1}{2}, \pi_v\otimes \phi_v) ) 
       L(\frac{1}{2}, \pi_v\otimes \phi_v),  
\end{align*}
and hence we obtain the second statement.

We prove the third statement.  
Suppose that $v\mid C$. 
By the condition (CF), $v=w_1w_2$ is split in $F$ and $\pi_v$ is unramified. 
Write $\varpi_v=(\varpi_{w_1}, \varpi_{w_2})$.
By the Schr${\ddot{\rm o}}$dinger realization (\ref{eq:weilsch}) of the action of Weil representation, 
we find that
\begin{align*}
   \omega_{V_v}(\varsigma_v)\varphi(h^{-1}x ) = p^{-2c_v}, \quad (c_v = {\rm ord}_v(C)).   
\end{align*}
It follows from Lemma \ref{dcoset} \ref{dcoset(i)} that
\begin{align*}
    &  ( p^{-2c_v} {\rm vol}({\mathcal U}_{N,1,v}, {\rm d}h) )^{-1} \times {\mathcal B}^0_v(\varphi,\phi)   \\
  =&  \sum_{0 \leq r_1,r_2  \leq c_v }\int_{F^\times_v} 
                  \phi_v(t)W_{ {\mathbf f}^\dag ,v }\left( \begin{pmatrix} t & 0 \\ 0 & 1 \end{pmatrix} \begin{pmatrix}   1  &  \varpi^{-r_1}_{w_1}  \varpi^{-r_2}_{w_2} \\ 0 & 1 \end{pmatrix} \right) 
                  {\rm d}^\times t \\ \displaybreak[0]
  =&  \sum_{r_1, r_2} \int_{F^\times_v} \phi_v(t) \psi_v(\varpi^{-r_1}_{w_1} t_{w_1}+ \varpi^{-r_2}_{w_2}t_{w_2} ) 
                   W_{{\mathbf f}^\dag ,v}\left(  \begin{pmatrix} t & 0 \\ 0 & 1 \end{pmatrix} \right) {\rm d}^\times t \\ \displaybreak[0]
  =&  \prod_{i=1,2}  \sum_{r_1, r_2} \sum^\infty_{j=0}  W_{{\mathbf f}^\dag ,w_i}\left(  \begin{pmatrix} \varpi^j_{w_i} & 0 \\ 0 & 1 \end{pmatrix} \right) 
                      \phi_{w_i}( \varpi^j_{w_i} )  \int_{{\mathcal O}^\times_{F, w_i}} \phi_{w_i}(u_{w_i} ) \psi(\varpi^{-r_i+j}_{w_i} u)  {\rm d}^\times u_{w_i}.
\end{align*}
Then \cite[1.1.1 Lemma]{sc02} shows that the integral vanishes unless $r_1=r_2=c_v, j=0$. 
Since $\phi$ is anticyclotomic, we have $\phi_{w_1}(\varpi^a_{w_1} u_{w_1})\phi_{w_2}(\varpi^a_{w_2} u_{w_2}) = \phi(u)$ 
      for each $a\in {\mathbf Z}$ and each $u=(u_{w_1},u_{w_2})\in {\mathcal O}^\times_{F,v}$.  
Since
\begin{align*}
   {\mathcal B}^0_v(\varphi,\phi)  
=  p^{-2c_v} {\rm vol}({\mathcal U}_v, {\rm d}h) \cdot 
   \prod_{i=1,2}  \int_{{\mathcal O}^\times_{F, w_i}} \phi_{w_i}( \varpi^{-c_v}_{w_i} u_{w_i} ) \psi( \varpi^{-c_v}_{w_i} u)  {\rm d}^\times u_{w_i}, 
\end{align*}
the proposition follows from the definition (\ref{def:ep}) of the $\epsilon$-factor of ${\rm GL}_1$. 
\end{proof}

\subsection{The archimedean local integrals}\label{sec:arBess}

In this subsection, we calculate the archimedean local integrals:
\begin{align*}
  {\mathcal B}^j_\infty(\varphi, \phi) := {\mathcal B}^j_\infty(\varphi,\phi)\left( \begin{pmatrix} \varsigma_\infty & 0_2 \\ 0_2 & {}^{\rm t}\varsigma^{-1}_\infty  \end{pmatrix}\right),  \quad (j \in \{ 0,1\}). 
\end{align*}

The following lemma is proved by Iwasawa decomposition in a straight forward way. 

\begin{lem}\label{coset_infty} 
{\itshape The following map gives a bijection:
  \begin{align*}
     {\mathbf C} \times \{\pm1\}\backslash{\rm SU}_2({\mathbf R})  \stackrel{\sim}{\to} H_{x}({\mathbf R}) \backslash H^0_1({\mathbf R}):
       (z, u) \mapsto \varrho\left( \begin{pmatrix} 1 & z \\ 0 & 1 \end{pmatrix} u, 1  \right).
  \end{align*}
}
\end{lem}

We fix a measure ${\rm d}h$ on $H_{x}({\mathbf R}) \backslash H^0_1({\mathbf R})$ as follows:
\begin{align*}
    h=\begin{pmatrix} 1 & x+\sqrt{-1}y \\ 0 & 1 \end{pmatrix} u, \quad 
    {\rm d}h = {\rm d}x\wedge {\rm d}y \wedge  {\rm d}u, \quad 
    {\rm vol}({\rm SU}_2({\mathbf R}), {\rm d}u)=1. 
\end{align*}
We write ${\rm d}z={\rm d}x\wedge {\rm d}y$ if $z=x+\sqrt{-1}y$.
Note that ${\rm d}h$ is induced by the measure on ${\rm SL}_2({\mathbf C})$ in Section \ref{s:ArchInt}.

The purpose of this section is to prove the following lemma: 

\begin{lem}\label{lem:arbessel}
{\itshape 
We have 
\begin{align*}
   {\mathcal B}^0_\infty(\varphi,\phi)
=&- 2^{-1} e^{-4\pi }  \Gamma_{\mathbf C} \left(\frac{k}{2}\right)^2\cdot
                      \sum^n_{\alpha=0}  \sqrt{-1}^{n-\alpha+1} 
                                                 \sum_{c\in {\mathbf Z}} (-1)^c \binom{\alpha}{c}\binom{n-\alpha}{\frac{n}{2}-c} 
                                                 \cdot \binom{n}{\alpha}  X^\alpha Y^{n-\alpha}, \\
  {\mathcal B}^1_\infty(\varphi,\phi) 
  =&-{\mathcal B}^0_\infty(\varphi,\phi).
\end{align*}
}
\end{lem}
\begin{proof}
We put ${\mathcal B}^j_\infty(\varphi,\phi) = \sum^n_{\alpha=0}  {\mathcal B}_\alpha^j \cdot \binom{n}{\alpha} X^\alpha Y^{n-\alpha}$ for $j=0,1$.
Then Lemma \ref{coset_infty} shows that 
\begin{align*}
       {\mathcal B}^0_\alpha
 = &  \int_{H_{x}({\mathbf R}) \backslash H^0_1({\mathbf R})} {\rm d}h
         \int_{{\mathbf C}^\times} {\rm d}^\times t
           \left\langle \omega_{V_\infty} (\varsigma_\infty) \varphi^\alpha_\infty(h^{-1}x), W_{{\mathbf f}^\dag ,\infty}\left( \begin{pmatrix} t & 0 \\ 0 & 1 \end{pmatrix} h\right)    \right\rangle_{\mathcal W}   
            \phi_\infty(t) \\ \displaybreak[0]
 = &  \frac{1}{2} \int_{\mathbf C} {\rm d} z
        \int_{{\mathbf C}^\times} {\rm d}^\times t
         \left\langle  \varphi^\alpha_\infty
         \left( \begin{pmatrix} 1 & z \\ 0 & 1 \end{pmatrix}^{-1}x \varsigma_\infty  \right), \psi_{F,\infty}(t z) W_{{\mathbf f}^\dag ,\infty}\left( \begin{pmatrix} t & 0 \\ 0 & 1 \end{pmatrix}\right)   \right\rangle_{\mathcal W}.
\end{align*}
Define a polynomial $Q^\alpha(X,Y)= Q^\alpha(z, \bar{z}; X,Y)$ so that
\begin{align*}
         \varphi^\alpha_\infty\left( \begin{pmatrix} 1 & z \\ 0 & 1 \end{pmatrix}^{-1}x \varsigma_\infty  \right)
 = &  \varphi^\alpha_\infty\left( \begin{pmatrix} 1 & -2\sqrt{-1}y_z \\ 0 & 1 \end{pmatrix},  \begin{pmatrix} \sqrt{-1} & 2\sqrt{-1} x_z \\ 0 & -\sqrt{-1} \end{pmatrix}  \right)
 =  Q^\alpha(z, \bar{z}; X,Y) e^{-\pi (4z\bar{z}+4)},  
\end{align*}
where we write $z=x_z+\sqrt{-1}y_z$.
Then we find that 
\begin{align*}
Q^\alpha(z, \bar{z}; X,Y) =    & P^\alpha\left( \begin{pmatrix} 1 & -\sqrt{-1} y_z \\ -\sqrt{-1} y_z & 1 \end{pmatrix},
                                              \begin{pmatrix} \sqrt{-1} & \sqrt{-1}x_z \\ \sqrt{-1}x_z & -\sqrt{-1} \end{pmatrix}   \right)        \\
  =& 2 \sqrt{-1}^{n-\alpha+1}
         ( zX^2  - 2 XY -\overline{z} Y^2 )
         (  X^2  -  (z-\bar{z})XY   + Y^2)^\alpha 
         (  X^2  +  (z+\bar{z})XY   - Y^2)^{n-\alpha} \\
  =& 2 \sqrt{-1}^{n-\alpha+1}
         ( zX^2 - 2XY -\overline{z} Y^2 )  \\
    &\quad      \sum_{p,q,u,v,r,s,a,b}
         \frac{\alpha!}{p!q!u!v!} \cdot \frac{(n-\alpha)!}{r!s!a!b!} 
           (-z)^u \overline{z}^v z^a \overline{z}^b (-1)^s  
           X^{2p} (XY)^{u+v} Y^{2q} X^{2r} (XY)^{a+b} Y^{2s}   \\
  =& 2 \sqrt{-1}^{n-\alpha+1}
         ( zX^2 - 2 XY -\overline{z} Y^2 )  \\
    &\quad       \sum_{p,q,u,v,r,s,a,b}
         \frac{\alpha!}{p!q!u!v!} \cdot \frac{(n-\alpha)!}{r!s!a!b!} 
         \cdot (-1)^{u+s} \cdot 
           z^{u+a} \overline{z}^{v+b}  
           X^{u+v+a+b+2p+2r}  Y^{u+v+a+b+2q+2s},
\end{align*}
where $p, q, u, v, r, s, a, b$ run over the set of rational integers with
$p+q+u+v=\alpha, r+s+a+b=n-\alpha$.
To obtain the formula for ${\mathcal B}_\alpha^0$, we consider the following  integral:
\begin{align*}
{\mathcal W}_\alpha^0(t) &:= \frac{1}{2}e^{-4\pi} \int\int_{{\mathbf R}^{\oplus 2}}  
                            \psi_{F, \infty}(zt)
                            Q^\alpha(z, \bar{z}; X,Y)   e^{-4\pi z \overline{z}  }  {\rm d}x {\rm d}y.
\end{align*}
To proceed the computation, we need the following lemma.
\begin{lem}\label{lem:FouEle}
{\itshape 
Let $c\in{\mathbf R}_{>0}$ and $A, B \in {\mathbf Z}_{\geq 0}$. 
Write $z=x+\sqrt{-1}y\in {\mathbf C}$. 
Then we have the following identities: 
\begin{enumerate}
\item  $\displaystyle \int\int_{{\mathbf R}^{\oplus 2}}      e^{-4 \pi z\bar{z}} \psi_{F, \infty}(zt) {\rm d}x {\rm d}y = \frac{1}{4}e^{- \pi t\bar{t}}$.
\item $\displaystyle     \left( \frac{1}{ c } \frac{\partial }{\partial t}  \right)^A
              \left( \frac{1}{ c } \frac{\partial }{\partial \overline{t} }  \right)^B
               e^{c t\overline{t}}   
            =  \sum_{j\in {\mathbf Z}}  
                   \binom{A}{j} \binom{B}{j} j! \cdot c^{-j} \overline{t}^{A-j} t^{B-j} e^{ct\overline{t}}$.
\item $\displaystyle  
        \int\int_{{\mathbf R}^{\oplus 2}}    
       z^A \overline{z}^B
         e^{-4\pi z\overline{z}}        
       \psi_{F, \infty}(zt)
       {\rm d} x{\rm d}y
      = 2^{-2-A-B} \sqrt{-1}^{A+B} \cdot  
       \sum_{j\in {\mathbf Z}}  
                   \binom{A}{j} \binom{B}{j} j! \cdot (-\pi)^{-j} \overline{t}^{A-j} t^{B-j} e^{-\pi t\overline{t}}$.         
\end{enumerate}
}
\end{lem}
\begin{proof}
The first two formulas easily follow. 
The third one follows from these two:
\begin{align*}
    \int\int_{{\mathbf R}^{\oplus 2}}    
         z^A \overline{z}^B
         e^{-4\pi z\overline{z}}        
         \psi_{F, \infty}(zt)
         {\rm d} x{\rm d}y
=& \left( \frac{1}{2\pi\sqrt{-1}} \frac{\partial }{\partial t}  \right)^A
      \left( \frac{1}{2\pi\sqrt{-1}} \frac{\partial }{\partial \overline{t} }  \right)^B
       \int\int_{{\mathbf R}^{\oplus 2}}    
         e^{-4\pi z\overline{z}}
        e^{{\rm Tr}_{{\mathbf C}/{\mathbf R}}(tz)}
       {\rm d} x{\rm d}y   \\
 =&  \left( \frac{1}{2\pi\sqrt{-1}} \frac{\partial }{\partial t}  \right)^A
      \left( \frac{1}{2\pi\sqrt{-1}} \frac{\partial }{\partial \overline{t} }  \right)^B
      \frac{1}{4} e^{-\pi t\overline{t}} \\
 =& 2^{-2-A-B} \sqrt{-1}^{A+B} \cdot  
      \left( \frac{1}{ -\pi } \frac{\partial }{\partial t}  \right)^A
      \left( \frac{1}{ -\pi } \frac{\partial }{\partial \overline{t} }  \right)^B
       e^{-\pi t\overline{t}}  \\
 =& 2^{-2-A-B} \sqrt{-1}^{A+B} \cdot  
       \sum_{j\in {\mathbf Z}}  
                   \binom{A}{j} \binom{B}{j} j! \cdot (-\pi)^{-j} \overline{t}^{A-j} t^{B-j} e^{-\pi t\overline{t}}. \qedhere  
\end{align*}
\end{proof}

Applying Lemma \ref{lem:FouEle}, we find that ${\mathcal W}_\alpha^0(z)$ is equal to 
\begin{align*}
& e^{-4\pi} \sqrt{-1}^{n-\alpha+1} 
      \sum_{p,q,u,v,r,s,a,b}
                 \frac{\alpha!}{p!q!u!v!}\frac{(n-\alpha)!}{r!s!a!b!} \cdot  (-1)^{u+s} 
                 \cdot  2^{-2 -(u+a) - (v+b) }  \sqrt{-1}^{u+v+a+b}    \\
             & \times    \sum_j 
                  \left(      2^{-1}\sqrt{-1} \binom{u+a+1}{j}\binom{v+b}{j} \bar{z} X^2  
                       - 2 \binom{u+a}{j}\binom{v+b}{j} XY   \right. \\
             & \quad \quad \quad \quad \quad \quad \quad \quad \quad \quad \quad \quad \quad \quad \quad \quad \quad \quad \quad \quad \quad 
                 \left.   - 2^{-1}\sqrt{-1} \binom{u+a}{j}\binom{v+b+1}{j} z Y^2  \right) \\
             & \times   j! \left(-\pi\right)^{-j}
                             z^{v+b-j} \overline{z}^{u+a-j}  
                             e^{-\pi z\bar{z}} 
                             X^{u+v+a+b+2p+2r}  Y^{u+v+a+b+2q+2s} \\
=& 2^{-3} e^{-4\pi} \sqrt{-1}^{n-\alpha+1} 
      \sum_{p,q,u,v,r,s,a,b}
                 \frac{\alpha!}{p!q!u!v!}\frac{(n-\alpha)!}{r!s!a!b!} 
                 \cdot  (-1)^{u+s}  2^{-u-v-a-b }   \sqrt{-1}^{u+v+a+b}   \\
             & \times    \sum_j 
                  \left(      \sqrt{-1} \binom{u+a+1}{j}\binom{v+b}{j} \bar{z} X^2 
                       -4 \binom{u+a}{j}\binom{v+b}{j} XY
                         - \sqrt{-1} \binom{u+a}{j}\binom{v+b+1}{j} z Y^2  \right) \\
             & \times   j! \left(-\pi\right)^{-j}
                             z^{v+b-j} \overline{z}^{u+a-j}  
                             e^{-\pi z\bar{z}} 
                             X^{u+v+a+b+2p+2r}  Y^{u+v+a+b+2q+2s},
\end{align*}
where $p, q, u, v, r, s, a, b, j$ run over rational integers with
$p+q+u+v=\alpha, r+s+a+b=n-\alpha$.
Write $z=te^{\sqrt{-1}\theta} (t\in {\mathbf R}_{>0})$. Then
\begin{align*}
  &  \left\langle  {\mathcal W}^0_\alpha(te^{\sqrt{-1}\theta}), W_{ {\mathbf f}^\dag, \infty}\left( \begin{pmatrix} t & 0 \\ 0 & 1  \end{pmatrix} 
     \begin{pmatrix} e^{\frac{\sqrt{-1}\theta}{2}} & 0 \\ 0 & e^{-\frac{\sqrt{-1}\theta}{2}}  \end{pmatrix}     \right)  \right\rangle_{\mathcal W} \\
 =&  2^{-3} e^{-4\pi} \sqrt{-1}^{n-\alpha+1} 
      \sum_{p,q,u,v,r,s,a,b, j}
                 \frac{\alpha!}{p!q!u!v!}\frac{(n-\alpha)!}{r!s!a!b!} 
                 \cdot  (-1)^{u+s}  2^{-u-v-a-b }   \sqrt{-1}^{u+v+a+b} j! \left(-\pi\right)^{-j}  e^{-\pi t^2}  \\
   &\left \langle  \left(      \sqrt{-1} \binom{u+a+1}{j}\binom{v+b}{j} t X^2 
                             -4  \binom{u+a}{j}\binom{v+b}{j} XY
                         - \sqrt{-1} \binom{u+a}{j}\binom{v+b+1}{j} t Y^2  \right)      \right. \\ 
  & \quad \quad 
      \left. \times  t^{u+v+a+b-2j}  e^{\sqrt{-1}\theta(-u+v-a+b+p+r-q-s)} X^{u+v+a+b+2p+2r}  Y^{u+v+a+b+2q+2s},   
                 W_{\mathbf f, \infty}\left( \begin{pmatrix} t & 0 \\ 0 & 1  \end{pmatrix}  \right)  \right\rangle_{\mathcal W}.                 
\end{align*} 
Define  
\begin{align*} 
l= \frac{\alpha}{2}, \quad m= \frac{n-\alpha}{2}, \quad 
l^\pm(w) = \frac{\alpha\pm w}{2}, \quad m^\pm(w) = \frac{n-\alpha\pm w}{2}, \quad (w\in {\mathbf Z}). 
\end{align*}
Consider the following identities:
\begin{align*}
     p= l^+(w) -v, \quad q = l^-(w) -u, \quad r= m^-(w)-b, \quad s= m^+(w) -a \quad ( w \in {\mathbf Z}),  
\end{align*} 
then we find that 
\begin{align*}
  & p+q+u+v = \alpha, \quad r+s+a+b=n-\alpha, \quad 
   -u+v-a+b+p+r-q-s = 0,  \\
  & u+v+a+b+2p+2r= n + a -b +u-v,  \quad   u+v+a+b+2q+2s= n - a +b -u+v.
\end{align*}
Thus  the integral on $\theta$ is given by  
\begin{align}\label{eq:sum0}
\begin{aligned}
   & \int^{2\pi}_0   \left\langle  {\mathcal W}^0_\alpha(te^{\sqrt{-1}\theta}), W_{ {\mathbf f}^\dag, \infty}\left( \begin{pmatrix} t & 0 \\ 0 & 1  \end{pmatrix} 
     \begin{pmatrix} e^{\frac{\sqrt{-1}\theta}{2}} & 0 \\ 0 & e^{-\frac{\sqrt{-1}\theta}{2}}  \end{pmatrix}     \right)  \right\rangle_{\mathcal W} \frac{{\rm d}\theta}{2\pi}  \\ 
 =  &  2^{-3} e^{-4\pi} \sqrt{-1}^{n-\alpha+1} 
      \sum_{u,v,a,b, j, w}
      \binom{\alpha}{l^+(w)}\binom{l^-(w)}{u}\binom{l^+(w)}{v}
                                   \binom{n-\alpha}{m^+(w)}\binom{m^+(w)}{a}\binom{m^-(w)}{b}  \\
   &  \times  (-1)^{m^+(w)+u+a}  2^{-u-v-a-b }   \sqrt{-1}^{u+v+a+b} j! \left(-\pi\right)^{-j}   t^{u+v+a+b-2j}  e^{-\pi t^2}  \\
   & \times \left\langle  \left(      \sqrt{-1} \binom{u+a+1}{j}\binom{v+b}{j} t X^2 
                       - 4 \binom{u+a}{j}\binom{v+b}{j} XY
                         - \sqrt{-1} \binom{u+a}{j}\binom{v+b+1}{j} t Y^2  \right)   \right.  \\
  & \left. \quad \quad  \quad   \times X^{n + a -b +u-v} Y^{ n - a +b -u+v},        
     W_{ {\mathbf f}^\dag, \infty}\left( \begin{pmatrix} t & 0 \\ 0 & 1  \end{pmatrix}  \right)  \right\rangle_{\mathcal W},    
\end{aligned}
\end{align}
where $u,v,a,b, j, w$ run over the set of rational integers.
An explicit formula (\ref{eq:expwhitt}) of Whittaker functions yields that  
    (\ref{eq:sum0}) is equal to  
\begin{align}\label{eq:sum1}
\begin{aligned}
   & \int^{2\pi}_0   \left\langle  {\mathcal W}^0_\alpha(t e^{\sqrt{-1}\theta} ), W_{ {\mathbf f}^\dag, \infty}\left( \begin{pmatrix} t e^{\sqrt{-1}\theta} & 0 \\ 0 & 1  \end{pmatrix}  \right)  \right\rangle_{\mathcal W} \frac{{\rm d}\theta}{2\pi}  \\ 
=& 2  e^{-4\pi}\sqrt{-1}^{n-\alpha+1}   
  \times \sum_{ u, v, a, b, j, w  } 
      \binom{\alpha}{l^+(w)}\binom{l^-(w)}{u}\binom{l^+(w)}{v}
                                   \binom{n-\alpha}{m^+(w)}\binom{m^+(w)}{a}\binom{m^-(w)}{b}     \\ 
   &  \times  (-1)^{m^+(w)+u+a}  2^{-u-v-a-b }   \sqrt{-1}^{u+v+a+b} j! \left(-\pi\right)^{-j}   t^{u+v+a+b-2j}  e^{-\pi t^2}    \\  
    & \times   \left\{      \sqrt{-1} \binom{u+a+1}{j}\binom{v+b}{j}  \sqrt{-1}^{a-b+u-v+1} t^{n+3} K_{a-b+u-v+1}(4\pi t)     \right.   \\
    &  \quad \quad \quad   - 4 \binom{u+a}{j}\binom{v+b}{j} \sqrt{-1}^{a-b+u-v} t^{n+2} K_{a-b+u-v}(4\pi t)   \\
    &  \quad \quad \quad  \left.    - \sqrt{-1} \binom{u+a}{j}\binom{v+b+1}{j} \sqrt{-1}^{a-b+u-v-1} t^{n+3} K_{a-b+u-v-1}(4\pi t)  \right\}.  
\end{aligned}
\end{align}
Changing variables $u$  by $u-1$ (resp. $b$ by $b-1$) in the first (resp. third) sum in (\ref{eq:sum1}), (\ref{eq:sum1}) is equal to 
\begin{align*}
  & 2  e^{-4\pi}\sqrt{-1}^{n-\alpha+1}  
  \times \sum_{ u, v, a, b, j, w  } 
            \binom{\alpha}{l^+(w)} \binom{n-\alpha}{m^+(w)} 
            \binom{m^+(w)}{a} \binom{l^+(w)}{v}    
            \binom{u+a}{j}  \binom{v+b}{j} \\ 
   &  \times  (-1)^{m^+(w)+u+ a+1}  2^{1-u-v-a-b }   \sqrt{-1}^{2u+2a} j! \left(-\pi\right)^{-j}   t^{n+u+v+a+b+2-2j}  e^{-\pi t^2}  K_{a-b+u-v}(4\pi t)    \\  
    & \times   \left\{      \binom{l^-(w)}{u-1} \binom{m^-(w)}{b}  
                                + 2 \binom{l^-(w)}{u} \binom{m^-(w)}{b} 
                                +  \binom{l^-(w)}{u} \binom{m^-(w)}{b-1}    \right\}  \\  
 = & -2^{2}  e^{-4\pi}\sqrt{-1}^{n-\alpha+1}  
       \times \sum_{ u, v, a, b, j, w  } 
            \binom{\alpha}{l^+(w)} \binom{n-\alpha}{m^+(w)} 
            \binom{m^+(w)}{a} \binom{l^+(w)}{v}    
            \binom{u+a}{j}  \binom{v+b}{j} \\ 
   &  \times  (-1)^{m^+(w)}  2^{-u-v-a-b }   j! \left(-\pi\right)^{-j}   t^{n+u+v+a+b+2-2j}  e^{-\pi t^2}  K_{a-b+u-v}(4\pi t)    \\  
    & \times   \left\{      \binom{l^-(w)+1}{u} \binom{m^-(w)}{b}  
                                +  \binom{l^-(w)}{u} \binom{m^-(w)+1}{b} \right\}.
\end{align*}
We put
\begin{align*}
s^w_{\ast}
=&   \sum_{ u, v, a, b, j, w  } 
              2^{-u-v-a-b }   j! \left(-\pi\right)^{-j}
            \binom{m^+(w)}{a} \binom{l^+(w)}{v}    
            \binom{u+a}{j}  \binom{v+b}{j} \\ 
   &  \quad  \times  \int^\infty_0   t^{n+u+v+a+b+2-2j}  e^{-\pi t^2}  K_{a-b+u-v}(4\pi t)   {\rm d}^\times t
       \times  \begin{cases}  \binom{l^-(w)+1}{u} \binom{m^-(w)}{b},    &  (\ast = \lambda),  \\
                             \binom{l^-(w)}{u} \binom{m^-(w)+1}{b},  &  (\ast = \mu).   \end{cases} 
\end{align*}
Then the above computations are summarized as follows:
\begin{align*}
 {\mathcal B}_\alpha^0 =&\int_{0}^\infty 
     \left\langle  {\mathcal W}_\alpha^0(t ),  
                 W_{{\mathbf f}^\dag, \infty}\left(  \begin{pmatrix}  t & 0 \\ 0 &  1 \end{pmatrix} \right) 
      \right\rangle_{\mathcal W} 
    {\rm d}^\times t  \\
 =& -2^{2} e^{-4\pi} \sqrt{-1}^{n-\alpha+1} 
                       \sum_w   (-1)^{m^+(w)}   \binom{\alpha}{l^+(w)} \binom{n-\alpha}{m^+(w)}
                   (s^w_\lambda +s^w_\mu  ).
\end{align*}
To compute $s^w_\lambda =  s^{2\alpha-n+w}_\mu$, we introduce some confluent hypergeometric functions:
\begin{align*}
\zeta(z; \alpha,\beta) :=\int^\infty_0 (1+x)^{\alpha-1}x^{\beta-1}e^{-zx} {\rm d}x;    \quad 
\omega(z; \alpha,\beta) :=z^{\beta}\Gamma(\beta)^{-1}\zeta(z;\alpha,\beta).
\end{align*}
The following lemma is well known:

\begin{lem}\label{hypgeom}
{\itshape 
We have the following identities:
\begin{align*}
\omega(z;-\alpha,\beta) =& \sum_{b\in{\mathbf Z}} \binom{n}{b} \frac{(b+\alpha)!}{\alpha !} z^{-b}\cdot \omega(z;-\alpha-b,\beta+n),   
        \quad (n\in{\mathbf Z}_{\geq 0}), \\
\omega(z;\alpha,\beta) =& \omega(z;1-\beta,1-\alpha), \\
\omega(z;\alpha,0) =& \omega(z;1,\beta)=1.
\end{align*}
In addition, for a constant $0<c \in {\mathbf R}$, we have
\begin{align*}
        \int^\infty_0  r^\alpha K_\beta(4\pi  r) e^{-c\pi r^2} {\rm d}^\times r
   =& \frac{1}{4} \Gamma\left(\frac{\alpha-\beta}{2}\right) \Gamma\left(\frac{\alpha+\beta}{2}\right)
         \cdot (2\pi )^{-\alpha}\cdot \omega\left(\frac{4\pi}{c}; \frac{\beta-\alpha}{2}+1,\frac{\alpha+\beta}{2}\right). 
\end{align*}
}
\end{lem}
\begin{proof}
The first three identities are proved in \cite[Chapter 9]{hi94}.
The fourth identity can be deduced easily from the integral expression of the Bessel function $K_s(z)$ which is given in (\ref{def:2bessel}). 
\end{proof}

Lemma \ref{hypgeom} shows that 
\begin{align*}
    s^w_\lambda 
 = & \sum_{ u, v, a, b, j  } 
              2^{-u-v-a-b }   j! \left(-\pi\right)^{-j}
            \binom{m^+(w)}{a} \binom{l^+(w)}{v}    
            \binom{u+a}{j}  \binom{v+b}{j} 
            \binom{l^-(w)+1}{u} \binom{m^-(w)}{b} \\ 
   &  \quad  \times  \int^\infty_0   t^{n+u+v+a+b+2-2j}  e^{-\pi t^2}  K_{a-b+u-v}(4\pi t)   {\rm d}^\times t   \\
 = & \sum_{ u, v, a, b, j  } 
              2^{-u-v-a-b }   j! \left(-\pi\right)^{-j}
            \binom{m^+(w)}{a} \binom{l^+(w)}{v}    
            \binom{u+a}{j}  \binom{v+b}{j} 
            \binom{l^-(w)+1}{u} \binom{m^-(w)}{b} \\ 
   &  \quad  \times  \frac{1}{4}  \Gamma\left( \frac{n}{2} + v + b + 1 -j  \right)  \Gamma\left( \frac{n}{2} + u + a + 1 -j \right)  \\
   &  \quad  \times (2\pi)^{-n-u-v-a-b-2+2j} \cdot \omega\left(4\pi; 1 - \left(\frac{n}{2} + v + b + 1 -j \right) ,  \frac{n}{2} + u + a + 1 -j    \right).  
\end{align*}
By changing the  variable $b$ (resp. $u$) by $b-v+j$ (resp. $u-a$), we find that   
\begin{align*}
    s^w_\lambda 
 = & \sum_{ u, v, a, b, j  } 
              2^{-u-b-j }   j! \left(-\pi\right)^{-j}
            \binom{m^+(w)}{a} \binom{l^+(w)}{v}    
            \binom{u}{j}  \binom{b+j}{j} 
            \binom{l^-(w)+1}{u-a} \binom{m^-(w)}{b-v+j} \\ 
   &  \quad  \times  \frac{1}{4}  \Gamma\left( \frac{n}{2} + b + 1   \right)  \Gamma\left( \frac{n}{2} + u  + 1 -j \right)  \\
   &  \quad  \times (2\pi)^{-n-u-b-2+j} \cdot \omega\left(4\pi; 1 - \left(\frac{n}{2} + b + 1 \right) ,  \frac{n}{2} + u  + 1 -j    \right)    \\
 = & 2^{-2}(2\pi)^{-n-2} \sum_{ u,  b, j  }
         \binom{u}{j}  \binom{b+j}{j}    
         \left( \sum_a  \binom{m^+(w)}{a}  \binom{l^-(w)+1}{u-a}  \right) 
         \left( \sum_v  \binom{l^+(w)}{v}   \binom{m^-(w)}{b-v+j} \right)  \\ 
   &  \quad  \times       \left( \frac{n}{2} + b    \right)!  \left( \frac{n}{2} + u  -j \right)!  
         \cdot  j! (-1)^j (4\pi)^{-u-b} \cdot \omega\left(4\pi;  - \left(\frac{n}{2} + b  \right) , 1 + \frac{n}{2} + u   -j    \right)    \\
 = & 2^{-2}(2\pi)^{-n-2} \sum_{ u,  b, j  }
         \binom{u}{j}  \binom{b+j}{j}    
         \binom{\frac{n}{2}+1}{u} \binom{\frac{n}{2}}{b+j}  \\ 
   &  \quad  \times       \left( \frac{n}{2} + b    \right)!  \left( \frac{n}{2} + u  -j \right)!  
         \cdot  j! (-1)^j (4\pi)^{-u-b} \cdot \omega\left(4\pi;  - \left(\frac{n}{2} + b  \right) , 1 + \frac{n}{2} + u   -j    \right).
\end{align*}

An elementary identity $\binom{b+j}{j}  \binom{\frac{n}{2}}{b+j} = \binom{\frac{n}{2}}{j}\binom{\frac{n}{2}-j}{b}$ shows that
\begin{align*}
    s^w_\lambda 
 = & 2^{-2}(2\pi)^{-n-2} \sum_{ u,  b, j  }
         \binom{u}{j} \binom{\frac{n}{2}+1}{u} 
         \binom{\frac{n}{2}}{j}\binom{\frac{n}{2}-j}{b}    \\ 
   &  \quad  \times       \left( \frac{n}{2} + b    \right)!  \left( \frac{n}{2} + u  -j \right)!  
         \cdot  j! (-1)^j (4\pi)^{-u-b} \cdot \omega\left(4\pi;  - \left(\frac{n}{2} + b  \right) , 1 + \frac{n}{2} + u   -j    \right)      \\
 = & 2^{-2}(2\pi)^{-n-2} \sum_{ u,  j  }
        \binom{\frac{n}{2}+1}{u} \binom{u}{j}  \left(\frac{n}{2}+u-j\right)! \times j! (-1)^j (4\pi)^{-u}  \\
    & \quad \times \binom{\frac{n}{2}}{j} \cdot \left(\frac{n}{2} \right)! 
        \sum_b \binom{\frac{n}{2}-j}{b} \cdot \frac{ (\frac{n}{2}+b)!  }{ (\frac{n}{2})! } \cdot (4\pi)^{-b}  \cdot \omega\left(4\pi;  - \left(\frac{n}{2} + b  \right) , 1 + \frac{n}{2} + u   -j    \right).      
\end{align*}
Lemma \ref{hypgeom} shows that 
\begin{align*}
    s^w_\lambda 
 = & 2^{-2}(2\pi)^{-n-2} \sum_{ u,  j  }
        \binom{\frac{n}{2}+1}{u} \binom{u}{j} \left(\frac{n}{2}+u-j \right)! \times j! (-1)^j (4\pi)^{-u}  
    \times \binom{\frac{n}{2}}{j} \cdot \left( \frac{n}{2} \right)! 
        \omega\left(4\pi; -\frac{n}{2}, 1 + u \right)   \\
 = & 2^{-2}(2\pi)^{-n-2}  \left( \frac{n}{2} \right)! 
       \sum_{ u  }
         \left( \sum_j   \binom{u}{j}  (-1)^j  \left(\frac{n}{2}+u-j \right)!  j!   \binom{\frac{n}{2}}{j}  \right)
        \binom{\frac{n}{2}+1}{u}   
       \cdot  (4\pi)^{-u}
        \omega\left(4\pi; -u, 1 + \frac{n}{2} \right).
\end{align*}
The following elementary lemma is easily verified by induction.

\begin{lem}\label{comb}
{\itshape 
For non-negative integers $A$ and $B$, we have
\begin{align*}
  \sum_j \binom{A}{j} (-1)^j\frac{(A-j+B)!}{(B-j)!} = A!.
\end{align*}
}
\end{lem}
By Lemma \ref{comb}, we see that 
\begin{align*}
\sum_j   \binom{u}{j}  (-1)^j \left( \frac{n}{2}+u-j \right)!  j!   \binom{\frac{n}{2}}{j}  = u! \left( \frac{n}{2} \right)!. 
\end{align*}
Hence, by using Lemma \ref{hypgeom} again, we obtain 
\begin{align*}
     s^w_\lambda
=&\frac{1}{4} (2\pi)^{-(n+2)}  \left( \frac{n}{2} \right)!^2 
    \sum_u   u! \binom{\frac{n}{2}+1}{u}  (4\pi )^{-u}
              \omega\left(4\pi ; -u, 1+\frac{n}{2} \right)  \\   \displaybreak[0]
=&\frac{1}{4} (2\pi)^{-(n+2)}  \left( \frac{n}{2} \right)!^2 \omega(4\pi ; 0, 0) \\  
=&\frac{1}{4} (2\pi)^{-(n+2)}  \left( \frac{n}{2} \right)!^2.
\end{align*}
In particular, for each $w\in {\mathbf Z}$, we find that 
\begin{align*}
s^w_\lambda
=s^w_\mu
= \frac{1}{4} (2\pi)^{-(n+2)}  \left(\frac{n}{2}\right)!^2.
\end{align*}
Hence we obtain
\begin{align*}
  {\mathcal B}^0_\alpha  
  =& -2^{2} e^{-4\pi} \sqrt{-1}^{n-\alpha+1} 
     \times \frac{1}{2} (2\pi)^{-(n+2)}  \left( \frac{n}{2} \right)!^2 
      \sum_w   (-1)^{m^+(w)}   \binom{\alpha}{l^+(w)} \binom{n-\alpha}{m^+(w)}
                   (s^w_\lambda +s^w_\mu  )  \\
  =& -2^{-1} e^{-4\pi }  \sqrt{-1}^{n-\alpha+1}
     \Gamma_{\mathbf C} \left(\frac{n+2}{2} \right)^2
      \sum_{c\in {\mathbf Z}} (-1)^c \binom{\alpha}{c} \binom{n-\alpha}{\frac{n}{2}-c},     
\end{align*}
where $\Gamma_{\mathbf C}(s)$ is defined to be $2(2\pi)^{-s}\Gamma(s)$. 
This proves the first identity in Lemma \ref{lem:arbessel}.

We proceed to compute ${\mathcal B}^1_\infty(\varphi,\phi)$. Set
\begin{align*}
    w_0=\begin{pmatrix}  0 & 1 \\ -1 & 0 \end{pmatrix}.
\end{align*}
Since $H_{x}({\mathbf R})$ is stable under the conjugation by $w_0$,
we can take $\left\{  \begin{pmatrix} 1 & 0 \\ z & 1 \end{pmatrix} u |  \ z\in {\mathbf C}, u\in {\rm SU}_2({\mathbf R}) \right\}$ to be a set of representatives of $H_{x}({\mathbf R})\backslash H^0_1({\mathbf R})$.
Hence, for $t\in {\mathbf C}^\times$, we have 
\begin{align}\label{eq:Wc1}
\begin{aligned}
{\mathcal W}^1_\alpha(t)
:=&\int_{H_{x}({\mathbf R}) \backslash H^0_1({\mathbf R})}
     \langle \varphi_\infty(h^{-1} \cdot x \varsigma_\infty ),  W_{ {\mathbf f}^\dag, \infty}\left(\begin{pmatrix} t & 0  \\ 0 & 1 \end{pmatrix}w_0h\right) \rangle_{\mathcal W}  {\rm d}h \\
=&\int_{{\mathbf C}} \langle \varphi_\infty\left( \begin{pmatrix} 1 & 0 \\ z & 1 \end{pmatrix}^{-1} \cdot x  \varsigma_\infty \right), 
                               W_{ {\mathbf f}^\dag, \infty} \left( \begin{pmatrix} t & 0  \\  0 & 1 \end{pmatrix} \begin{pmatrix} 1 & -z \\ 0 & 1 \end{pmatrix} w_0 \right) \rangle_{\mathcal W} {\rm d}z   \\
=&\int_{{\mathbf C}} \langle \tau_{2n+2}(w_0) \varphi_\infty\left( \begin{pmatrix} 1 & 0 \\ - z & 1 \end{pmatrix}^{-1} \cdot x  \varsigma_\infty \right), 
                               W_{ {\mathbf f}^\dag, \infty}\left( \begin{pmatrix} t & 0  \\  0 & 1 \end{pmatrix} \begin{pmatrix} 1 & z \\ 0 & 1 \end{pmatrix}  \right) \rangle_{\mathcal W} {\rm d}z   \\
=&\frac{1}{2}e^{-4\pi } \cdot \int_{{\mathbf C}} e^{-4\pi z\bar{z}}\psi_{F, \infty} (tz)  
           \cdot \langle Q^\alpha(\bar{z},z;Y,-X), W_{ {\mathbf f}^\dag, \infty} \left( \begin{pmatrix} t & 0 \\ 0 & 1 \end{pmatrix} \right) \rangle_{\mathcal W}   {\rm d}z.
\end{aligned}
\end{align}
Note that
\begin{align*}
   Q^\alpha(\bar{z},z;Y,-X) =(-1)^{\alpha+1}  Q^\alpha(z,\bar{z}; X,Y).
\end{align*}
Thus  (\ref{eq:Wc1}) yields that 
\begin{align*}
{\mathcal W}^1_\alpha(t)
=&(-1)^{\alpha+1}  {\mathcal W}^0_\alpha(t).
\end{align*}
This shows the claim in the case where $\alpha$ is even.
If $\alpha$ is odd, 
(\ref{oddsum}) shows that 
\begin{align*}
{\mathcal B}^1_\alpha = (-1)^{\alpha+1}{\mathcal B}^0_\alpha = 0.
\end{align*}
This finishes the proof of Lemma \ref{lem:arbessel}.
\end{proof}

\section*{Appendix} 

In Section \ref{s:ArchInt}, we write a conjecture on certain integral $I_n$ (Conjecture \ref{conj:arint}),  
    which gives us an explicit formula of a local archimedean integral (Lemma \ref{lem:arintInn})  
    and an explicit inner product formula (Theorem \ref{HSTInnprd}).  
Although Conjecture \ref{conj:arint} can be checked by a computer for $n=0,2,4,6,8$, 
we need certain special formulas on Bessel functions and hypergeometric functions in the computation. 
Hence we include this appendix which gives a proof of Conjecture \ref{conj:arint} in the case that $n=0$ 
to explain how we check the conjecture. 

We recall some notation in Section \ref{s:ArchInt}.  
Let 
\begin{align*}
   V_\infty = \left\{   \begin{pmatrix} p &  \sqrt{-1}q  \\   \sqrt{-1} r  &  \overline{p}   \end{pmatrix} |\ p \in {\mathbf C}, q, r \in {\mathbf R} \right\};   \quad  
  {\mathbf X}_\infty  
     =  V^{\oplus 2}_\infty.
\end{align*}
Recall that $\varphi_\infty \in {\mathcal S}({\mathbf X}_\infty)\otimes_{\mathbf C} {\mathcal W}_{2n+2} ({\mathbf C}) \otimes_{\mathbf C}{\mathcal L}_\lambda({\mathbf C})$
   is the Bruhat-Schwartz function which is introduced in (\ref{def:phiinf}).   
Write 
\begin{align*}
   \varphi_\infty(x) 
     = \sum^n_{\alpha=0} \sum^{n+1}_{j=-n-1} 
         P^\alpha_j (x) e^{-\pi {\rm Tr}(x_1{}^{\rm t}\overline{x_1} + x_2{}^{\rm t}\overline{x_2}   )}  X^{n+1+j} Y^{n+1-j} 
             \otimes \binom{n}{\alpha} X^{\alpha} Y^{n-\alpha}, 
\end{align*}
where $P^\alpha_j (x)$ is a polynomial function in $x=(x_1, x_2)\in {\mathbf X}_\infty$. 
We also recall that $W_{\sigma, \infty}$ associated with $\sigma_\infty$  is the Whittaker function which is characterized as follows:
\begin{align*}
   W_{\sigma, \infty} \left(  \begin{pmatrix} t & 0 \\ 0 & 1   \end{pmatrix}  \right) 
   = 2^4\times\sum^{n+1}_{j=-n-1} 
          \sqrt{-1}^j  t^{n+2} K_j(4\pi t)
                (X^{n+1+j} Y^{n+1-j})^\vee, 
     \quad (0<t \in {\mathbf R}).
\end{align*}
Denote by $\langle\cdot, \cdot\rangle_{\mathcal W}=\langle\cdot,\cdot\rangle_{2n+2}$ (resp. $\langle\cdot, \cdot\rangle_{\mathcal L}=\langle\cdot,\cdot\rangle_{n}$) 
   the pairing on ${\mathcal W}_{2n+2} ({\mathbf C})^{\otimes 2}$ (resp. ${\mathcal L}_\lambda({\mathbf C})^{\otimes 2}$) which is introduced in Section \ref{algrep}.
Then the definition of $\langle\cdot, \cdot\rangle_{\mathcal W}$ shows that 
\begin{align*}
   & \left\langle \varphi_\infty(x),  W_{\sigma, \infty} \left(  \begin{pmatrix} t & 0 \\ 0 & 1   \end{pmatrix}  \right)  \right\rangle_{\mathcal W}  \\
 =& \sum^n_{\alpha=0} \sum^{n+1}_{j=-n-1} 
      P^\alpha_j(x)  e^{-\pi {\rm Tr}(x_1{}^{\rm t}\overline{x_1} + x_2{}^{\rm t}\overline{x_2}   )}
       \cdot 2^4  \sqrt{-1}^j  t^{n+2} K_j(4\pi t)
        \binom{n}{\alpha} X^\alpha Y^{n-\alpha}.  
\end{align*}
Then integral $I_n$  in Conjecture \ref{conj:arint} is written as 
\begin{align*}
I_n=&  \int^1_0 \frac{(a-a^{-1})^2 {\rm d}a}{a}      
     \int_{{\mathbf X}_\infty}  {\rm d}x
          \int^\infty_0  \frac{{\rm d} t}{t}    \\
    &\times 2^8 \sum^n_{\alpha=0} \sum^{n+1}_{j, j^\prime=-n-1}  
        P^\alpha_j\left( \begin{pmatrix} p_1 & \sqrt{-1} a^{-1} q_1 \\ \sqrt{-1} a r_1 & \overline{p_1} \end{pmatrix}, 
                                    \begin{pmatrix} p_2 & \sqrt{-1} a^{-1} q_2 \\ \sqrt{-1} a r_2 & \overline{p_2} \end{pmatrix}  \right)  \\
    & \times \overline{  P^{n-\alpha}_{j^\prime}\left( \begin{pmatrix} p_1 & \sqrt{-1} q_1 \\ \sqrt{-1} r_1 & \overline{p_1} \end{pmatrix}, 
                                                       \begin{pmatrix} p_2 & \sqrt{-1} q_2 \\ \sqrt{-1} r_2 & \overline{p_2} \end{pmatrix}  \right) } \times (-1)^\alpha \binom{n}{\alpha} \\
     &\times   \sqrt{-1}^{j-j^\prime}  a^{n+2} t^{2n+4} K_j(4\pi a t) K_{j^\prime}(4\pi t)
        e^{- \pi (  4p_1\overline{p_1}  +  (1+a^{-2}) q^2_1 + (1+a^2) r^2_1  +  4 p_2\overline{p_2} +  (1+a^{-2}) q^2_2 + (1+a^2) r^2_2   )  }, 
\end{align*}
where
\begin{align*}
     x= \left( \begin{pmatrix} p_1 & \sqrt{-1} q_1 \\ \sqrt{-1} r_1 & \overline{p_1} \end{pmatrix}, 
                                    \begin{pmatrix} p_2 & \sqrt{-1} q_2 \\ \sqrt{-1} r_2 & \overline{p_2} \end{pmatrix} \right) \in {\mathbf X}_\infty.
\end{align*}

In Appendix,  we prove the following proposition: 

\begin{prop}
 {\itshape 
  We have $\displaystyle I_0 = 2^{-8} \cdot \pi^{-6}$. 
  Hence Conjecture \ref{conj:arint} is true for $n=0$.
  }
\end{prop}
 \begin{proof}
 Write $x=(x_1, x_2)= \left( \begin{pmatrix} p_1 & \sqrt{-1} q_1 \\ \sqrt{-1} r_1 & \overline{p_1} \end{pmatrix}, 
                                    \begin{pmatrix} p_2 & \sqrt{-1} q_2 \\ \sqrt{-1} r_2 & \overline{p_2} \end{pmatrix} \right) \in {\mathbf X}_\infty$.
Recall 
\begin{align*}
  {\mathbf p}\left(  \begin{pmatrix} a & b \\  b & c  \end{pmatrix} \right) 
     =   a X^2 + 2b XY + cY^2  
\end{align*}
and that the Bruhat-Schwartz function $\varphi_\infty$ which we fixed in (\ref{def:phiinf}) is written as 
\begin{align*}
    & \varphi_\infty (x)  \\
   =& {\mathbf p}\left(  \frac{1}{2} \left(  \frac{1}{2}(x_1+{}^{\rm t} x_1) \cdot w_0 \cdot \frac{1}{2}(x_2+{}^{\rm t} x_2)    
                                                                                    + {}^{\rm t}\left(  \frac{1}{2}(x_1+{}^{\rm t} x_1) \cdot w_0 \cdot \frac{1}{2}(x_2+{}^{\rm t} x_2 ) 
                                                                                                           \right)      
                                                                              \right)   
                                                    \right)  
        e^{-\pi {\rm Tr} (x_1 {}^{\rm t}\overline{x_1} + x_2 {}^{\rm t}\overline{x_2} ) }   \\
   =& \frac{1}{2} 
           \left\{    \sqrt{-1}(p_1(q_2+r_2) - (q_1+r_1)p_2) X^2  
                       + 2 (p_1\overline{p_2} - \overline{p_1}p_2  ) XY    
                       - \sqrt{-1} (\overline{p_1}(q_2+r_2) - (q_1+r_1) \overline{p_2} )Y^2   \right\}  \\
     &\quad \times     e^{- \pi ( 2 p_1\overline{p_1}  +  q^2_1 +r^2_1  + 2 p_2\overline{p_2} + q^2_2 +r^2_2   )  }.
\end{align*}
Hence we find that
\begin{align*}
     &  \overline{ \left\langle \varphi_\infty(x), W_{\sigma,\infty}\left(  \begin{pmatrix}  t & 0 \\ 0 & 1 \end{pmatrix}  \right)   \right\rangle_{\mathcal W}  }   \\
  = &  2^{3}  e^{- \pi ( 2 p_1\overline{p_1}  +  q^2_1 +r^2_1  + 2 p_2\overline{p_2} + q^2_2 +r^2_2   )  }   
      \times \left\{  
            -\sqrt{-1}( \overline{p_1}(q_2+r_2)-(q_1+r_1)\overline{p_2})  \times \sqrt{-1} t^2 K_{-1}(4\pi t)    \right.  \\ 
     & \quad \quad \quad \quad \quad \quad \quad \quad \quad \quad \quad \quad \quad \quad \quad \quad     + 2 (\overline{p_1}p_2 - p_1\overline{p_2}   )  \times t^2 K_{0}(4\pi t)   \\  
     & \quad \quad \quad \quad   \quad \quad \quad \quad \quad \quad \quad \quad  \quad  \quad \quad \quad  \left.   + \sqrt{-1} ( p_1 (q_2+r_2) - (q_1+r_1) p_2 )Y^2  \times (-\sqrt{-1}) t^2 K_{1}(4\pi t)       \right\}  \\
  = &   2^{3} e^{- \pi ( 2 p_1\overline{p_1}  +  q^2_1 +r^2_1  + 2 p_2\overline{p_2} + q^2_2 +r^2_2   )  }   \\
     & \quad   \times \left\{  
            2(\overline{p_1}p_2 - p_1\overline{p_2}   )   t^2 K_{0}(4\pi t) 
        +    ( (p_1+  \overline{p_1})(q_2+r_2)  -  (q_1+r_1)(p_2+  \overline{p_2} ) )   t^2 K_{1}(4\pi t)   \right\}, \\
     & \left\langle \varphi_\infty \left(  \begin{pmatrix}  a^{-\frac{1}{2}}  & 0 \\ 0 & a^{\frac{1}{2}} \end{pmatrix} x \right),  
                                            W_{\sigma,\infty}\left(  \begin{pmatrix}  t & 0 \\ 0 & 1 \end{pmatrix}   \begin{pmatrix}  a^{\frac{1}{2}}  & 0 \\ 0 & a^{-\frac{1}{2}} \end{pmatrix}    \right)   \right\rangle_{\mathcal W}  \\   
  = &  2^{3} e^{- \pi ( 2 p_1\overline{p_1}  +  a^{-2} q^2_1 + a^2 r^2_1  + 2 p_2\overline{p_2} + a^{-2} q^2_2 + a^2 r^2_2   )  }   
       \times \left\{  
            2 (p_1\overline{p_2} - \overline{p_1}p_2   )  a^2 t^2 K_{0}(4\pi at)   \right.  \\
      & \quad \quad \quad \quad \quad \quad \quad \quad \quad \quad \quad \quad \quad 
              \left.   +    ( (p_1+  \overline{p_1})(a^{-1 }q_2 + a r_2)  -  (a^{-1} q_1 + a r_1)(p_2+  \overline{p_2} ) )  a^2 t^2 K_{1}(4\pi at)   \right\}.
\end{align*}
These two identities show that 
\begin{align*}
  & \left\langle \varphi_\infty \left(  \begin{pmatrix}  a^{-\frac{1}{2}}  & 0 \\ 0 & a^{\frac{1}{2}} \end{pmatrix} x \right),  
                                            W_{\sigma,\infty}\left(  \begin{pmatrix}  t & 0 \\ 0 & 1 \end{pmatrix}   \begin{pmatrix}  a^{\frac{1}{2}}  & 0 \\ 0 & a^{-\frac{1}{2}} \end{pmatrix}    \right)   \right\rangle_{\mathcal W} 
       \overline{ \left\langle \varphi_\infty(x), W_{\sigma,\infty}\left(  \begin{pmatrix}  t & 0 \\ 0 & 1 \end{pmatrix}  \right)   \right\rangle_{\mathcal W}  }   \\
 =&  2^{6} e^{- \pi ( 2 p_1\overline{p_1}  +  a^{-2}q^2_1 +a^2r^2_1  + 2 p_2\overline{p_2} + a^{-2}q^2_2 +a^2r^2_2   )  }   \\
   & \quad \times \left\{  t^4 K_0(4\pi t) K_0(4\pi at)  4a^2 (\overline{p_1}p_2 - p_1\overline{p_2}   ) (p_1\overline{p_2} - \overline{p_1}p_2   )    \right.   \\
    & \quad \quad +    t^4 K_0(4\pi t) K_1(4\pi at)  2(\overline{p_1}p_2 - p_1\overline{p_2}   )  
                                                                              ( (p_1+  \overline{p_1})(a q_2+ a^3 r_2)  -  (a q_1 + a^3 r_1)(p_2+  \overline{p_2} ) )   \\
    & \quad \quad +    t^4 K_1(4\pi t) K_0 (4\pi at)  2a^2 (  (p_1+  \overline{p_1})(q_2+r_2)  -  (q_1+r_1)(p_2+  \overline{p_2} )  ) 
                                                                               (p_1\overline{p_2} - \overline{p_1}p_2   )  \\         
    & \quad \quad   +    t^4 K_1(4\pi t) K_1(4\pi at)   (  (p_1+  \overline{p_1})(q_2+r_2)  -  (q_1+r_1)(p_2+  \overline{p_2} )  ) \\
    & \quad \quad \quad \quad \quad \quad \quad \quad \quad \quad \quad \quad \quad \quad \quad \quad \quad \quad  
         \left.  \times   ( (p_1+  \overline{p_1})(a q_2+ a^3 r_2)  -  (a q_1 + a^3 r_1)(p_2+  \overline{p_2} ) )    \right\}.
\end{align*}
Considering the integral on ${\mathbf X}_\infty$,  
 only the following terms contribute the integral:
\begin{align*} 
   &   e^{- \pi (  4p_1\overline{p_1}  +  (1+a^{-2}) q^2_1 + (1+a^2) r^2_1  +  4 p_2\overline{p_2} +  (1+a^{-2}) q^2_2 + (1+a^2) r^2_2   )  }  \\
   & \quad \times \left\{  t^4 K_0(4\pi t) K_0(4\pi at) 4 a^2 ( p_1\overline{p_2} - \overline{p_1}p_2 )^2         \right.   \\
    & \quad \quad \left.  +    t^4 K_1(4\pi t) K_1(4\pi at)   (  (p_1+\overline{p_1})^2 (a q^2_2 + a^3r^2_2)
                                                                                          +  (aq^2_1 + a^3 r^2_1)(p_2+\overline{p_2})^2    )    \right\}, 
\end{align*}
since 
\begin{align*}
  \int^\infty_{-\infty} u e^{-\alpha u^2}  {\rm d} u =0 
\end{align*}
for each $\alpha \in {\mathbf R}_{>0}$. 
In the straightforward way, we find that 
\begin{align*}
   \begin{aligned}
    & \int_{{\mathbf X}_\infty}  
           e^{- \pi (  4p_1\overline{p_1}  +  (1+a^{-2}) q^2_1 + (1+a^2) r^2_1  +  4 p_2\overline{p_2} +  (1+a^{-2}) q^2_2 + (1+a^2) r^2_2   )  }  \\
              & \quad \quad \quad \quad \quad \quad \quad \quad \quad \quad \quad \quad \quad \quad \quad \quad
                 \times 4a^2  ( p_1\overline{p_2} - \overline{p_1}p_2 )
                                    ( \overline{p_1}p_2 - p_1\overline{p_2}   )   {\rm d}x
   \end{aligned}  
=&  \frac{a^4}{2^5(1+a^2)^2\pi^2}, \\
  \begin{aligned}
      & \int_{{\mathbf X}_\infty}  e^{- \pi (  4p_1\overline{p_1}  +  (1+a^{-2}) q^2_1 + (1+a^2) r^2_1  +  4 p_2\overline{p_2} +  (1+a^{-2}) q^2_2 + (1+a^2) r^2_2   )  }  \\
              & \quad \quad \quad \quad \quad \quad \quad \quad \quad \quad \quad       
               \times   (  (p_1+\overline{p_1})^2 (a q^2_2 + a^3r^2_2)
                                                                                          +  (aq^2_1 + a^3 r^2_1)(p_2+\overline{p_2})^2    )   {\rm d}x  
     \end{aligned}          
=& \frac{a^5}{2^4(1+a^2)^3\pi^2}.
\end{align*}
Put $z=1-a^2$. Then \cite[page 101]{mos66} shows that
\begin{align*}
       \int^\infty_0  t^4 K_0(4\pi t) K_0(4\pi at) \frac{{\rm d}t}{t}   
   =& 2 \cdot (4\pi)^{-4} \frac{1}{\Gamma(4)}   
         \Gamma\left(\frac{4}{2} \right)^4  \times {}_2F_1\left( \frac{4}{2},  \frac{4}{2} ; 4; z \right)  \\ 
   =& 2^{-8}\cdot 3^{-1} \cdot \pi^{-4} \cdot \frac{\Gamma(4)}{\Gamma(2)\Gamma(2)}  \sum^\infty_{n=0}  \frac{\Gamma(n+2)\Gamma(n+2)}{\Gamma(n+4)}  \frac{  z^n}{n!}  \\       
   =& 2^{-7} \cdot \pi^{-4} \cdot  \sum^\infty_{n=0} \left(   \frac{2}{n+3}  - \frac{1}{n+2}  \right)  z^n \\       
   =& 2^{-7}\cdot  \pi^{-4} \cdot  \left(  \frac{2}{z^3}  \left( \log(1-z) - z -\frac{z^2}{2}  \right)  - \frac{1}{z^2}  (\log(1-z) - z )  \right),  \\       
       \int^\infty_0   t^4 K_1(4\pi t) K_1(4\pi at)      \frac{{\rm d}t}{t}   
  = & 2\cdot (4\pi)^{-5} (4\pi a) \frac{1}{\Gamma(4)}   
         \Gamma\left(\frac{6}{2}\right) \Gamma\left(\frac{4}{2}\right)^2 \Gamma\left(\frac{2}{2}\right) \times {}_2F_1\left( \frac{6}{2},  \frac{4}{2} ; 4; 1- a^2 \right) \\
  = & 2^{-7} \cdot 3^{-1} \cdot \pi^{-4} a \times {}_2F_1( 3,  2 ; 4; 1- a^2 )  \\
   =& 2^{-7}\cdot 3^{-1} \cdot \pi^{-4}  a \cdot \frac{\Gamma(4)}{\Gamma(3)\Gamma(2)}  \sum^\infty_{n=0}  \frac{\Gamma(n+2)\Gamma(n+3)}{\Gamma(n+4)}  \frac{  z^n}{n!}  \\       
   =& 2^{-7} \cdot \pi^{-4} a\cdot  \sum^\infty_{n=0} \left(  1  -  \frac{2}{n+3}  \right)  z^n \\       
   =& 2^{-7} \cdot \pi^{-4}  a\cdot  \left(  \frac{1}{1-z}  - \frac{2}{z^3}  \left( \log(1-z) - z - \frac{z^2}{2} \right)  \right).
\end{align*}
Since 
\begin{align*}
  &     2^{-7}\cdot  \pi^{-4}  \cdot \left(  \frac{2}{z^3}  \left(\log(1-z) - z -\frac{z^2}{2}     \right)  - \frac{1}{z^2}  (\log(1-z) - z )  \right)
         \cdot \frac{a^4}{2^5(1+a^2)^2\pi^2}    \\
      &\quad +  
         2^{-7} \cdot \pi^{-4}  a \cdot  \left(  \frac{1}{1-z}  - \frac{2}{z^3}  \left(\log(1-z) - z - \frac{z^2}{2} \right)  \right)
         \frac{a^5}{2^4(1+a^2)^3\pi^2} \\
 =& 2^{-12} \cdot \pi^{-6}   \left\{  \left(  \frac{4 \log (a)}{(1-a^2)^3}   - 2 (1-a^2)^{-2} -  (1-a^2)^{-1}    
                                                     -  \frac{ 2 \log(a)}{(1-a^2)^2}   + (1-a^2)^{-1}  \right) \cdot  \left( - \frac{a^4}{(1+a^2)^2} \right)      \right.  \\
   &     \quad \quad \quad \quad  \quad \quad     \left.  + 2 a \left(  a^{-2}   - \frac{4 \log (a)}{(1-a^2)^3} + 2(1-a^2)^{-2} + (1-a^2)^{-1}    \right) \cdot \frac{a^5}{(1+a^2)^3}   \right\}  \\
 =& 2^{-11} \cdot \pi^{-6}   \cdot  \frac{a^4 \log(a)}{(1-a^2)(1+a^2)^3},  
\end{align*}
 the integral $I_0$ is given by  
\begin{align*}
  2^{6}  \int^1_0  2^{-11} \cdot \pi^{-6}   \cdot  \frac{a^4 \log(a)}{(1-a^2)(1+a^2)^3}   \frac{(a-a^{-1})^2 {\rm d} a}{a} 
  =  - 2^{-5} \cdot \pi^{-6}       
             \int^1_0   \frac{a (1-a^2) \log(a)}{(1+a^2)^3}    {\rm d} a.
\end{align*}
Since 
\begin{align*}
   \frac{{\rm d} }{{\rm d} a} \left( \frac{ 1+a^2 + 2a^2  \log(a)}{4(1+a^2)^2}   \right)
   =  \frac{a (1-a^2) \log(a)}{(1+a^2)^3},  
\end{align*}
we find that 
\begin{align*}
\int^1_0   \frac{a (1-a^2) \log(a)}{(1+a^2)^3}    {\rm d} a = - 2^{-3}. 
\end{align*}
Therefore $I_0=2^{-8}\pi^{-6}$.
 \end{proof}

\section*{Acknowledgements} 

The author is sincerely grateful for Ming-Lun Hsieh.  
Without discussions with him, this work would not have been possible.  
The author also thanks to Tobias Berger for comments on the earlier version of the manuscript. 
This work is done while author's stay at 
Taida Institute for Mathemativcal Sciences,  
National Center for Theoretical Sciences 
and Academia Sinica. 
He thanks for their hospitality. 
The author was supported by JSPS Grant-in-Aid for Young Scientists (B) Grand Number JP17K14174.

\end{document}